\documentclass[12pt]{amsart}
\usepackage{amsmath,amssymb,amsthm,amscd,amsxtra,eucal,latexsym,mathrsfs}

\newcommand{\nc}{\newcommand}

\newcommand{\cA}{{\mathcal A}}
\newcommand{\cB}{{\mathcal B}}
\newcommand{\cC}{{\mathcal C}}
\newcommand{\cD}{{\mathcal D}}
\newcommand{\cH}{{\mathcal H}}
\newcommand{\cE}{{\mathcal E}}

\newcommand{\cI}{{\mathcal I}}

\newcommand{\cO}{{\mathcal O}}
\newcommand{\cL}{{\mathcal L}}
\newcommand{\cM}{{\mathcal M}}

\newcommand{\cF}{{\mathcal F}}
\newcommand{\cK}{{\mathcal K}}
\newcommand{\cP}{{\mathcal P}}

\newcommand{\cS}{{\mathcal S}}
\newcommand{\cT}{{\mathcal T}}
\newcommand{\cU}{{\mathcal U}}
\newcommand{\cV}{{\mathcal V}}

\newcommand{\cX}{{\mathcal X}}

\renewcommand{\AA}{{\mathbb A}}

\newcommand{\GG}{{\mathbb G}}

\newcommand{\ZZ}{{\mathbb Z}}
\newcommand{\QQ}{{\mathbb Q}}

\newcommand{\PP}{{\mathbb P}}
\newcommand{\OO}{{\mathbb O}}

\newcommand{\SSS}{{\mathbb S}}

\renewcommand{\gg}{\mathfrak{g}}  

\newcommand{\gp}{\mathfrak{p}}
\newcommand{\gq}{\mathfrak{q}}
\newcommand{\gu}{\mathfrak{u}}

\newcommand{\gt}{\mathfrak{t}}
\newcommand{\gr}{\mathfrak{r}}
\newcommand{\gs}{\mathfrak{s}}

\newcommand{\gU}{\mathfrak{U}}
\newcommand{\gB}{\mathfrak{B}}

\newcommand{\bF}{\mathbf{F}}
\newcommand{\bO}{{\mathbf O}}

\newcommand{\on}{\operatorname}

\newcommand{\Rep}{{\on{Rep}}}

\newcommand{\Qlb}{\mathbb{\bar Q}_\ell}
\newcommand{\Gm}{\mathbb{G}_m}

\newcommand{\A}{\mathbb{A}}

\newcommand{\toup}[1]{\stackrel{#1}{\to}}
\newcommand{\touplong}[1]{\stackrel{#1}{\longrightarrow}}
\newcommand{\hook}[1]{\stackrel{#1}{\hookrightarrow}}

\newcommand{\getsup}[1]{\stackrel{#1}{\gets}}
\newcommand{\getsuplong}[1]{\stackrel{#1}{\longleftarrow}}
\newcommand{\Sp}{\on{\mathbb{S}p}}
\newcommand{\Spin}{\on{\mathbb{S}pin}}

\newcommand{\IC}{\on{IC}}
\newcommand{\CT}{\on{CT}}
\newcommand{\Hom}{\on{Hom}}
\newcommand{\Mod}{\on{Mod}}
\newcommand{\Modt}{\wt\Mod}
\newcommand{\Ext}{\on{Ext}}

\newcommand{\Sym}{\on{Sym}}

\newcommand{\Aut}{\on{Aut}}

\newcommand{\RG}{\on{R\Gamma}}

\newcommand{\Pic}{\on{Pic}}

\newcommand{\Bun}{\on{Bun}}

\newcommand{\Bunb}{\on{\overline{Bun}} }
\newcommand{\Bunt}{\on{\widetilde\Bun}}

\newcommand{\Spec}{\on{Spec}}

\newcommand{\Gr}{\on{Gr}}

\newcommand{\PSL}{\on{PSL}}

\newcommand{\Fr}{{\on{Fr}}}

\newcommand{\Eis}{{\on{Eis}}}
\newcommand{\pr}{\on{pr}}
\newcommand{\id}{\on{id}}

\newcommand{\QED}{$\square$} 
\newcommand{\Fq}{\mathbb{F}_q}  

\newcommand{\iso}{{\widetilde\to}}

\newcommand{\comp}{\circ}
\newcommand{\Four}{\on{Four}}
\renewcommand{\H}{{\on{H}}}   
\newcommand{\IH}{{\on{IH}}}

\newcommand{\D}{\on{D}}       
\newcommand{\wt}{\widetilde}
\newcommand{\ov}[1]{\overline{#1}}
\newcommand{\select}[1]{{\it{#1}}}

\newcommand{\<}{\langle}
\renewcommand{\>}{\rangle}

\newcommand{\Conv}{\on{Conv}}
\newcommand{\Loc}{\on{Loc}}

\newcommand{\Sph}{\on{Sph}}
\newcommand{\Res}{\on{Res}}

\newcommand{\ttimes}{\tilde\times}

\newcommand{\dimrel}{\on{dim.rel}}

\newcommand{\Funct}{\on{Funct}}

\newcommand{\Cov}{\on{Cov}}

\newcommand{\SL}{\on{SL}}

\newcommand{\tboxtimes}{\,\tilde\boxtimes\,}

\newcommand{\ra}{\rightarrow}
\newcommand{\la}{\leftarrow}

\nc{\Perv}{\on{Perv}}

\newcommand{\cPic}{\on{{\cP}ic}}

\nc{\Gra}{\on{Gra}}
\nc{\PPerv}{\on{{\PP}erv}}

\nc{\oX}{\overset{\scriptscriptstyle\circ}{X}}
\nc{\oBun}{\overset{\scriptscriptstyle\circ}{\Bun}}
\nc{\owtBun}{\overset{\scriptscriptstyle\circ}{\Bunt}}
\nc{\gRes}{\on{gRes}}
\nc{\Sign}{\on{Sign}}
\nc{\Tr}{\on{Tr}}

\newtheorem{Lm}{Lemma}[section]
\newtheorem{Th}{Theorem}[section]
\newtheorem{Pp}{Proposition}[section]
\newtheorem{Cor}{Corollary}[section]
\newtheorem{Con}{Conjecture}[section]

\theoremstyle{remark}
\newtheorem{Rem}{Remark}[section]

\theoremstyle{definition}
\newtheorem{Def}{Definition}[section]

\newenvironment{Prf}{\par\noindent {\it Proof }}{\QED}
\newcommand{\Step}[1]{\par\noindent{\bf Step {#1}}.}

\begin{document}

\title{Geometric Eisenstein series: twisted setting}
\author{Sergey Lysenko}
\begin{abstract} 
Let $G$ be a simple simply-connected group over an algebraically closed field $k$, $X$ be a smooth connected projective curve over $k$. In this paper we develop the theory of geometric Eisenstein series on the moduli stack $\Bun_G$ of $G$-torsors on $X$ in the setting of the quantum geometric Langlands program (for \'etale $\Qlb$-sheaves) in analogy with \cite{BG}. We calculate the intersection cohomology sheaf on the version of Drinfeld compactification in our twisted setting. In the case $G=\SL_2$ we derive some results about the Fourier coefficients of our Eisenstein series. In the case of $G=SL_2$ and $X=\PP^1$ we also construct the corresponding theta-sheaves and prove their Hecke property.
\end{abstract}

\maketitle
\tableofcontents


\section{Introduction}

\subsubsection{} 
\label{Section_101}
In this paper we develop an analog of the theory of geometric Eisenstein series from \cite{BG} for the twisted geometric Langlands correspondence. Here `twisted' refers to the quantum Langlands correspondence (as outlined in \cite{Ga2, Ga3, Sch}) with the quantum parameter being a root of unity.  In the case of a split torus the corresponding geometric setting was proposed in \cite{L}. 

 The conjectural extension of the Langlands program for metaplectic groups was proposed by M. Weissman in (\cite{W1,W2}, see also \cite{GG, Gao}). In his approach the metaplectic group is a central extension of a reductive group by a finite cyclic group coming from Brylinski-Deligne theory \cite{BD}.
Our approach is a geometrization of this evolving program.

 For motivations, take $k=\Fq$. Let $X$ be a smooth projective curve over $k$, $G$ be a split reductive group. Let $\AA$ be the adeles ring of $F=k(X)$, $\cO\subset\AA$ be the integer adeles. Let $n\ge 1$, $n\mid q-1$. A Brylinski-Deligne extension of $G$ by $K_2$ gives rise to an extension $1\to \mu_n(k)\to \tilde G\to G(\AA)\to 1$ together with its splittings over $G(F)$ and $G(\cO)$. 
Pick an injective character $\bar\zeta: \mu_n(k)\to\Qlb^*$. The global nonramified Langlands program for $\tilde G$ aims to find the spectral decomposition of the space $\Funct_{\bar\zeta}(G(F)\backslash \tilde G/G(\cO))$ of $\Qlb$-valued functions that change by $\bar\zeta$ under the action of $\mu_n(k)$. 

 A fundamental tool for this program is the theory of Eisenstein series. Let $M\subset P\subset G$ be a Levi subgroup of a parabolic subgroup of $G$. By restriction the above yield the metaplectic extensions $\tilde M$ (resp., $\tilde P$) of $M(\AA)$ (resp., $P(\AA)$). One gets the diagram of projections
$$
M(F)\backslash\tilde M/M(\cO) \getsup{\gq} P(F)\backslash \tilde P/P(\cO)\toup{\gp} G(F)\backslash \tilde G/G(\cO)\, .
$$
For a compactly supported function $f\in \Funct_{\bar\zeta}(M(F)\backslash\tilde M/M(\cO))$ the associated Eisenstein series (up to a normalization factor) is $\gp_!\gq^*f$. 

We study a geometric analog of this construction.
We work with  \'etale $\Qlb$-sheaves to keep a close relation with the more classical Langlands program for the corresponding metaplectic groups.

\subsubsection{} Let $k$ be an algebraically closed field, $G$ be a simple, simply-connected group over $k$. In this case there is a canonical Brylinski-Deligne extension of $G$ by $K_2$ (the others are obtained from it up to isomorphism as its multiples). 

 Let $X$ be a smooth projective connected curve over $k$. Let $\Bun_G$ be the stack of $G$-torsors on $X$. Pick $n\ge 1$. We introduce some $\mu_N$-gerb $\Bunt_G\to\Bun_G$ with $N=2\check{h}n$, it comes from the canonical Brylinski-Deligne extension of $G$ by $K_2$. Here $\check{h}$ is the dual Coxeter number for $G$. We pick an injective character $\zeta: \mu_N(k)\to\Qlb^*$ and consider the derived category $\D_{\zeta}(\Bunt_G)$ of \'etale $\Qlb$-sheaves on $\Bunt_G$, on which $\mu_N(k)$ acts by $\zeta$. To these metaplectic data there corresponds a dual group $\check{G}_n$ defined in \cite{FL}. The category $\Rep(\check{G}_n)$ of finite-dimensional representations of $\check{G}_n$ acts on $\D_{\zeta}(\Bunt_G)$ by Hecke functors. The twisted geometric Langlands problem in this case is to construct Hecke eigen-sheaves in $\D_{\zeta}(\Bunt_G)$ (or even a spectral decomposition thereof). 
 
  Let $P\subset G$ be a parabolic subgroup, $M$ its Levi factor. We similarly get a $\mu_N$-gerb $\Bunt_M\to\Bun_M$ and the derived category $\D_{\zeta}(\Bunt_M)$. As in \cite{FL}, one has the corresponding Levi subgroup $\check{M}_n\subset \check{G}_n$, and $\Rep(\check{M}_n)$ acts on $\D_{\zeta}(\Bunt_M)$ by Hecke functors. 
  
  The Eisenstein series from Section~\ref{Section_101} admits an immediate geometrization 
$$
\Eis': \D_{\zeta}(\Bunt_M)\to\D_{\zeta}(\Bunt_G)
$$ 
However, $\Eis'$ does not commute with the Verdier duality and may be improved replacing $\Bunt_P$ by the relative Drinfeld's compactification $\Bunt_{\tilde P}$ along the fibres of the projection $\Bun_P\to\Bun_G$ as in \cite{BG}.
  
  We define the corresponding compactified Eisenstein series functor $\Eis: \D_{\zeta}(\Bunt_M)\to \D_{\zeta}(\Bunt_G)$ and study its properties. It is defined using a twisted version $\IC_{\zeta}$ of the $\IC$-sheaf of $\Bunt_P$. One of our main results is the description of $\IC_{\zeta}$ generalizing \cite{ICDC}. As in \cite{BG}, we show that $\Eis$ commutes with the Hecke functors with respect to the embedding $\check{M}_n\subset \check{G}_n$. 
  
  We formulate a conjectural functional equation of $\Eis$. We also show as in \cite{BG} that the formation of Eisenstein series is transitive for the diagram $T\subset M\subset G$, where $T$ is a maximal torus of $M$. 
    
  In the case of $G=\SL_2$ we get some partial description of the Fourier coefficients of $\Eis$, the answer is expressed in terms of a sheaf that appeared in the book \cite{BFS} on factorizable sheaves (and also in \cite{Ga2, L5}). The relation so obtained between these Fourier coefficients and quantum groups seems a promising phenomenon that has to be better understood. 
   
    As an application, we get an important formula for the first Whittaker coefficient of our Eisenstein series for metaplectic extensions of $\SL_2$ (cf. Corollary~\ref{Cor_6.2}). It turns out to be an $\ell$-adic analog of the space of conformal blocks in Wess-Zumino-Witten model studied in \cite{BFS}. It also could be seen as a generalization of the notion of central value of an abelian L-function (cf. Remark~\ref{Rem_Witten_conf_blocks}). 
    
  Among other results, we construct new automorphic sheaves on $\Bunt_G$ in the case of $G=\SL_2$ and $X=\PP^1$ corresponding to the trivial $\check{G}_n$-local system and a principal $\SL_2\to \check{G}_n$ of Arthur. We call them theta-sheaves as they generalize the theta-sheaves studied in \cite{L4}.

\section{Main results}
\label{Section_Main_results}

\subsubsection{Notations} 
\label{Section_Notations}
Work over an algebraically closed field $k$. Let $G$ be a simple algebraic group over $k$. Assume it simply-connected (hopefully, the non simply-connected case could also be done using \cite{S}). Let $T\subset B\subset G$ be a maximal torus and a Borel subgroup. Let $\gg$ be the Lie algebra of $G$.  Write $\Lambda$ for the coweights lattice of $T$, $\check{\Lambda}$ for the weight lattice of $T$. Let $\Lambda^+$ be the set of dominant coweights, $\check{\Lambda}^+$ the dominant weights. Write $\check{h}$ for the dual Coxeter number of $G$. Write $W$ for the Weyl group of $(G,T)$, let $w_0$ be the longest element in $W$. Let $\cI$ denote the set of vertices of the Dynkin diagram of $G$. For $i\in\cI$ write $\alpha_i$ (resp., $\check{\alpha}_i$) for the simple coroot (resp., simple root) of $G$ corresponding to $i$. 

 We ignore the Tate twists everywhere (they are easy to recover if necessary). 

 Let $X$ be a smooth projective connected curve. Let $\Bun_G$ be the stack of $G$-torsors on $X$. Let $F=k(X)$. For $x\in X$ we let $D_x$ denote the formal neighbourhood of $x$ in $X$, $D_x^*$ the punctured formal neighbourhood of $x\in X$. A trivial $G$-torsor on a base is denoted $\cF^0_G$. 
 
  Let $\iota: \Lambda\otimes\Lambda\to\ZZ$ be the unique symmetric bilinear $W$-invariant form such that $\iota(\alpha,\alpha)=2$ for a short coroot $\alpha$. The induced map $\iota: \Lambda\to \check{\Lambda}$ is also denoted by $\iota$. If $\alpha$ is a simple coroot then $\iota(\alpha)=\frac{\iota(\alpha,\alpha)}{2}\check{\alpha}$. Our convention is that a super line is a $\ZZ/2\ZZ$-graded line. 
  
  Recall the groupoid $\cE^s(T)$ defined in (\cite{L}, Section~3.2.1). Its objects are pairs: a symmetric bilinear form $\kappa: \Lambda\otimes\Lambda\to\ZZ$ and a central super extension $1\to k^*\to \tilde\Lambda^s\to\Lambda\to 1$ such that its commutator is $(\gamma_1, \gamma_2)_c=(-1)^{\kappa(\gamma_1,\gamma_2)}$. This means that for every $\gamma\in\Lambda$ we are given a super line $\epsilon^{\gamma}$, and for $\gamma_1,\gamma_2\in\Lambda$ a $\ZZ/2\ZZ$-graded isomorphism 
\begin{equation}
\label{iso_c_gamma_i_section_notations}
c^{\gamma_1,\gamma_2}: \epsilon^{\gamma_1}\otimes \epsilon^{\gamma_2}\,\iso\, \epsilon^{\gamma_1+\gamma_2}
\end{equation}
such that $c$ is associative, and one has $c^{\gamma_1,\gamma_2}=(-1)^{\kappa(\gamma_1,\gamma_2)}c^{\gamma_2,\gamma_1}\sigma$. Here $\sigma: \epsilon^{\gamma_1}\otimes \epsilon^{\gamma_2}\,\iso\, \epsilon^{\gamma_2}\otimes \epsilon^{\gamma_1}$ is the super commutativity constraint. Then $\cE^s(T)$ is a Picard groupoid with respect to the tensor product of central extensions. 
  
We have a canonical object $(\iota, \tilde\Lambda^{can})\in \cE^s(T)$ corresponding to a canonical extension of $G$ by $K_2$ in the sense of (\cite{BD}, Theorem~4.7). It is equipped with a $W$-equivariant structure. 
We pick once and for all a square root $\cE_X$ of $\Omega$. 

 Recall the Picard groupoid $\cP^{\theta}(X,\Lambda)$ of $\theta$-data from (\cite{L}, Section~4.2.1). Its objects are triples $\theta=(\kappa,\lambda, c)$, where $\kappa:\Lambda\otimes\Lambda\to\ZZ$ is a symmetric bilinear form, $\lambda$ is a rule that assigns to each $\gamma\in\Lambda$ a super line bundle $\lambda^{\gamma}$ on $X$, and $c$ is a rule that assigns to each pair $\gamma_1,\gamma_2\in\Lambda$ an isomorphism $c^{\gamma_1,\gamma_2}: \lambda^{\gamma_1}\otimes \lambda^{\gamma_2}\,\iso\, \lambda^{\gamma_1+\gamma_2}\otimes\Omega^{\kappa(\gamma_1,\gamma_2)}$ on $X$. They are subjects to the conditions from \select{loc.cit}. In particular, the parity of $\lambda^{\gamma}$ is $\kappa(\gamma,\gamma)\!\mod 2$. 

Denote by $\theta^{can}\in \cP^{\theta}(X,\Lambda)$ the image of $(\iota, \tilde\Lambda^{can})$ under the functor $\cE^s(T)\to \cP^{\theta}(X,\Lambda)$ of (\cite{L}, Lemma~4.1). That is, $\theta^{can}=(\iota, \lambda, {'c})$, where $\lambda^{\gamma}=\cE_X^{\otimes -\iota(\gamma,\gamma)}\otimes \epsilon^{\gamma}$, and 
$$
'c^{\gamma_1,\gamma_2}: \lambda^{\gamma_1}\otimes \lambda^{\gamma_2}\,\iso\, \lambda^{\gamma_1+\gamma_2}\otimes\Omega^{\iota(\gamma_1,\gamma_2)}
$$
is the evident product obtained from (\ref{iso_c_gamma_i_section_notations}). 

 For an algebraic stack $S$ locally of finite type write $\D(S)$ for the category introduced in (\cite{LO}, Remark~3.21) and denoted $\D_c(S,\Qlb)$ in \select{loc.cit.} It should be thought of as the unbounded derived category of constructible $\Qlb$-sheaves on $S$. 
  
  If $V\to S$ and $V^*\to S$ are dual rank $r$ vector bundles on a base stack $S$, we normalize the Fourier transform $\Four_{\psi}: \D^b(V)\to \D^b(V^*)$ by $\Four_{\psi}(K)=(p_{V*})_!(\xi^*\cL_{\psi}\otimes p_V^*K)[r]$, where $p_V,p_{V^*}$ are the projections, and $\xi: V\times V^*\to\A^1$ is the pairing. 
  
 
 If $S$ is a stack, $L$ is a super line bundle on $S$ purely of parity zero, we will use the stack of $n$-th roots of $L$. Its $T$-point is a map $T\to S$, a super line bundle $\cU$ on $T$ purely of parity zero, and a $\ZZ/2\ZZ$-graded isomorphism $\cU^n\,\iso\, L\mid_{T}$.   
 
\subsubsection{}  
\label{Section_202}
If $M\subset G$ is a Levi subgroup, denote by $\Bun_M$ the stack of $M$-torsors on $X$. For $\mu\in \pi_1(M)$, we denote by $\Bun_M^{\mu}$ the connected component of $\Bun_M$ classifying $M$-torsors on degree $-\mu$. This notation agrees with \cite{BG}, but does not agree with \cite{L}. Write $\cPic(\Bun_T)$ for the Picard groupoid of super line bundles on $\Bun_T$. If $\theta=(\kappa, \lambda, c)$ is an object of $\cP^{\theta}(X,\Lambda)$ we also denote by $\lambda$ the super line bundle on $\Bun_T$ obtained from $\theta$ via the functor $\cP^{\theta}(X,\Lambda)\to \cPic(\Bun_T)$ defined in (\cite{L}, Section~4.2.1, formula (18)). 

  The group $T$ acts on $\Bun_T$ by 2-automorphisms, so if $\cF\in\Bun_T$ then $T$ acts naturally on the fibre at $\cF$ of each line bundle on $\Bun_T$. According to our convention, for $\cF\in\Bun_T^{\mu}$, $\mu\in\Lambda$ the group $T$ acts on $\lambda_{\cF}$ by $-\kappa(\mu)$. 
 
\subsubsection{} 
\label{Section_0.2}
Let $\cL$ be the line bundle on $\Bun_G$ with fibre $\det\RG(X, \gg_{\cF})^{-1}\otimes \det\RG(X, \gg\otimes\cO)$ at $\cF\in \Bun_G$. This notation agrees with that of \cite{FL}. Pick $n$ invertible in $k$. Pick a line bundle $\cL_c$ on $\Bun_G$ equipped with $\cL_c^{2\check{h}}\,\iso\, \cL$, here $\cL_c$ is a generator of $\Pic(\Bun_G)\,\iso\, \ZZ$. 

 Let $\Bunt_{G, \cL_c}$ be the stack of $n$-th roots of $\cL_c$. Let $\bar\zeta: \mu_n(k)\to\Qlb^*$ be an injective character. We are interested in the derived category $\D_{\bar\zeta}(\Bunt_{G, \cL_c})$ of $\Qlb$-sheaves on $\Bunt_{G, \cL_c}$, on which $\mu_n(k)$ acts by $\bar\zeta$. 
 
 Assume that $N=2\check{h}n$ is invertible in $k$. Write $\Bunt_G$ for the gerb of $N$-th roots of $\cL$ over $\Bun_G$. Pick an injective character $\zeta: \mu_{N}(k)\to\Qlb^*$ such that $\zeta\mid_{\mu_n(k)}=\bar\zeta$. Denote by $\D_{\zeta}(\Bunt_G)$ the derived category of $\Qlb$-sheaves on $\Bunt_G$ on which $\mu_{N}(k)$ acts by $\zeta$. We have a natural map
$\alpha: \Bunt_{G, \cL_c}\to\Bunt_G$, and $\alpha^*: \D_{\zeta}(\Bunt_G)\to\D_{\bar\zeta}(\Bunt_{G, \cL_c})$ is an equivalence. 

Let $\check{G}_n$ be the $n$-th dual group of $G$ over $\Qlb$ defined in (\cite{FL}, Theorem~2.9). By construction, it is equipped with the Borel subgroup $\check{B}_n$ corresponding to $B\subset G$. 
 
 Let $\cL_T$ be the restriction of $\cL$ under the natural map $\Bun_T\to\Bun_G$. For $\check{\lambda}\in \check{\Lambda}$ and $\cF\in\Bun_T$  denote by $\cL^{\check{\lambda}}_{\cF}$ the line bundle on $X$ obtained from $\cF$ via the extensions of scalars $\check{\lambda}: T\to\Gm$. Given $\check{\lambda}\in \check{\Lambda}$ let $R^{\check{\lambda}}$ be the line bundle on $\Bun_T$ defined in (\cite{L}, 5.2.6, Example (2)). The fibre of $R^{\check{\lambda}}$ at $\cF\in\Bun_T$ is
$$
\det\RG(X, \cL^{\check{\lambda}}_{\cF})\otimes\det\RG(X, \cL^{-\check{\lambda}}_{\cF})\otimes\det\RG(X, \cO)^{-2}
$$
One has $\cL_T^{-1}\,\iso\,\otimes_{\check{\alpha}>0} R^{\check{\alpha}}$, the product being taken over the positive roots of $G$. Set
$
\kappa=-\sum_{\check{\alpha}>0} 2(\check{\alpha}\otimes \check{\alpha})
$, the sum over the positive roots of $G$. This is a symmetric bilinear form $\kappa: \Lambda\otimes\Lambda\to\ZZ$. By (\cite{FL}, Lemma~2.1), we have $\kappa=-2\check{h}\iota$. 

 Define $\theta^{Kil}=(\kappa,\lambda, c)$ by $\theta^{Kil}=(\theta^{can})^{-2\check{h}}\in \cP^{\theta}(X,\Lambda)$, here $Kil$ refers to the Killing form on $\Lambda$.
The corresponding line bundle $\lambda$ on $\Bun_T$ identifies canonically with $\cL_T$.
 
 Let $\Bunt_T$ be the gerb of $N$-th roots of $\cL_T$. Let $\D_{\zeta}(\Bunt_T)$ be the derived category of $\Qlb$-sheaves on $\Bunt_T$ on which $\mu_{N}(k)$ acts by $\zeta$, similarly for $G$. Let
$$
\Lambda^{\sharp}=\{\mu\in\Lambda\mid \kappa(\mu,\nu)\in N\ZZ\;\,\mbox{for all}\; \nu\in \Lambda\}=\{\mu\in\Lambda\mid \iota(\mu,\nu)\in n\ZZ\;\,\mbox{for all}\; \nu\in \Lambda\}
$$ 
Set $T^{\sharp}=\Gm\otimes\Lambda^{\sharp}$. 
 
  In Section~\ref{Section_3.1} we define a $\mu_N\times\mu_N$-gerb $\Bunb_{\tilde B}\to \Bunb_B$ together with a daigram
$$
\Bunt_T\getsup{\tilde\gq} \Bunb_{\tilde B}\toup{\tilde\gp}\Bunt_G
$$  
We also define the perverse sheaf $\IC_{\zeta}$ on $\Bunb_{\tilde B}$. It gives rise to the Eisenstein series functor $\Eis: \D_{\zeta}(\Bunt_T)\to \D_{\zeta}(\Bunt_G)$ given by
$$
\Eis(K)=\tilde\gp_!(\tilde\gq^*K\otimes \IC_{\zeta})[-\dim\Bun_T]
$$
The analog of (\cite{BG}, Theorem~2.1.2) in our setting is as follows.
\begin{Th}
\label{Th_Eis_commutes_with_Verdier}
i) The functor $\D_{\zeta}(\Bunt_T)\to \D(\Bunb_{\tilde B})$, $K\mapsto \tilde\gq^*K\otimes\IC_{\zeta}[-\dim\Bun_T]$ is exact for the perverse t-structures and commutes with the Verdier duality. \\
ii) The functor $\Eis$ commutes with the Verdier duality. 
\end{Th}

Write $\check{T}^{\sharp}$ for the Langlands dual to $T^{\sharp}$ over $\Qlb$. Recall that $\check{T}^{\sharp}\subset \check{B}_n\subset\check{G}_n$ is canonically included as a maximal torus. Set $\Lambda^{\sharp, +}=\Lambda^{\sharp}\cap\Lambda^+$, these are dominant weights of $\check{G}_n$. For $\nu\in \Lambda^{\sharp,+}$ denote by $V^{\nu}$ the irreducible representation of $\check{G}_n$ with highest weight $\nu$. For $\mu\in \Lambda^{\sharp}$ write $V^{\nu}(\mu)\subset V^{\nu}$ for the subspace on which $\check{T}^{\sharp}$ acts by $\mu$. 

 In Section~\ref{section_Hecke_functors_for_tilde_G} we define the action of the category of representations $\Rep(\check{G}_n)$ by the Hecke functors on $\D_{\zeta}(\Bunt_G)$. For $\nu\in\Lambda^{\sharp,+}$ we get the Hecke functor $\H^{\nu}_G: \D_{\zeta}(\Bunt_G)\to \D_{\zeta}(\Bunt_G\times X)$. 

 The action of $\Rep(\check{T}^{\sharp})$ on $\D_{\zeta}(\Bunt_T)$ by Hecke functors is defined in Section~\ref{section_Hecke_functors_for_T}. For $\nu\in\Lambda^{\sharp}$ we get the Hecke functor $\H^{\nu}_T: \D_{\zeta}(\Bunt_T)\to \D_{\zeta}(\Bunt_T\times X)$. The following is an analog of (\cite{BG}, Theorem~2.1.5) in our setting. 

\begin{Th}
\label{Th_1} For each $\nu\in\Lambda^{\sharp,+}$ and $K\in \D_{\zeta}(\Bunt_T)$ one has a functorial isomorphism
$$
\H^{\nu}_G\Eis(K)\,\iso\, \oplus_{\mu\in\Lambda^{\sharp}} (\Eis\boxtimes\id)\H^{\mu}_T(K)\otimes V^{\nu}(\mu),
$$
where $\Eis\boxtimes\id: \D_{\zeta}(\Bunt_T\times X)\to\D_{\zeta}(\Bunt_G\times X)$ is the corresponding functor. 
\end{Th} 

 One checks in addition that the isomorphism of Theorem~\ref{Th_1} is compatible with the convolution of Hecke functors. 
 
 Let $E$ be a $\check{T}^{\sharp}$-local system on $X$. For $\nu\in \Lambda^{\sharp}$ denote by $E^{\nu}$ the local system obtained from $E$ 
via the extension of scalars $\nu: \check{T}^{\sharp}\to\Gm$. Let $\cK_E$ be the eigen-sheaf on $\Bunt_T$ constructed in (\cite{L}, Proposition~2.2). It satisfies the isomorphisms $\H^{\nu}_T(\cK_E)\,\iso\, \cK_E\boxtimes E^{-\nu}[1]$ for $\nu\in\Lambda^{\sharp}$. That is, $\cK_E$ is a $E^*$-Hecke eigen-sheaf.

\begin{Cor} 
\label{Cor_very_first}
Let $E_{\check{G}_n}$ be the $\check{G}_n$-local system induced from $E^*$. Then $\Eis(\cK_E)$ is a $E_{\check{G}_n}$-Hecke eigen-sheaf in $\D_{\zeta}(\Bunt_G)$. 
\end{Cor}

\subsubsection{} Write $C^*(\check{G}_n)$ for the cocenter of $\check{G}_n$, the quotient of $\Lambda^{\sharp}$ by the roots lattice of $\check{G}_n$. The category $\Rep(\check{G}_n)$ is graded by $C^*(\check{G}_n)$ according to the action of the center of $\check{G}_n$. In Section~\ref{section_Some_gradings} we introduce the corresponding grading on the derived category $\D_{\zeta}(\Bunt_G)$. We show in Proposition~\ref{Lm_great_about_gradings} that the Hecke functors for $G$ are compatible with these gradings on $\Rep(\check{G}_n)$ and on $\D_{\zeta}(\Bunt_G)$. This generalizes (\cite{L2}, Lemma~1). We also describe the corresponding grading on the geometric Eisenstein series in Section~\ref{section_Some_gradings}. 

 Write $\rho_n$ for the half sum of positive roots of $\check{G}_n$. In Section~\ref{Section_Towards the functional equation} we define the twisted $W$-action on $\Bunt_T$ and formulate the conjectural functional equation for $\Eis$  (Conjecture~\ref{Con_functional_equation}). To this end, we also introduce the full triangluated subcategory $\D_{\zeta}(\Bunt_G)^{reg}\subset \D_{\zeta}(\Bunt_G)$ of regular complexes. The appearance of the shift by $\rho_n$ here is analogous to the shift by $\rho$ in the functional equation for the usual geometric Eisenstein series (\cite{BG}, Theorem~2.1.8). Our formulation of the functional equation is justified by the fact that it is compatible with our calculation of the constant terms of $\Eis$ for $G=\SL_2$ (see Proposition~\ref{Pp_CT_of_Eis}). Besides, it agrees with the results of \cite{L2}. In view of Theorem~\ref{Th_composing_Eis} below, the proof of the functional equation is reduced to the case of rank one. However, we don't know how to prove it for groups of rank one. 
 
 In Section~\ref{Section_action_Bun_Z(G)} we give a relation between the action of $\Bun_{Z(G)}$ on $\Bunt_G$ and Hecke functors (and also the action of $\Bun_{Z(G)}$ on the Eisenstein series). 
 
\subsubsection{Parabolic Eisenstein series} Let $P\subset G$ be a parabolic containing $B$, $M$ be its Levi factor. Write $\cI_M\subset\cI$ for the corresponding subset. Write $\Lambda_{G,P}$ for the quotient of $\Lambda$ by the span on $\alpha_i, i\in\cI_M$. Let $\cL_M$ denote the restriction of $\cL$ under $\Bun_M\to\Bun_G$. Let $\Bunt_M$ denote the gerb of $N$-th roots of $\cL_M$.  

 In Section~\ref{Section_4.1} we define a diagram of projections
$$ 
\Bunt_M\getsup{\tilde\gq} \Bunt_{\tilde P} \toup{\tilde\gp} \Bunt_G
$$
and a perverse sheaf $\IC_{\zeta}$ on $\Bunt_{\tilde P}$ generalizing our previous definition for $B$. 
It gives rise to the parabolic Eisenstein series functor
$\Eis: \D_{\zeta}(\Bunt_M)\to\D_{\zeta}(\Bunt_G)$ given by
$$
\Eis(K)=\tilde\gp_!(\tilde\gq^*K\otimes\IC_{\zeta})[-\dim\Bun_M]
$$
We write $\Eis_M^G=\Eis$ if we need to express the dependence on $M$. 
\begin{Th} 
\label{Th_Eis_P_commutes_with_Verdier}
i) The functor $\D_{\zeta}(\Bunt_M)\to \D(\Bunt_{\tilde P}), K\mapsto \tilde\gq^*K\otimes\IC_{\zeta}[-\dim\Bun_M]$ is exact for the perverse t-structures and commutes with the Verdier duality.\\
ii) The functor $\Eis: \D_{\zeta}(\Bunt_M)\to\D_{\zeta}(\Bunt_G)$ commutes with the Verdier duality.
\end{Th}

\begin{Rem} Let us explain at this point that the notations $\Bunb_B, \Bunt_P$ thoughout the paper are reserved for the corresponding Drinfeld compactifications (we assume $P\ne G$, so $\Bunt_P$ should not be confused with $\Bunt_G$). The gerbs over $\Bunb_B, \Bunt_P$ appearing in this paper such as $\Bunb_{\tilde B}, \Bunt_{\tilde P}$ (or $\Bunb_{B, \tilde G}, \Bunt_{P, \tilde G}$ below) are distinguished in our notation by some decoration above or next to the corresponding letter $B,P,G$.
\end{Rem}

 Set $\Lambda_{M,0}=\{\lambda\in\Lambda\mid\<\lambda\check{\alpha}_i\>=0\;\mbox{for all}\; i\in\cI_M\}$. Let $\check{\Lambda}_{M,0}$ denote the dual lattice. In Section~\ref{Section_4.1.1_some_vanishing} we associate to $\kappa$ a homomoprhism $\kappa_M: \Lambda_{G,P}\to \check{\Lambda}_{M,0}$ and prove the following generalization of (\cite{L}, Proposition~2.1). 

\begin{Pp} 
\label{Pp_vanishing_over_Bunt_M^theta}
Let $\theta\in \Lambda_{G,P}$ with $\kappa_M(\theta)\notin N\check{\Lambda}_{M,0}$. Then $\D_{\zeta}(\Bunt_M^{\theta})$ vanishes.
\end{Pp}

Recall that $\cI$ is canonically in bijection with the set of simple roots of $\check{G}_n$. Let $\check{M}_n\subset \check{G}_n$ be the standard Levi subgroup corresponding to $\cI_M$. 

 In Section~\ref{Section_Hecke functors for M} we define the action of the category $\Rep(\check{M}_n)$ of $\check{M}_n$-representations on $\D_{\zeta}(\Bunt_M)$ by Hecke functors. Set $\Lambda^{\sharp, +}_M=\Lambda^+_M\cap \Lambda^{\sharp}$, these are dominant weights of $\check{M}_n$. For $\nu\in\Lambda^{\sharp, +}_M$ we get the Hecke functor $\H^{\nu}_M: \D_{\zeta}(\Bunt_M)\to\D_{\zeta}(\Bunt_M\times X)$. 
 
 For $\nu\in \Lambda^{\sharp, +}_M$ denote by $U^{\nu}$ the irreducible representation of $\check{M}_n$ of highest weight $\nu$. The following is an analog of (\cite{BG}, Theorem~2.3.7) in our setting.
 
\begin{Th} 
\label{Th_Hecke_property_of_Eis^G_M}
For $\lambda\in\Lambda^{\sharp,+}$ there is an isomorphism functorial in $K\in \D_{\zeta}(\Bunt_M)$
$$
\H^{\lambda}_G\Eis_M^G(K)\;\iso\,\mathop{\oplus}\limits_{\nu\in\Lambda^{\sharp, +}_M} (\Eis^G_M\boxtimes\id)\H^{\nu}_M(K)\otimes\Hom_{\check{M}_n}(U^{\nu}, V^{\lambda}),
$$
here $\Eis^G_M\boxtimes\id: \D_{\zeta}(\Bunt_M\times X)\to \D_{\zeta}(\Bunt_G\times X)$ is the corresponding functor. 
\end{Th}

 One checks moreover that the isomorphism of Theorem~\ref{Th_Hecke_property_of_Eis^G_M} is compatible with the convolution of Hecke functors.
 
\begin{Cor} Let $E$ be a $\check{M}_n$-local system on $X$, $\cK\in\D_{\zeta}(\Bunt_M)$ be a $E$-Hecke eigen-sheaf. Then $\Eis^G_M(K)\in\D_{\zeta}(\Bunt_G)$ is a $E_{\check{G}_n}$-Hecke eigen-sheaf. Here $E_{\check{G}_n}$ is the $\check{G}_n$-local system induced from $E$.
\end{Cor}
  
 One of our main results is Theorem~\ref{Th_main_Section1} in Section~\ref{Section_Description of IC_zeta} generalizing the description of the $\IC$-sheaf of $\Bunt_P$ from \cite{ICDC} to our twisted setting. Write $\Lambda_{G,P}^{pos}$ for the $\ZZ_+$-span of $\{\alpha_i, i\in \cI-\cI_M\}$ in $\Lambda_{G,P}$. 
Pick $\theta\in\Lambda_{G,P}^{pos}$. Let $\gU(\theta)$ be a decomposition of $\theta$ as in (\cite{ICDC}, Section~1.4). Let $\check{\gu}_n(P)$ denote the Lie algebra of the unipotent radical of the standard parabolic $\check{P}_n\subset \check{G}_n$ corresponding to $\cI_M\subset\cI$.  One has a locally closed substack
$$
_{\gU(\theta)}\Bunt_P\,\iso\, \Bun_P\times_{\Bun_M} \cH^{+,\gU(\theta)}_M\hook{} \Bunt_P
$$ 
(see Section~\ref{Section_Description of IC_zeta} for these notations). Let $_{\gU(\theta)}\Bunt_{\tilde P}$ be obtained from $_{\gU(\theta)}\Bunt_P$ by the base change $\Bunt_{\tilde P}\to \Bunt_P$. Theorem~\ref{Th_main_Section1} describes the $*$-restriction of $\IC_{\zeta}$ to $_{\gU(\theta)}\Bunt_{\tilde P}$ in terms of the $\check{M}_n$-module $\check{\gu}_n(P)$ and
the twisted Satake equivalence $\Loc: \Rep(\check{M}_n)\,\iso\, \PPerv^{\natural}_{M,G,n}$ for $\check{M}_n$ (see Section~\ref{Section_Description of IC_zeta}). 
The proof actually establishes more (Theorem~\ref{Th_3} and Corollary~\ref{Cor_3} do not reduce to Theorem~\ref{Th_main_Section1}).

 In Section~\ref{section_composing} we prove the following result, which is an analog of (\cite{BG}, Theorem~2.3.10) in our setting.
\begin{Th} 
\label{Th_composing_Eis}
There is an isomorphism of functors $\D_{\zeta}(\Bunt_T)\to \D_{\zeta}(\Bunt_G)$
$$
\Eis_T^G\,\iso\, \Eis_M^G\comp \Eis_T^M
$$
\end{Th}
 
\subsubsection{} In Section~\ref{Section_The case of SL_2} we specialize to the case of $G=\SL_2$. As in Section~\ref{Section_Zastava spaces}, we have a $\mu_N$-gerb $\wt Z^{\theta}\to Z^{\theta}$ and a local version $\IC_{Z^{\theta},\zeta}$ of the perverse sheaf $\IC_{\zeta}$. Here $\IC_{Z^{\theta},\zeta}$ is a perverse sheaf on $\wt Z^{\theta}$ (see Sections~\ref{Section_Zastava spaces} and \ref{Section_The case of SL_2} for notations). For $G=\SL_2$ the Zastava space $Z^{\theta}$ is a vector bundle over $X^{\theta}$, and it is important to calculate the Fourier transform $\Four_{\psi}(\IC_{Z^{\theta},\zeta})$ over the dual vector bundle. This calculation at the classical level is a part of the theory of Weyl group multiple Dirichlet series (see \cite{BBF}, \cite{B} for a survey).

 The description of $\IC_{Z^{\theta},\zeta}$ is known (see Theorem~\ref{Th_3} and Corollary~\ref{Cor_3}). For $n=2$ the description of $\Four_{\psi}(\IC_{Z^{\theta},\zeta})$ is easily reduced to the description of $\IC_{Z^{\theta},\zeta}$ itself (see Section~\ref{Section_Generalization}, this was also used in \cite{L2}). For $n\ge 3$ we can not completely describe $\Four_{\psi}(\IC_{Z^{\theta},\zeta})$, and only establish Proposition~\ref{Lm_6.2_about_Fourier_of_IC_zeta}, which calculates the desired Fourier transform over the open substack $_{\Omega}\check{\tilde Z}^{\theta}_{max}\subset {_{\Omega}\check{\tilde Z}^{\theta}}$ (see Section~\ref{Section_Generalization} for notations). 
 
  The answer in Proposition~\ref{Lm_6.2_about_Fourier_of_IC_zeta} is given in terms of the perverse sheaf $\IC_{_{\Omega}\tilde X^{\theta}, \bar\zeta}$ that has been completely described in \cite{BFS} in terms of cohomologies of a (part of) the quantum $\mathfrak{sl}_2$ at a suitable root of unity. This is a manifestation of the phenomenon that cohomology of quantum groups appear in the quantum geometric Langlands program (the quantum groups were brought into the quantum geometric Langlands program in \cite{Ga2, L5}). 
  
  In Proposition~\ref{Pp_nondegenerate_Whit_coeff_of_Eis} we give a global application of Proposition~\ref{Lm_6.2_about_Fourier_of_IC_zeta}, it expresses the non-degenerate Whittaker coefficients of $\Eis(K)$, $K\in\D_{\zeta}(\Bunt_T)$ in terms of the perverse sheaf $\IC_{\check{Z}^{\theta}_c, \bar\zeta}$. This, in turn, yields a formula for the first Whittaker coefficient of $\Eis(K)$ (see Corollary~\ref{Cor_6.1} and \ref{Cor_6.2}). The complex appearing in Corollary~\ref{Cor_6.1} is an $\ell$-adic analog of the space of conformal blocks in Wess-Zumino-Witten model studied in \cite{BFS}.
  
  In Section~\ref{Section_Constant term of Eis} we calculate the constant terms of $\Eis(K), K\in \D_{\zeta}(\Bunt_T)$ in terms of \select{integral Hecke functors} for $\Bunt_T$. Here `integral' means that we apply Hecke functors at a collection of points and further integrate over this collection of points. The answer is given in Proposition~\ref{Pp_CT_of_Eis}, which (together with the results of \cite{L2}) explains our formulation of the functional equation. 
  
\subsubsection{Some special sheaves} Let $E$ be a $\check{T}^{\sharp}$-local system on $X$, $\cK_E\in\D_{\zeta}(\Bunt_T)$ be the $E$-Hecke eigensheaf as in Corollary~\ref{Cor_very_first}. This is a local system over the components of $\Bunt_T$ corresponding to $\Lambda^{\sharp}$. In Section~\ref{Section_Some special sheaves} we describe some irreducible perverse sheaves $\IC(E,d)\in\D_{\zeta}(\Bunt_G)$, $d>0$ that appear in $\Eis(\cK_E)$. 
  
  We then specialize to the case of genus $g=0$. In this case $E$ is trivial, we set $\IC_d=\IC(\Qlb, d)$, $d>0$ for brevity, and also define an irreducible perverse sheaf $\IC_0$ that appear in $\Eis(\cK_E)$. Then any irreducible perverse sheaf appearing in $\Eis(\cK_E)$ is isomorphic to some $\IC_d$, $d\ge 0$. Let $\cP_{\zeta,n}$ denote the category of pure complexes on $\Bunt_G$, which are the direct sums of  $\IC_d[r](\frac{r}{2})$, $d\ge 0$, $r\in\ZZ$. Then $\cP_{\zeta,n}$ is a module over $\Rep(\check{G}_n)$ acting by Hecke functors.
  
  We explicitely describe the action of Hecke functors on each of the objects $\IC_d$, $d\ge 0$ (see Lemma~\ref{Lm_6.5} and Theorem~\ref{Th_last}). We also describe all the $*$-fibres of each of the perverse sheaf $\IC_d$, $d\ge 0$ (see Lemmas~\ref{Lm_restricting_IC_d_case_d>0} and \ref{Lm_great_fibres_of_IC_0}). Here is an immediate consequence of our results.
Recall that $\check{G}_n\,\iso\, \SL_2$ for $n$ even (resp., $\check{G}_n\,\iso\, \PSL_2$ for $n$ odd).  
  
\begin{Cor} Let $m\ge 1$ with $n-m$ even. One has the equivalence $\cP_{\zeta, n}\,\iso\, \cP_{\zeta, m}$ sending $\IC_d[r](\frac{r}{2})$ to itself (and preserving direct sums). This equivalence commutes with the Hecke action (with respect to the evident isomorphism $\check{G}_n\,\iso\, \check{G}_m$).
\end{Cor}

\begin{Def} 
\label{Def_theta-sheaves}
For $n$ odd set $\Aut=\IC_0$. For $n$ even set $\Aut=\IC_0\oplus \IC_1$. As in \cite{L4}, we call $\Aut$ the theta-sheaf on $\Bunt_G$.
\end{Def}

\begin{Cor} 
\label{Cor_theta_sheaves}
The perverse sheaf $\Aut$ is a Hecke eigen-sheaf corresponding to the trivial $\check{G}_n$-local system and the principal $\SL_2$ of Arthur homomorphism $\SL_2\to \check{G}_n$.
\end{Cor}  

\subsubsection{} Let us indicate some problems arising for future research:
\begin{itemize}
\item[1)] Is it true that the theta-sheaves $\Aut$ satisfying the Hecke property as in Corollary~\ref{Cor_theta_sheaves} exist for all $G$ and any curve $X$?
\item[2)] Is it true that $\Aut$ constructed in Corollary~\ref{Cor_theta_sheaves} is the geometric analog of a matrix coefficient of a suitable nonramified automorphic representation of the corresponding $\Fq^*/(\Fq^*)^n$-metaplectic cover of $G$? Construct the corresponding representations for a local and global field (according to \cite{Gao}, they should exist). 
\item[3)] Calculate the Fourier coefficients of the theta-sheaves given in Definition~\ref{Def_theta-sheaves}. 
\end{itemize}

\subsubsection{} In Appendix~\ref{Section_appendixA} we assume in addition that $k=\Fq$, and prove Theorem~\ref{Th_trace_of_Frob_Eis'} below. 

 A version of our results holds also over $\Fq$, in particular the construction of $\cK_E$ and the description of $\IC_{\zeta}$ given in Corollary~5.1. It is understood that in the corresponding description the Tate twists are recovered accordingly. Let $E$ be a $\check{T}^{\sharp}$-local system on $X$, $\cK_E$ be the eigen-sheaf on $\Bunt_T$ constructed in (\cite{L}, Proposition~2.2). For $\mu\in\Lambda^{\sharp}$ write $\cK^{\mu}_E$ for the restriction of $\cK_E$ to $\Bunt_T^{\mu}$. By (\cite{L}, Proposition~2.1), $\cK_E^{\mu}$ vanishes unless $\mu\in\Lambda^{\sharp}$. 
 
 For $\mu\in\Lambda^{\sharp}$ define $\Eis'(\cK_E^{\mu})$ as 
$$
\gp_!(\gq^*(\cK_E^{\mu})\otimes \IC_{\zeta}),
$$ 
where the maps $\gp,\gq$ are those of the diagram $\Bunt_T\getsup{\gq} \Bun_{\tilde B}\toup{\gp} \Bunt_G$. Here $\Bun_{\tilde B}$ is obtained from $\Bun_B$ via the base change $\Bunt_T\times\Bunt_G\to\Bun_T\times\Bun_G$. 

 Denote by $\Funct(\Eis(\cK^{\mu}_E))$ (resp., $\Funct(\Eis'(\cK^{\mu}_E))$) the function trace of Frobenius on the set $\Bunt_T(\Fq)$ corresponding to 
$\Eis(\cK^{\mu}_E)$ and $\Eis'(\cK^{\mu}_E)$ respectively. 

 For $B=P$ the set $J$ defined in Section~\ref{Section_Description of IC_zeta} identifies with the set of positive roots of $\check{G}_n$ for $\check{B}_n$. We denote by $\Lambda_{G,B}^{pos,pos}$ the free abelian semigroup with base $J$. Set $\Lambda^{\sharp, pos}=\Lambda^{\sharp}\cap\Lambda^{pos}$. Let $\bar c_P: \Lambda_{G,B}^{pos,pos}\to \Lambda^{\sharp, pos}$ be the morphism of semigroups given on $J$ by the natural inclusion $J\hook{} \Lambda^{\sharp, pos}$. We write the elements of $\Lambda_{G,B}^{pos,pos}$ as $\gB(\theta)=\sum_{\nu\in J} n_{\nu}\nu$ with $\theta=\bar c_P(\gB(\theta))$.
 
  The following is an analog of (\cite{BG}, Theorem~2.2.11) in our setting.

\begin{Th}
\label{Th_trace_of_Frob_Eis'} The function $\Funct(\Eis(\cK^{\mu}_E))$ vanishes unless $\mu\in \Lambda^{\sharp}$, and in the latter case it equals
$$
\sum_{\gB(\theta)\in \Lambda_{G,B}^{pos,pos}} \Funct(\Eis'(\cK_E^{\mu-\theta})) \prod_{\nu\in J}\Tr(\Fr, \RG(X^{(n_{\nu})}\otimes\ov{\Fq}, (E^{-\nu})^{(n_{\nu})})\otimes \Qlb(n_{\nu})),
$$
where $\gB(\theta)=\sum_{\nu\in J} n_{\nu}\nu$. Here $X^{(m)}$ denotes the $m$-th symmetric power of $X$, and for a local system $W$ on $X$, $W^{(m)}$ denotes its $m$-th symmetric power. 
\end{Th} 

 As in (\cite{BG}, Theorem~2.2.12), Theorem~\ref{Th_trace_of_Frob_Eis'} may be reformulated as follows in terms of generating series.
 
 Consider the group ring $\Qlb[\Lambda^{\sharp}]$, the ring of regular functions on the torus $\check{T}^{\sharp}$. For $\mu\in\Lambda^{\sharp}$ denote by $t^{\mu}$ the corresponding element of $\Qlb[\Lambda^{\sharp}]$. 
Form a completed ring $\widehat{\Qlb[\Lambda^{\sharp}]}$ by allowing infinite expressions of the form
$$
\sum_{\mu} a_{\mu} t^{\mu},
$$
where $\mu$ runs over a subset of $\Lambda^{\sharp}$  of the form $\mu\ge\mu'$, where $\mu'$ is some fixed element of $\Lambda$. 

 The classical Eisenstein series can be thought of as a $\widehat{\Qlb[\Lambda^{\sharp}]}$-valued function on $\Bunt_G(\Fq)$ equal to
$$
\Eis_{cl}(\cK_E)(t)=\sum_{\mu\in\Lambda^{\sharp}} \Funct(\Eis'(\cK_E^{\mu})) t^{\mu}\, .
$$
Consider the modified Eisenstein series defined as
$$
\Eis_{mod}(\cK_E)(t)=\sum_{\mu\in\Lambda^{\sharp}} \Funct(\Eis(\cK_E^{\mu})) t^{\mu},
$$
viewed as a function $\Bunt_G(\Fq)\to \widehat{\Qlb[\Lambda^{\sharp}]}$. 
For $\nu\in J$ consider the abelian $L$-series $L(E^*, \nu, t)\in \widehat{\Qlb[\Lambda^{\sharp}]}$ equal to
$$
\sum_{n\ge 0} \Tr(\Fr, \RG(X^{(n)}\otimes\ov{\Fq}, (E^{-\nu})^{(n)})\otimes \Qlb(n)) t^{n\nu}
$$
\begin{Th} For any $\check{T}^{\sharp}$-local system $E$ on $X$ one has
$$
\Eis_{mod}(\cK_E)(t)=\Eis_{cl}(\cK_E)(t) \prod_{\nu\in J} L(E^*, \nu, t)
$$
\end{Th}

   It is known that $\Eis_{cl}(\cK_E)(t)$ satisfies the functional equation (cf. \cite{MW}, \cite{Gao}). This is a strong argument supporting our geometric functional equation (Conjecture~\ref{Con_functional_equation}). 
   
\section{Principal geometric Eisenstein series}

\subsection{Definitions} 
\label{Section_3.1}
Keep notations of Section~\ref{Section_Main_results}. Denote by $\Bunb_B$ the Drinfeld compactification of $\Bun_B$ from \cite{BG}. We have the diagram for the Drinfeld compactification $\Bun_T\getsup{\bar\gq} \Bunb_B\toup{\bar\gp}\Bun_G$. Write $\Bunb_{B, \tilde G}=\Bunb_B\times_{\Bun_G}\Bunt_G$. 
Set 
$$
\Bunb_{\tilde B}=\Bunb_{B, \tilde G}\times_{\Bun_T}\Bunt_T
$$
A point of $\Bunb_{\tilde B}$ is given by $\cF_T\in\Bun_T, \cF\in\Bun_G$, a collection of inclusions
$$
\nu^{\check{\lambda}}: \cL^{\check{\lambda}}_{\cF_T} \hook{} \cV^{\check{\lambda}}_{\cF}
$$
for all dominant weights $\check{\lambda}$ of $T$ satisfying the Plucker relations, and $\ZZ/2\ZZ$-graded lines of parity zero $\cU,\cU_G$ equipped with 
$$
\cU^N\,\iso\, (\cL_T)_{\cF_T}, \;\;\; \cU_G^N\,\iso\, \cL_{\cF}
$$ 
Here $\cV^{\check{\lambda}}$ is the Weyl module corresponding to $\check{\lambda}$. Consider the open substack $\Bun_B\subset \Bunb_B$, let $\Bun_{\tilde B}$ be the restriction of $\Bunb_{\tilde B}$ to this open substack. For a point of $\Bun_B$ as above we have canonically $\cL_{\cF}\,\iso\, (\cL_T)_{\cF_T}$. 
 
  Let $a: \Spec k\to B(\mu_N)$ be the natural map, let $\cL_{\zeta}$ be the direct summand in $a_*\Qlb$ on which $\mu_N(k)$ acts by $\zeta$. 
Let 
$$
\Bun_{B, \tilde G}=\Bun_B\times_{\Bun_G}\Bunt_G,
$$
it classifies $\cF_B$, $\cU_G$ and an isomorphism $\cU_G^N\,\iso\, \cL_{\cF_B}$. We get an isomorphism 
\begin{equation}
\label{eq_one}
B(\mu_N)\times \Bun_{B,\tilde G}\,\iso\, \Bun_{\tilde B}
\end{equation} 
sending $(\cF_B, \cU_G, \cU_0\in B(\mu_N))$ with $\cU_0^N\,\iso\, k$ to 
$(\cF_B, \cU_G, \cU)$ with $\cU=\cU_G\otimes \cU_0^{-1}$. Write
\begin{equation}
\label{diag_def_of_Eis}
\Bunt_T\getsup{\tilde\gq} \Bunb_{\tilde B}\toup{\tilde\gp}\Bunt_G
\end{equation}
for the projections, so 
$$
\tilde\gq(\cF_T,\cF, \nu, \cU, \cU_G)=(\cF_T, \cU)\;\;\;\mbox{and}\;\;\;
\tilde\gp(\cF_T,\cF, \nu, \cU, \cU_G)=(\cF, \cU_G)
$$
View $\cL_{\zeta}\boxtimes\IC(\Bun_{B, \tilde G})$ as a perverse sheaf on $\Bun_{\tilde B}$ via (\ref{eq_one}). 
Let $\IC_{\zeta}$ be its intermediate extension to $\Bunb_{\tilde B}$. 
\begin{Def} For $K\in \D_{\zeta}(\Bunt_T)$ set
$$
\Eis(K)=\tilde\gp_!(\tilde\gq^*K\otimes \IC_{\zeta})[-\dim\Bun_T]
$$
\end{Def}

Let $\mu_N(k)\times\mu_N(k)$ act on $\Bunb_{\tilde B}$ by 2-automorphisms so that $(a,a_G)$ acts as $a$ on $\cU$, as $a_G$ on $\cU_G$ and trivially on $(\cF_T,\cF,\nu)$. Then $(a,a_G)$ acts on $\IC_{\zeta}$ by $\zeta(\frac{a_G}{a})$. If $K\in \D_{\zeta}(\Bunt_T)$ then $a\in\mu_N(k)\subset \Aut(U)$ acts on $K$ as $\zeta(a)$. Then $(a,a_G)$ acts on ${\tilde\gq}^*K\otimes \IC_{\zeta}$ as $\zeta(a_G)$. So, $a_G$ acts on $\Eis(K)$ by $\zeta(a_G)$.
 
\subsection{Hecke functors}
\label{section_Hecke_functors_for_tilde_G}
Let $\Lambda^{\sharp,+}=\Lambda^{\sharp}\cap\Lambda^+$. 
We use some notations from \cite{FL}. In particular, $\bO=k[[t]]\subset \bF=k((t))$, $\Gr_G=G(\bF)/G(\bO)$. By abuse of notations, $\cL$ will also denote the $\ZZ/2\ZZ$-graded line bundle on $\Gr_G$ whose fibre at $gG(\bO)$ is $\det(\gg(\bO):\gg(\bO)^g)$. Write $\Gra_G$ for the punctured total space of $\cL$. Let $\wt\Gr_G$ be the stack quotient of $\Gra_G$ by $\Gm$, where $z\in\Gm$ acts as the multiplication by $z^N$ with $N=2\check{h}n$. 
 
  Let $\Perv_{G,n}$ be the category of $G(\bO)$-equivariant perverse sheaves on $\Gra_G$ with $\Gm$-monodromy $\zeta$. Let 
$$
\PPerv_{G,n}=\Perv_{G,n}[-1]\subset \D(\Gra_G)
$$  
We view $\PPerv_{G,n}$ as the category of $G(\bO)$-equivariant perverse sheaves on $\wt\Gr_G$, on which $\mu_{N}(k)$ acts by $\zeta$. Namely, a $G(\bO)$-equivariant perverse sheaf $K$ on $\wt\Gr_G$, on which $\mu_{N}(k)$ acts by $\zeta$, is identified with $\pr^*K\in \PPerv_{G,n}$, where $\pr: \Gra_G\to \wt\Gr_G$ is the quotient map under the $\Gm$-action. 
 
 As in (\cite{FL}, Section~2.1), we pick a trivialization $\gg^{\check{\alpha}}\,\iso\, k$ of the root space for all the positive roots $\check{\alpha}$ and denote by $\Phi$ this collection of trivializations. 
 
 Let $\Omega(\bO)$ denote the completed module of relative differentials of $\bO$ over $k$. Write $\Omega(\bO)^{\frac{1}{2}}$ for the groupoid of square roots of $\Omega(\bO)$. We pick $\cE\in \Omega(\bO)^{\frac{1}{2}}$. As in (\cite{FL}, Section~2.1), for $\nu\in \Lambda^{\sharp,+}$ we define
the local system $E^{\nu}_{\cE}$ on $\Gra_G^{\nu}$ and $\cA^{\nu}_{\cE}\in\PPerv_{G,n}$. By abuse of notation, $E^{\nu}_{\cE}$ also denotes the corresponding local system on $\wt\Gr_G^{\nu}$.

 Write $\cH_G$ for the Hecke stack classifying $\cF,\cF'\in\Bun_G$, $x\in X$ and an isomorphism $\beta: \cF\mid_{X-x}\,\iso\, \cF'\mid_{X-x}$. We have a diagram
$$
\Bun_G\times X\;\getsup{h^{\la}_G\times\pi}\; \cH_G\;\toup{h^{\ra}_G}\;\Bun_G,
$$
where $h^{\la}_G$ (resp., $h^{\ra}_G$) sends the above point to $\cF$ (resp., to $\cF'$). Here $\pi(\cF,\cF',\beta,x)=x$. These notations agree with \cite{BG}. 

 For $\nu\in\Lambda^+$ we define $\ov{\cH}_G^{\nu}$ as in (\cite{BG}, Section~2.1.4). So, the closed substack $\ov{\cH}_G^{\nu}\subset \cH_G$ is given by the condition that for each $G$-module $\cV$ whose weights are $\le \check{\lambda}$ one has
$$
\cV_{\cF}(-\<\nu,\check{\lambda}\>x)\subset \cV_{\cF'}\subset \cV_{\cF}(-\<w_0(\nu),\check{\lambda}\>x)
$$
This is equivalent to requiring that $\cF'$ is in the position $\le\nu$ with respect to $\cF$ at $x$. Let $\cH^{\nu}_G\subset \ov{\cH}^{\nu}_G$ be the open substack given by the property that that $\cF'$ is in the position $\nu$ with respect to $\cF$ at $x$.

 Let $\Gr_{G,X}$ be the ind-scheme classifying $x\in X$, and a $G$-torsor $\cF$ with a trivialization $\beta: \cF\,\iso\, \cF^0_G\mid_{X-x}$. Let $G_X$ be the functor classifying $x\in X$, and an automorphism of $\cF^0_G$ over the formal neighbourhood of $x$. Write $\cL_X$ for the ($\ZZ/2\ZZ$-graded of parity zero) line bundle on $\Gr_{G,X}$ whose fibre at $(\cF,x,\beta)$ is $\det\RG(X, \gg\otimes\cO_X)\otimes \det\RG(X, \gg_{\cF})^{-1}$. Let $\wt\Gr_{G,X}$ be the gerb of $N$-th roots of $\cL_X$ over $\Gr_{G,X}$. Let $\Gra_{G,X}$ be the punctured total space of the line bundle $\cL_X$ on $\Gr_{G,X}$. 
 
 Write $\Bun_{G,X}$ for the stack classifying $(\cF\in\Bun_G, x\in X,\nu)$, where $\nu: \cF\,\iso\, \cF^0_G\mid_{D_x}$ is a trivialization over the formal neighbourhood $D_x$ of $x$. Note that $\Bun_{G,X}$ is a $G_X$-torsor over $\Bun_G\times X$. Set $\Bunt_{G,X}=\Bunt_G\times_{\Bun_G}\Bun_{G,X}$.
 
 Let $\gamma^{\la}$ (resp., $\gamma^{\ra}$) denote the isomorphism $\Bun_{G,X}\times_{G_X} \Gr_{G,X}\,\iso\, \cH_G$ such that the projection to the first term corresponds to $h^{\la}_G$ (resp., $h^{\ra}_G$). The line bundle $\cL\boxtimes \cL_X$ on $\Bun_{G,X}\times \Gr_{G,X}$ is naturally $G_X$-equivariant, we denote by $\cL\tboxtimes \cL_X$ its descent to 
$$
\Bun_{G,X}\times_{G_X} \Gr_{G,X}
$$  
Note that $\ov{\cH}^{\nu}_G$ identifies with $\Bun_{G,X}\times_{G_X} \ov{\Gr}_{G,X}^{\nu}$ under $\gamma^{\la}$. We have canonically 
\begin{equation}
\label{iso_factorization_of_detRG}
(\gamma^{\ra})^*(h^{\la}_G)^*\cL\,\iso\, \cL\tboxtimes \cL_X
\end{equation}
 
 Let $\cH_{\tilde G}$ be the stack obtained from $\Bunt_G\times\Bunt_G$ be the base change $h^{\la}\times h^{\ra}: \cH_G\to\Bun_G\times\Bun_G$. Denote by $\tilde h^{\la}_G$, $\tilde h^{\ra}_G$ the projections in the diagram
$$
\begin{array}{ccccc}
\Bunt_G & \getsup{\tilde h^{\la}_G} & \cH_{\tilde G} & \toup{\tilde h^{\ra}_G} & \Bunt_G\\
\downarrow && \downarrow &&\downarrow\\
\Bun_G& \getsup{h^{\la}_G}& \cH_G &\toup{h^{\ra}_G}&\Bun_G 
\end{array}
$$
The stack $\cH_{\tilde G}$ classifies $(\cF,\cF',\beta,x)\in\cH_G$ and lines $\cU, \cU'$ equipped with $\cU^N\,\iso\, \cL_{\cF}$, $\cU'^N\,\iso\, \cL_{\cF'}$. 

 The isomorphism (\ref{iso_factorization_of_detRG}) yields a $G_X$-torsor $\tilde\gamma^{\ra}: \Bunt_{G,X}\times_X \wt\Gr_{G,X}\to \cH_{\tilde G}$ extending the $G_X$-torsor 
$$
\Bun_{G,X}\times_X \Gr_{G,X}\to \Bun_{G,X}\times_{G_X} \Gr_{G,X}\toup{\gamma^{\ra}}\cH_G
$$
Namely, it sends $(x, \beta': \cF'\,\iso\, \cF^0_G\mid_{D_x}, \; \beta_1: \cF_1\,\iso\, \cF^0_G\mid_{X-x},\; \cU'^N\,\iso\, \cL_{\cF'}, \;\cU_1^N\,\iso\, (\cL_X)_{(\cF_1,\beta_1,x)})$ to 
$$
(\cF, \cF', \cU,\cU', \beta: \cF\mid_{X-x}\,\iso\, \cF'\mid_{X-x}),
$$
where $\cF$ is obtained as the gluing of $\cF'\mid_{X-x}$ with $\cF_1\mid_{D_x}$ via $\beta_1^{-1}\comp\beta': \cF'\,\iso\, \cF_1\mid_{D_x^*}$. We have canonically $\cL_{\cF'}\otimes (\cL_X)_{(\cF_1, \beta_1,x)}\,\iso\, \cL_{\cF}$, and $\cU=\cU'\otimes\cU_1$ is equipped with the induced isomorphism $\cU^N\,\iso\, \cL_{\cF}$. 

 Let $\Sph(\wt\Gr_{G,X})$ be the category of $G_X$-equivariant perverse sheaves on $\wt\Gr_{G,X}$. Now for $\cS\in \Sph(\wt\Gr_{G,X})$ and $\cT\in \D(\Bunt_G)$ we can form their twisted tensor product $(\cT\tboxtimes \cS)^r$, which is the descent via $\tilde\gamma^{\ra}$. Similarly, one may define $\tilde\gamma^{\la}$ and the complex $(\cT\tboxtimes \cS)^l$ on $\cH_{\tilde G}$ (as in \cite{BG}, Section~3.2.4).  
 
 We also denote the composition $\cH_{\tilde G}\to \cH_G\toup{\pi} X$ by $\pi$. 
Let $\cS\in \Sph(\wt\Gr_{G,X})$ and $\cT\in \D_{\zeta}(\Bunt_G)$. If $a_1\in \mu_N(k)\subset \Aut(\cU_1)$ acts on $\cS$ as $\zeta(a_1)$
then $(a,a')\in \mu_N(k)\times\mu_N(k)\subset \Aut(\cU)\times\Aut(\cU')$ acts on $(\cT\tboxtimes \cS)^r$ as $\zeta(a)$, so
$(\tilde h^{\la}_G\times\pi)_!((\cT\tboxtimes \cS)^r)\in \D_{\zeta}(\Bunt_G\times X)$. 

 As in \cite{FL} write $\PPerv_{G,n,X}$ for the category of compexes $K\in \D(\wt\Gr_{G,X})$ such that $K[1]$ is perverse, $G_X$-equivariant, and $\mu_N(k)$ acts on $K$ by $\zeta$. Our choice of $\cE_X$ (see Section~\ref{Section_Notations}) yields a fully faithful localization functor 
$$
\tau^0: \PPerv_{G,n}\to \PPerv_{G,n,X}
$$
defined in (\cite{FL}, Section~2.3). Now for $\nu\in \Lambda^{\sharp,+}$ we get $\cA^{\nu}:=\tau^0(\cA^{\nu}_{\cE})\in \PPerv_{G,n,X}$. Define
\begin{equation}
\label{functor_H^nu_G_def}
\H^{\nu}_G: \D_{\zeta}(\Bunt_G)\to \D_{\zeta}(\Bunt_G\times X)
\end{equation}
by
$$
\H^{\nu}_G(\cT)=(\tilde h^{\la}_G\times\pi)_!((\cT\tboxtimes \cA^{-w_0(\nu)})^r)
$$
We used that $\Lambda^{\sharp,+}$ is preserved by $-w_0$. This definition agrees with \cite{BG}. For $\nu\in \Lambda^{\sharp,+}$
write also $\IC^{\nu}=(\IC(\Bunt_{G,X})\tboxtimes \cA^{-w_0(\nu)})^r$. Let 
$$
\ov{\cH}^{\nu}_{\tilde G}=\cH_{\tilde G}\times_{\cH_G}\ov{\cH}^{\nu}_G
$$ 
Then $\IC^{\nu}$ is an irreducible perverse sheaf, the extension by zero from $\ov{\cH}^{\nu}_{\tilde G}$. 
For $\cT\in \D_{\zeta}(\Bunt_G)$ we may rewrite 
\begin{equation}
\label{formula_for_H^nu_G_Section04}
\H^{\nu}_G(\cT)=(\tilde h^{\la}_G\times\pi)_!((\tilde h^{\ra}_G)^*\cT\otimes \IC^{\nu}))[-\dim\Bun_G]
\end{equation}

 Recall the covariant functor $\star: \PPerv_{G,n,\zeta^{-1}}\to \PPerv_{G,n,\zeta}$ defined in (\cite{FL}, Remark~2.8), see also (\cite{L6}, Remark~2.2).
For $\nu\in\Lambda^{\sharp,+}$ it sends $\cA^{\nu}_{\cE}$ to $\cA^{-w_0(\nu)}_{\cE}$. More generally, for $\cS\in \PPerv_{G,n}, \cT\in \D_{\zeta}(\Bunt_G)$ set 
$$
\H^{\la}_G(\cS,\cT)=(\tilde h^{\la}_G\times\pi)_!((\cT\tboxtimes \tau^0(\star\cS))^r)\;\;\;\mbox{and}\;\;\;
\H^{\ra}_G(\cS,\cT)=(\tilde h^{\ra}_G\times\pi)_!(\cT\tboxtimes \tau^*(\cS))^l
$$
These are analogs of the corresponding functors from (\cite{BG}, Section~3.2.4), they satisfy similar properties. In particular, $\H^{\la}_G$ (resp., $\H^{\ra}_G$) defines a left (resp., right) action on $\D_{\zeta}(\Bunt_G)$. 

\subsubsection{Hecke functors for $T$} 
\label{section_Hecke_functors_for_T}
For $\nu\in \Lambda^{\sharp}$ define the Hecke functor  
$$
\H^{\nu}_T: \D_{\zeta}(\Bunt_T)\to \D_{\zeta}(\Bunt_T\times X)
$$
as follows. Our definition will be consistent with (\ref{functor_H^nu_G_def}) but will differ from those of \cite{L}. 

 Recall $\theta^{Kil}=(\kappa, \lambda, c)$ from Section~\ref{Section_0.2}. For $\nu\in \Lambda$ the line bundle $\lambda^{\nu}$ is the restriction of $\cL_T$ under $X\to \Bun_T$, $x\mapsto \cF^0_T(\nu x)$. Note that $\lambda^{\nu}\,\iso\,\Omega^{\check{h}\iota(\nu,\nu)}$ by (\cite{L}, Lemma~4.1). 

  Let $(\iota, \tilde\Lambda^{\sharp, can})$ denote the restriction of $(\iota, \tilde\Lambda^{can})$ to $\Lambda^{\sharp}$, its is equipped with a $W$-equivariant structure. We pick an object $(\frac{\iota}{n}, \tilde\Lambda^{\sharp})$ in $\cE^s(T^{\sharp})$ and a $W$-equivariant isomorphism 
$$
(\frac{\iota}{n}, \tilde\Lambda^{\sharp})^n\,\iso\, (\iota, \tilde\Lambda^{\sharp, can})
$$
in $\cE^s(T^{\sharp})$. Now using (\cite{L}, Lemma~4.1), the above object yields $(\tau, \frac{\kappa}{N}, c)\in \cP^{\theta}(X, \Lambda^{\sharp})$ and a $W$-equivariant isomorphism
\begin{equation}
\label{iso_root_of_theta_sharp}
(\tau, \frac{\kappa}{N}, c)^N\,\iso\, \theta^{Kil}\mid_{\Lambda^{\sharp}}
\end{equation}

Note that $\frac{\kappa}{N}=-\frac{\iota}{n}: \Lambda^{\sharp}\otimes\Lambda^{\sharp}\to\ZZ$ may take odd values, so $\tau$ is a super line bundle in general.

 We also write $\tau$ for the line bundle on $\Bun_{T^{\sharp}}$ obtained from (\ref{iso_root_of_theta_sharp}) applying the corresponding functor $\cP^{\theta}(X,\Lambda^{\sharp})\to \cPic(\Bun_{T^{\sharp}})$ as in (\cite{L}, Section~4.2.1).
It is equipped with a $W$-equivariant structure and a $W$-equivariant isomorphism $\tau^N\,\iso\, i_X^*\lambda$. Here $i_X: \Bun_{T^{\sharp}}\to\Bun_T$ is the natural map.

 For $\nu\in\Lambda^{\sharp}$ let $m^{\nu}: \Bunt_T\times X\to\Bunt_T$ be the map sending $x\in X$, $(\cF, \cU)\in\Bunt_T$ to $(\cF'=\cF(-\nu x), \cU')$, where 
\begin{equation}
\label{cU'_for_def_of_mnu}
\cU'=\cU\otimes (\cL^{-\frac{\kappa(\nu)}{N}}_{\cF})_x\otimes \tau_{\cO(-\nu x)}
\end{equation}
equipped with the induced isomorphism $\cU'^N\,\iso\, (\cL_T)_{\cF'}$. It coincides with the map denoted $m_{-\nu}$ in (\cite{L}, Section~5.2.3). The Hecke functor 
$$
\H^{\nu}_T: \D_{\zeta}(\Bunt_T)\to \D_{\zeta}(\Bunt_T\times X)
$$ 
is defined by $\H^{\nu}_T(K)=(m^{\nu})^*K[1]$. 

\subsection{Proof of Theorem~\ref{Th_1}} We will use the following result. Let $\Gr_T$, $\Gr_B$ be the affine grassmanians for $T,B$. Let $\wt\Gr_T\to \wt\Gr_B\to \wt\Gr_G$ be obtained from $\Gr_T\to \Gr_B\to\Gr_G$ by the base change $\wt\Gr_G\to\Gr_G$. As in (\cite{FL}, Section~4.1) for $\nu\in\Lambda^{\sharp,+}$ and $\mu\in \Lambda$ one has the diagram of ind-stacks
$$
\wt\Gr_T^{\mu}\getsuplong{\gt_B^{\mu}} \wt\Gr_B^{\mu} \touplong{\gs_B^{\mu}} \wt\Gr_G
$$
The connected component $\wt\Gr_T^{\mu}$ is the one containing $t^{\mu}T(\bO)$, similarly for $\wt\Gr_B^{\mu}$. 

 If $\mu\in\Lambda^{\sharp}$ as in (\cite{FL}, Section~4.2) we denote by $a_{\mu}: \cE_{\bar c}^{\iota(\mu,\mu)/n}-\{0\}\to \Omega_{\bar c}^{\check{h}\iota(\mu,\mu)}-\{0\}$ the map $z\mapsto z^{2\check{h}n}$. 
 
 \begin{Pp}[\cite{FL}]
\label{Pp_description_weight_space}
 Let $\nu\in\Lambda^{\sharp,+}$, $\mu\in \Lambda$. The complex $a_{\mu}^*(\gt^{\mu}_B)_!(\gs_B^{\mu})^*\cA^{\nu}_{\cE}$ vanishes unless $\mu\in \Lambda^{\sharp}$. In the latter case this complex is constant and identifies canonically with $V^{\nu}(\mu)[-\<\mu, 2\check{\rho}\>]$.
\end{Pp}

\subsubsection{} Pick $\nu\in\Lambda^{\sharp,+}$. Consider a version of the basic diagram from (\cite{BG}, Section~3.1.1). Set $\bar Z=\ov{\cH}^{\nu}_{\tilde G}\times_{\Bunt_G}\Bunb_{\tilde B}$, where we used the map $\tilde h^{\ra}_G: \ov{\cH}^{\nu}_{\tilde G}\to \Bunt_G$ to define the fibred product. 

\begin{Lm} 
\label{Lm_def_of_phi}
There is a morphism of stacks $\phi: \bar Z\to \Bunb_{\tilde B}\times X$ that fits into a commutative diagram
\begin{equation}
\label{diag_basic_diagram}
\begin{array}{ccccc}
\Bunb_{\tilde B}\times X & \getsuplong{\phi} & \bar Z & \touplong{'h^{\ra}_G} & \Bunb_{\tilde B}\\
\downarrow\lefteqn{\scriptstyle \tilde\gp\times \id} && \downarrow\lefteqn{\scriptstyle '\tilde\gp} && \downarrow\lefteqn{\scriptstyle \tilde\gp}\\
\Bunt_G\times X & \getsuplong{\tilde h^{\la}_G\times\pi} &  \ov{\cH}^{\nu}_{\tilde G} & \touplong{\tilde h^{\ra}_G} & \Bunt_G
\end{array}
\end{equation}
\end{Lm}
\begin{Prf}
A point of $\bar Z$ is given by 
\begin{equation}
\label{point_of_barZ}
(\cF'_T, \cF', \nu', \cU'^N\,\iso\, (\cL_T)_{\cF'_T}, \cU'^N_G\,\iso\, \cL_{\cF'})\in \Bunb_{\tilde B}
\end{equation} 
and $(x, \cF, \cF',\beta, \cU'^N_G\,\iso\, \cL_{\cF'}, \cU_G^N\,\iso\, \cL_{\cF})\in \ov{\cH}^{\nu}_{\tilde G}$. For this point we let $\cF_T=\cF'_T(w_0(\nu)x)$ with the system of induced inclusions 
$$
\nu^{\check{\lambda}}: \cL^{\check{\lambda}}_{\cF'_T}(\<w_0(\nu), \check{\lambda}\>x)\hook{} \cV^{\check{\lambda}}_{\cF_G}
$$
for all $\check{\lambda}\in \check{\Lambda}^+$. The map $\phi$ sends the above point to 
$$
(\cF_T, \cF, \nu, \cU^N_G\,\iso\, \cL_{\cF}, \cU^N\,\iso\, (\cL_T)_{\cF_T}),
$$ 
where $(\cF_T, \cU, \cU^N\,\iso\, (\cL_T)_{\cF_T})$ is the image of $(\cF'_T, \cU'^N\,\iso\, (\cL_T)_{\cF'_T}, x)$ under $m^{-w_0(\nu)}$. 
\end{Prf}

\medskip

 Set 
$$
\IC(\bar Z)_{\zeta}=('\tilde\gp)^*\IC^{\nu}\otimes ('h_G^{\ra})^*\IC_{\zeta}[-\dim\Bun_G]
$$ 
Since $\tilde h^{\ra}_G$ in (\ref{diag_basic_diagram}) is a locally trivial fibration in smooth topology, $\IC(\bar Z)_{\zeta}$ is an irreducible perverse sheaf on $\bar Z$. For a point (\ref{point_of_barZ}) let 
$$
(a, a',b')\in \mu_N(k)\times\mu_N(k)\times\mu_N(k)\subset \Aut(\cU_G)\times\Aut(\cU'_G)\times\Aut(\cU')
$$ 
acting trivially on $(\cF'_T,\cF', \cF)$. This 2-automorphism acts on $\IC(\bar Z)_{\zeta}$ as $\zeta(\frac{a}{b'})$. 

 For each $\mu\in\Lambda^{pos}_G$ one has the closed embedding $i_{\mu}: \Bunb_B\times X\hook{} \Bunb_B\times X$ defined in (\cite{BG}, Section~3.1.3). 
For $\mu\in \Lambda^{pos}_G\cap\Lambda^{\sharp}$ we lift it to a map
$$
\tilde i_{\mu}: \Bunb_{\tilde B}\times X\hook{} \Bunb_{\tilde  B}\times X
$$
sending $(x, \nu, \cF, \cF_T, \cU^N\,\iso\, (\cL_T)_{\cF_T}, \cU_G^N\,\iso\, \cL_{\cF})$ to $(x, \cF, \cF_T(-\mu x), \bar \cU, \cU_G)$, where $(\cF_T(-\mu x), \bar \cU)=m^{\mu}(\cF_T, \cU)$ is equipped with the induced inclusions
$$
\cL^{\check{\lambda}}_{\cF_T(-\mu x)}\hook{} \cV^{\check{\lambda}}_{\cF}
$$
Set for brevity $_{\mu}\IC_{\zeta}=\tilde i_{\mu !}(\IC_{\zeta}\boxtimes \IC(X))$. 
 
\begin{Pp} 
\label{Pp_direct_image_by_phi}
One has canonically 
$$
\phi_!\IC(\bar Z)_{\zeta}\,\iso\, \mathop{\oplus}\limits_{\mu\in \Lambda^{pos}_G\cap\Lambda^{\sharp}}\;\, {_{\mu}\IC_{\zeta}}\otimes V^{\nu}(\mu+w_0(\nu))
$$
\end{Pp}

\subsubsection{Proof of Theorem~\ref{Th_1}} 
\begin{Lm} 
\label{Lm_for_proof_of_Th1}
1) The maps $\tilde\gq('h^{\ra}_G)$ and $m^{w_0(\nu)}(\tilde\gq\times\id)\phi$ from $\bar Z$ to $\Bunt_T$ coincide.\\
2) For any $\mu\in \Lambda^{\sharp}\cap\Lambda_G^{pos}$ the diagram is canonically 2-commutative
$$
\begin{array}{ccc}
\Bunb_{\tilde B}\times X & \toup{\tilde\gq\times\id} & \Bunt_T\times X\\
\downarrow\lefteqn{\scriptstyle \tilde i_{\mu}} && \downarrow\lefteqn{\scriptstyle m^{\mu}\times\id}\\
\Bunb_{\tilde B}\times X & \toup{\tilde\gq\times\id} & \Bunt_T\times X
\end{array}
$$
\end{Lm}

Let $K\in \D_{\zeta}(\Bunt_T)$. The complex
$(\tilde h^{\ra}_G)^*\Eis(K)\otimes \IC^{\nu}$ over $\ov{\cH}^{\nu}_{\tilde G}$ identifies with
\begin{multline*}
\IC^{\nu}\otimes ('\tilde\gp)_!(('h^{\ra}_G)^*\tilde\gq^*K\otimes('h^{\ra}_G)^*\IC_{\zeta})[-\dim\Bun_T]\,\iso\, \\
('\tilde\gp)_!(('h^{\ra}_G)^*\tilde\gq^*K\otimes \IC(\bar Z)_{\zeta})[\dim\Bun_G-\dim\Bun_T]
\end{multline*}
So, 
\begin{multline}
\label{complex_to_calculate_proof_Th1}
\H^{\nu}_G\Eis(K)\,\iso\, (\tilde\gp\times\id)_!\phi_!(('h^{\ra}_G)^*\tilde\gq^*K\otimes \IC(\bar Z)_{\zeta})[-\dim\Bun_T]\,\iso\\ 
(\tilde\gp\times\id)_!((\tilde\gq\times\id)^*(m^{w_0(\nu)})^*K\otimes\phi_!\IC(\bar Z)_{\zeta})[-\dim\Bun_T]
\end{multline}
By Proposition~\ref{Pp_direct_image_by_phi}, this identifies with the direct sum over $\mu\in \Lambda^{pos}_G\cap\Lambda^{\sharp}$ of
\begin{multline*}
(\tilde\gp\times\id)_!((\tilde\gq\times\id)^*(m^{w_0(\nu)})^*K\otimes \tilde i_{\mu !}(\IC_{\zeta}\boxtimes\IC(X)))\otimes V^{\nu}(\mu+w_0(\nu))[-\dim\Bun_T]\,\iso\\
(\tilde\gp\times\id)_!\tilde i_{\mu !}(\tilde i_{\mu}^*(\tilde\gq\times\id)^*(m^{w_0(\nu)})^*K\otimes (\IC_{\zeta}\boxtimes\IC(X)))\otimes V^{\nu}(\mu+w_0(\nu))[-\dim\Bun_T]
\end{multline*}

 By Lemma~\ref{Lm_for_proof_of_Th1}, we get
$$
\tilde i_{\mu}^*(\tilde\gq\times\id)^*(m^{w_0(\nu)})^*K\,\iso\, (\tilde\gq\times\id)^*(m^{\mu}\times\id)^*(m^{w_0(\nu)})^*K\,\iso\,  (\tilde\gq\times\id)^*\H^{\mu+w_0(\nu)}_T(K)[-1],
$$
because $m^{w_0(\nu)}(m^{\mu}\times\id)=m^{\mu+w_0(\nu)}$. So, (\ref{complex_to_calculate_proof_Th1}) identifies with the sum over $\mu\in \Lambda^{pos}_G\cap\Lambda^{\sharp}$ of
\begin{multline*}
(\tilde\gp\times\id)_!((\tilde\gq\times\id)^*\H^{\mu+w_0(\nu)}_T(K)\otimes (\IC_{\zeta}\boxtimes\Qlb))\otimes V^{\nu}(\mu+w_0(\nu))[-\dim\Bun_T]\,
\iso\\ 
(\Eis\boxtimes\id)\H^{\mu+w_0(\nu)}_T(K)\otimes V^{\nu}(\mu+w_0(\nu))
\end{multline*}
Indeed, by definition, 
$$
(\Eis\boxtimes\id)(\cS)=(\tilde\gp\times\id)_!((\tilde\gq\times\id)^*\cS\otimes (\IC_{\zeta}\boxtimes\Qlb))[-\dim\Bun_T]
$$ 
Theorem~\ref{Th_1} is reduced to Proposition~\ref{Pp_direct_image_by_phi}.

\subsubsection{Proof of Proposition~\ref{Pp_direct_image_by_phi}} As in (\cite{BG}, Section~3.3.1), we fix $x\in X$, let $_x\bar Z$ be obtained from $\bar Z$ by the base change $\Spec k\toup{x} X$. We make this base change in the basic diagram and get
\begin{equation}
\label{diag_basic_diagram_x}
\begin{array}{ccccc}
\Bunb_{\tilde B} & \getsuplong{\phi} & _x\bar Z & \touplong{'h^{\ra}_G} & \Bunb_{\tilde B}\\
\downarrow\lefteqn{\scriptstyle \tilde\gp} && \downarrow\lefteqn{\scriptstyle '\tilde\gp} && \downarrow\lefteqn{\scriptstyle \tilde\gp}\\
\Bunt_G & \getsuplong{\tilde h^{\la}_G} &  _x\ov{\cH}^{\nu}_{\tilde G} & \touplong{\tilde h^{\ra}_G} & \Bunt_G
\end{array}
\end{equation}
Let $\IC(_x\bar Z)_{\zeta}=\IC(\bar Z)_{\zeta}\mid_{_x\bar Z}[-1]$. We will prove a version of Proposition~\ref{Pp_direct_image_by_phi} with $x$ fixed.
The maps are denoted by the same letters as for $x$ varying.

For $\mu\in\Lambda_G^{pos}$ we have the stacks $_{x, \ge\mu}\Bunb_B$, $_{x, \mu}\Bunb_B$, $_{x, \mu}\Bun_B$ defined as in (\cite{BG}, Section~3.3.2). We write 
$$
_{x, \ge\mu}\Bunb_{\tilde B}, \;\;{_{x, \mu}\Bunb_{\tilde B}}
$$
and so on for the stacks obtained by the base change $\Bunb_{\tilde B}\to \Bunb_B$ from the previous ones. Let $j_{\mu}: {_{x, \mu}\Bunb_{\tilde B}}\hook{} \Bunb_{\tilde B}$ be the natural embedding.

\begin{Pp} 
\label{Pp_3}
The complex $j_{\mu}^*\phi_!\IC(_x\bar Z)_{\zeta}$ satisfies the following.\\
a) It lives in perverse degrees $\le 0$. \\
b) The $0$-th perverse cohomology sheaf of $j_{\mu}^*\phi_!\IC(_x\bar Z)_{\zeta}$ identifies with 
$$
(_{\mu}\IC_{\zeta})\otimes V^{\nu}(\mu+w_0(\nu))
$$
\end{Pp}

For $\mu\in\Lambda_G^{pos}$ let $Z^{?, \mu}$ (resp., $Z^{\mu, ?}$) denote the preimage in $_x\bar Z$ of $_{x,\mu}\Bunb_{\tilde B}$ under $'h^{\ra}_G$ (resp., $\phi$). For $\mu,\mu'\in \Lambda_G^{pos}$ let $Z^{\mu,\mu'}=Z^{\mu, ?}\cap Z^{?,\mu'}$. Recall that $Z^{\mu,\mu'}$ is empty unless $\mu\ge \mu'$. 
 
  If $\nu'\in \Lambda^+$ with $\nu'\le\nu$ then we write 
$$
Z^{\mu,\mu', \nu'}=Z^{\mu,\mu'}\cap {'\tilde\gp^{-1}}(_x\cH^{\nu'}_{\tilde G})
$$

\begin{Pp} 
\label{Pp_K_three_superscripts}
For $\mu,\mu'\in\Lambda_G^{pos}$ and $\nu'\in\Lambda^+$ with $\nu'\le\nu$  let $K^{\mu,\mu',\nu'}\in \D(_{x,\mu}\Bunb_{\tilde B})$ be defined as
$$
\phi_!(\IC(_x\bar Z)_{\zeta}\mid_{Z^{\mu,\mu', \nu'}})
$$ 
a) The complex $K^{\mu,\mu',\nu'}$ is placed in perverse degrees $\le 0$, the equality is strict unless $\mu'=0$ and $\nu'=\nu$.\\
b) The $*$-restriction of $K^{\mu, 0, \nu}$ to $_{x,\mu}\Bunb_{\tilde B}-{_{x,\mu}\Bun_{\tilde B}}$ is placed in perverse degrees $<0$. \\
c) The $0$-th perverse cohomology sheaf of $K^{\mu, 0, \nu}$ over $_{x,\mu}\Bun_{\tilde B}$ vanishes unless $\mu\in\Lambda^{\sharp}\cap\Lambda_G^{pos}$. In the latter case it identifies with 
$$
(_{\mu}\IC_{\zeta})\mid_{_{x,\mu}\Bun_{\tilde B}}\otimes V^{\nu}(\mu+w_0(\nu))
$$
\end{Pp}

 A version of (\cite{BG}, Lemma~3.3.6) holds with obvious changes. For $\nu\in\Lambda^+$ write $\ov{\Gr}^{\nu}_{\tilde G}=\ov{\Gr}^{\nu}_G\times_{\Gr_G}\wt\Gr_G$. For $\mu\in\Lambda$ the scheme $S^{\mu}_G$ is defined in (\cite{BG}, Section~3.2.5). For $\mu\in\Lambda$ write 
$$
S^{\mu}_{\tilde G}=S^{\mu}_G\times_{\Gr_G}\wt\Gr_G
$$  
\begin{Lm} 
\label{Lm_some_fibrations_description}
(a) The map $'h^{\ra}_G: Z^{?, \mu'}\to {_{x,\mu'}\Bunb_{\tilde B}}$ is a locally trivial fibration with typical fibre $\ov{\Gr}^{-w_0(\nu)}_{\tilde G}$. \\
(b) The morphism $'h^{\ra}_G: Z^{\mu,\mu',\nu'}\to {_{x,\mu'}\Bunb_{\tilde B}}$ identifies using the notations of (a) with a subfibration with typical fibre 
$$
\Gr_{\tilde G}^{-w_0(\nu')}\cap S_{\tilde G}^{-w_0(\nu)-\mu+\mu'}
$$
(c) The map $\phi: Z^{\mu,\mu',\nu'}\to {_{x,\mu}\Bunb_{\tilde B}}$ is a locally trivial fibration with fibre 
$$
\Gr_{\tilde G}^{\nu'}\cap S_{\tilde G}^{\mu-\mu'+w_0(\nu)}
$$
\end{Lm} 
 
\begin{Prf}\select{of Proposition~\ref{Pp_K_three_superscripts}} The complex
$\IC(_x\bar Z)_{\zeta}\mid_{Z^{?,\mu'}}$ is a twisted external product
$$
\cA^{-w_0(\nu)}_{\cE}\tboxtimes \IC_{\zeta}\mid_{_{x,\mu'}\Bunb_{\tilde B}}
$$
with the notations of Lemma~\ref{Lm_some_fibrations_description}(a). So, $\IC(_x\bar Z)_{\zeta}\mid_{Z^{\mu,\mu', \nu'}}$ is a twisted external product
$$
\cA^{-w_0(\nu)}_{\cE}\mid_{\Gr_{\tilde G}^{-w_0(\nu')}\cap S_{\tilde G}^{-w_0(\nu)-\mu+\mu'}}\tboxtimes \IC_{\zeta}\mid_{_{x,\mu'}\Bunb_{\tilde B}}
$$
The complex $\IC_{\zeta}\mid_{_{x,\mu'}\Bunb_{\tilde B}}$ is placed in perverse degrees $\le 0$, and the inequality is strict unless $\mu'=0$. 
Further, 
$$
\cA^{-w_0(\nu)}_{\cE}\mid_{\Gr_{\tilde G}^{-w_0(\nu')}}
$$ 
is a constant complex placed in perverse degrees $<0$ unless $\nu'=\nu$. Now exactly as in (\cite{BG}, Section~3.3.7) one proves a). This only uses the following. For a morphism $f: Y_1\to Y_2$ such that the maximal dimension of the fibres of $f$ is $\le d$ and a perverse sheaf $F$ on $Y_1$ the complex $f_!F$ is placed in perverse degrees $\le d$. The point b) is proved similarly (as in \cite{BG}, Section~3.3.7). 

\smallskip\noindent
c) Let $^0Z^{\mu,0,\nu}$ be the preimage of $\Bun_{\tilde B}\subset {_{x,0}\Bunb_{\tilde B}}$ under $'h^{\ra}_G$. One has 
\begin{equation}
\label{complex_K_mu_0_nu}
K^{\mu,0,\nu}\mid_{_{x,\mu}\Bun_{\tilde B}}=\phi_!(\IC(_x\bar Z)_{\zeta}\mid_{^0Z^{\mu,0,\nu}})
\end{equation}
A point of $_{x,\mu}\Bun_{\tilde B}$ is given by $(\cF_T,\cF, \cU^N\,\iso\, (\cL_T)_{\cF_T}, \cU_G^N\,\iso\, \cL_{\cF})$. Let 
$$
(a,b)\in \mu_N(k)\times\mu_N(k)\subset \Aut(\cU_G)\times\Aut(\cU)
$$ 
acting trivially on $\cF_T,\cF$ then $(a,b)$ acts on (\ref{complex_K_mu_0_nu}) as $\zeta(\frac{a}{b})$. One shows as in (\cite{L}, Lemma~11) that any bounded complex in $\D(_{x,\mu}\Bun_{\tilde B})$, on which $(a,b)$ acts as $\zeta(\frac{a}{b})$, vanishes unless $\mu\in\Lambda^{\sharp}$. Assume $\mu\in\Lambda^{\sharp}$.

 The complex $\IC(_x\bar Z)_{\zeta}\mid_{^0Z^{\mu,0,\nu}}$ is the twisted external product with the notations of Lemma~\ref{Lm_some_fibrations_description}(a)
\begin{equation}
\label{complex_about_mu_0_nu}
\cA_{\cE}^{-w_0(\nu)}\mid_{\Gr_{\tilde G}^{-w_0(\nu)}\cap S_{\tilde G}^{-w_0(\nu)-\mu}}\tboxtimes \IC_{\zeta}\mid_{\Bun_{\tilde B}}
\end{equation}
and  $\cA_{\cE}^{-w_0(\nu)}\mid_{\Gr_{\tilde G}^{-w_0(\nu)}}\,\iso\, E_{\cE}^{-w_0(\nu)}[\<\nu, 2\check{\rho}\>]$. 

 Now in the notation of Lemma~\ref{Lm_some_fibrations_description}(c), the map $\phi: {^0Z^{\mu,0,\nu}}\to {_{x,\mu}\Bun_{\tilde B}}$ and (\ref{complex_about_mu_0_nu}) becomes a locally trivial fibration 
$$
(\Gr_{\tilde G}^{\nu}\cap S_{\tilde G}^{\mu+w_0(\nu)})\ttimes {_{x,\mu}\Bun_{\tilde B}}\to {_{x,\mu}\Bun_{\tilde B}}
$$ 
and the complex $(E_{\cE}^{\nu})^*\tboxtimes (_{\mu}\IC_{\zeta}\mid_{_{x,\mu}\Bun_{\tilde B}})[2\<\mu, \check{\rho}\>]$ on the source.

 Applying now Proposition~\ref{Pp_description_weight_space} (with the character $\zeta$ replaced by $\zeta^{-1}$), we see that the $0$-th perverse cohomology sheaf of $K^{\mu,0,\nu}\mid_{_{x,\mu}\Bun_{\tilde B}}$ identifies with $(_{\mu}\IC_{\zeta})\mid_{_{x,\mu}\Bun_{\tilde B}}\otimes V^{\nu}(\mu+w_0(\nu))$. 
\end{Prf}

\medskip

 So, Proposition~\ref{Pp_3} is also proved. This concludes the proof of Proposition~\ref{Pp_direct_image_by_phi} (and hence of Theorem~\ref{Th_1}).

\subsection{Proof of Theorem~\ref{Th_Eis_commutes_with_Verdier}}  An analog of (\cite{BG}, Theorem~5.1.5) holds in our situation.

\begin{Pp} 
\label{Pp_ULA_property}
The perverse sheaf $\IC_{\zeta}$ is ULA with respect to $\tilde\gq: \Bunb_{\tilde B}\to \Bunt_T$.
\end{Pp}
\begin{Prf} 
Pick $\lambda_1,\ldots, \lambda_r\in \Lambda^{\sharp,+}$ such that they form a base in $\Lambda^{\sharp}\otimes_{\ZZ}\QQ$. Let $m\ge 2g-1$, let $\vartriangle\hook{}X^{mr}$ denote the divisor of diagonals. 

 Let $\cH^?_{\tilde G}$ be the stack classifying 
$$
\{x_{1,1},\ldots, x_{1,m}, x_{2,1},\ldots, x_{r,1},\ldots, x_{r,m}\}\in X^{mr}-\vartriangle, \;(\cF,\cU_G), (\cF',\cU'_G)\in\Bunt_G
$$ 
an isomorphism $\beta: \cF\,\iso\, \cF'\mid_{X-\{x_{1,1},\ldots, x_{r,m}\}}$ such that for all $i,j$, $\cF$ is in the position $\lambda_i$ with respect to $\cF'$ at $x_{i,j}$.

  Let $h^{\la}_G, h^{\ra}_G: \cH^?_{\tilde G}\to \Bunt_G$ denote the projections sending the above point to $(\cF,\cU_G)$ and $(\cF',\cU'_G)$ respectively. Let $\pi: \cH^?_{\tilde G}\to X^{mr}-\vartriangle$ be the projection. 
  
  Let $\bar Z=\cH^?_{\tilde G}\times_{\Bunt_G} \Bunb_{\tilde B}$, where we used $h^{\ra}_G$ to define the fibred product. Let $\phi: \bar Z\to \Bunb_{\tilde B}\times (X^{mr}-\vartriangle)$ be the map defined as in Lemma~\ref{Lm_def_of_phi}, we get a commutative diagram
$$
\begin{array}{ccccc}
\Bunb_{\tilde B}\times (X^{mr}-\vartriangle) &\getsuplong{\phi} & \bar Z & \touplong{'h^{\ra}_G} &\Bunb_{\tilde B}\\
\downarrow\lefteqn{\scriptstyle \tilde\gp\times\id} && \downarrow && \downarrow\lefteqn{\scriptstyle \tilde\gp}\\
\Bunt_G\times (X^{mr}-\vartriangle) &\getsuplong{h^{\la}_G\times\pi} & \cH^?_{\tilde G} & \touplong{h^{\ra}_G} & \Bunt_G
\end{array}
$$
Let $AJ: X^{mr}-\vartriangle\to \Bun_{T^{\sharp}}$ be the map sending $(x_{i,j})$ to 
$$
\cF^0_{T^{\sharp}}(\sum_{i,j} \lambda_i x_{i,j}),
$$  
it is smooth. The composition $\bar Z\toup{'h^{\ra}_G} \Bunb_{\tilde B}\toup{\tilde\gq}\Bunt_T$ equals the composition
$$
\bar Z\;\toup{\phi}\;\Bunb_{\tilde B}\times (X^{mr}-\vartriangle) \;\touplong{\id\times AJ}\; \Bunb_{\tilde B}\times \Bun_{T^{\sharp}}\touplong{m_{\tilde\gq}} \Bunt_T
$$
Here $m_{\tilde\gq}$ denotes the composition $\Bunb_{\tilde B}\times \Bun_{T^{\sharp}}\toup{\tilde\gq\times\id} \Bunt_T\times \Bun_{T^{\sharp}}\toup{a} \Bunt_T$, where $a$ is the action map defined in (\cite{L}, Section~5.2.3). 

 Define $Z\subset\bar Z$ as the open substack classifying $(x_{i,j}, \cF, \cU_G, \cF',\cU'_G, \cF'_T, \cU', \nu')\in\bar Z$ such that for all $\check{\lambda}\in\check{\Lambda}^+$ the following holds:
\begin{itemize} 
\item[a)] The map $\nu'^{\check{\lambda}}: \cL^{\check{\lambda}}_{\cF'_T}\hook{}\cV^{\check{\lambda}}_{\cF'_G}$ has no zero at $x_{i,j}$.
\item[b)] The map $\nu^{\check{\lambda}}: \cL^{\check{\lambda}}_{\cF'_T}(-\sum_{i,j}\<\lambda_i, \check{\lambda}\>x_{i,j})\hook{} \cV^{\check{\lambda}}_{\cF_G}$
has no zero at each $x_{i,j}$.
\end{itemize}
As in (\cite{BG}, Theorem~5.1.5), the map $'h^{\ra}_G: Z\to\Bunb_{\tilde B}$ is smooth and surjective, $\phi: Z\to \Bunb_{\tilde B}\times (X^{mr}-\vartriangle)$ is smooth. We get a diagram
$$
\begin{array}{ccc}
Z & \touplong{'h^{\ra}_G} & \Bunb_{\tilde B}\\
\downarrow\lefteqn{\scriptstyle  b} && \downarrow\lefteqn{\scriptstyle \tilde\gq}\\
\Bunb_{\tilde B}\times \Bun_{T^{\sharp}} & \touplong{m_{\tilde\gq}} &\Bunt_T,
\end{array}
$$
where $b=(\id\times AJ)\comp\phi$, so $b$ is smooth. It is easy to see that there is a rank one local system $\cE$ on $Z$ with $\cE^N\,\iso\, \Qlb$ and an isomorphism 
$$
('h^{\ra}_G)^*\IC_{\zeta}\,\iso\, \cE\otimes b^*(\IC_{\zeta}\boxtimes \IC(\Bun_{T^{\sharp}}))
$$ 
So, it suffices to show that $\IC_{\zeta}\boxtimes \IC(\Bun_{T^{\sharp}})$ is ULA with respect to $m_{\tilde\gq}$. Note that $m_{\tilde\gq}$ is the composition
$$
\Bunb_{\tilde B}\times \Bun_{T^{\sharp}}\toup{\delta} \Bunb_{\tilde B}\times\Bunt_T \toup{\pr_2}\Bunt_T,
$$ 
where $\delta$ composed with the projection to $\Bunb_{\tilde B}$ (resp., to $\Bunt_T$) is $\pr_1$ (resp., $m_{\tilde\gq}$). Since $\Bunt_T$ is smooth $\IC_{\zeta}\boxtimes \IC(\Bunt_T)$ is ULA over $\Bunt_T$. Since $\delta$ is smooth, our claim follows from (\cite{BG}, 5.1.2 (2)). See Remark~\ref{Rem_ULA} below.
\end{Prf}   

\begin{Rem} 
\label{Rem_ULA}
Let $f: Y\to Z$ be a morphism of schemes with $Z$ smooth. Let $H$ be a smooth group scheme acting on $Z$, assume the stabilizor 
in $H$ of any point of $Z$ is smooth. Assume that for any $k$-point $z\in Z$ the map $H\to Z, h\mapsto hz$ is smooth. Let $K\in \D(Y)$. Let $m_f$ denote the composition $H\times Y\toup{\id\times f} H\times Z\toup{a}Z$, where $a$ is the action map. Then $\Qlb\boxtimes K$ is ULA with respect to $m_f$. Indeed, $m_f$ is written as the composition $H\times Y\toup{\delta} Z\times Y\toup{\pr_1} Z$, where $\delta$ composed with the projection to $Y$ is $\pr_2$. The map $\delta$ is smooth, because it is obtained by base change from the map $H\times Z\to Z\times Z, (h,z)\mapsto (hz, z)$.
\end{Rem}

 Theorem~\ref{Th_Eis_commutes_with_Verdier} follows from Proposition~\ref{Pp_ULA_property} by applying (\cite{BG}, Section~5.1.2). 
 
\subsection{Some gradings} 
\label{section_Some_gradings}

The center $Z(G)$ acts on $\Bun_G$ by 2-automorphisms, namely $z\in Z(G)$ yields an automorphism $\cF\to\cF$, $f\mapsto fz$ of $\cF\in \Bun_G$. This automorphism acts trivially on $\gg_{\cF}$. We let $Z(G)$ act on $\Bunt_G$ by  2-automorphisms so that for $(\cF, \cU)\in\Bunt_G$, $z\in Z(G)$ acts naturally on $\cF$ and trivially on $\cU$. For a character $\chi: Z(G)\to\Qlb^*$ we get the full triangulated subcategory 
$\D_{\zeta, \chi}(\Bunt_G)\subset \D_{\zeta}(\Bunt_G)$ of objects on which $Z(G)$ acts by $\chi$. 

 Write $C^*(\check{G}_n)$ for the cocenter of $\check{G}_n$, the quotient of $\Lambda^{\sharp}$ by the roots lattice of $\check{G}_n$. One has canonically $\Hom(Z(\check{G}_n), \Qlb^*)\,\iso\, C^*(\check{G}_n)$. If $\mu\in\Lambda^{\sharp,+}$ then $Z(\check{G}_n)$ acts on $V^{\mu}$ by the character, which is the image of $\mu$ in $C^*(\check{G}_n)$.
 
  There is a natural map $\xi: C^*(\check{G}_n)\to \Hom(Z(G), \mu_n(k))$ sending $\nu\in\Lambda^{\sharp}$ to the character $\frac{\iota(\nu)}{n}$. The latter sends $z\in Z(G)$ to $\frac{\iota(\nu)}{n}(z)$. If $\alpha$ is a simple coroot of $G$ then let $\delta$ denote the denominator of $\frac{\iota(\alpha,\alpha)}{2n}$. Recall that $\delta\alpha$ is the corresponding simple root of $\check{G}_n$ (\cite{FL}, Theorem~2.9). Since $\frac{\iota(\delta\alpha)}{n}$ lies in the roots lattice of $G$, $\frac{\iota(\delta\alpha)}{n}(z)=1$ for $z\in Z(G)$ by Remark~\ref{Rem_about_roots_lattice} below. Thus, $\xi$ is correctly defined.
Write $C^*(G)_n$ for the $n$-torsion subgroup of $C^*(G)$. 
  
\begin{Lm} The map $C^*(\check{G}_n)\to C^*(G)$ sending $\nu\in \Lambda^{\sharp}$ to $\frac{\iota(\nu)}{n}$ is injective and takes values in $C^*(G)_n$. The above map $C^*(\check{G}_n)\to C^*(G)_n$ is not always surjective.\footnote{The only case when it is not surjective is indicated in part 3) of the proof.}
\end{Lm}
\begin{Prf}
Let us check this case by case for all simple simply-connected groups.\\
1) If $G$ is simply-laced then for each simple coroot $\alpha$ of $G$ the corresponding root of $\check{G}_n$ is $n\alpha$. Besides, $\iota(\alpha)=\check{\alpha}$ for each simple coroot $\alpha$ of $G$. So, if $\lambda\in\Lambda^{\sharp}$ and $\iota(\lambda)$ lies in $n\check{Q}$, where $\check{Q}$ is the roots lattice of $G$ then $\lambda\in n\Lambda$. Our claim follows in this case. For $G$ simply-laced the map $C^*(\check{G}_n)\to C^*(G)$ identifies $C^*(\check{G}_n)$ with the $n$-torsion subgroup in $C^*(G)$. 

\smallskip\noindent
2) If $G=\Sp_{2m}$ then the nontrivial case is $n$ even. In this case let $\alpha_i$ be the standard simple coroots, so $\iota(\alpha_i)=2\check{\alpha}_i$ for $i<m$ and $\iota(\alpha_m)=\check{\alpha}_m$. In this case $\Lambda^{\sharp}=\frac{n}{2}\Lambda$, the simple roots of $\check{G}_n$ are $\{\frac{n}{2}\alpha_i, i<m$ and $n\alpha_m\}$. So, $\frac{n}{2}\alpha_m\in C^*(\check{G}_n)\,\iso\, \ZZ/2\ZZ$ is a generator. Since $\frac{\check{\alpha}_n}{2}$ is not in the roots lattice of $G$, our claim follows, the map under consideration is actually an isomorphism.

\smallskip\noindent
3) If $G=\Spin_{2m+1}$ with $m\ge 2$ then the only nontrivial case is $n$ even and $nm/2$ even. In this case $C^*(\check{G}_n)\,\iso\, \ZZ/2\ZZ$. Let $\alpha_i$ denote the standard simple coroots. Then $\iota(\alpha_i)=\check{\alpha}_i$ for $i<m$, and $\iota(\alpha_m)=2\check{\alpha}_m$. The simple roots of $\check{G}_n$ are $n\alpha_i$, $i<m$ and $\frac{n}{2}\alpha_m$. The roots lattice of $\check{G}_n$ is $n\Lambda$, where $\Lambda=\{(a_1,\ldots, a_m)\mid \sum a_i=0\mod 2\}$. So, $(\frac{n}{2}, \ldots, \frac{n}{2})\in\Lambda^{\sharp}$ generates $\Lambda^{\sharp}/n\Lambda$. The roots lattice of $G$ is $\ZZ^n$, and $\frac{\iota(\nu)}{n}$ sends the above generator to $(\frac{1}{2}, \ldots, \frac{1}{2})$, which is not in $\ZZ^n$. The map under consideration is an isomorphism in this case.

 However, if $n$ is even and $nm/2$ is odd then the map under consideration $C^*(\check{G}_n)\to C^*(G)$ is not surjective!
 
\smallskip\noindent
4) For $G_2$ and $F_4$ the claim is trivial, as the center is trivial.
\end{Prf} 

\begin{Pp} 
\label{Lm_great_about_gradings}
Let $\nu\in \Lambda^{\sharp,+}$ and $K\in \D_{\zeta}(\Bunt_G)$. Assume that $Z(G)$ acts on $K$ (by functoriality from the above 2-action on $\Bunt_G$) by a character $\chi: Z(G)\to \Qlb^*$. Then $z\in Z(G)$ acts on $\H^{\nu}_G(K)$ as $\chi(z)\zeta(\frac{i(\nu)}{n}(z))$. 
\end{Pp}
\begin{Prf}
Recall that a point of $\ov{\cH}^{\nu}_{\tilde G}$ is given by $(\cF,\cF',\beta: \cF\,\iso\, \cF'\mid_{X-x})\in \ov{\cH}^{\nu}_G$ and  $\cU,\cU'$. Let $Z(G)$ act on $\ov{\cH}^{\nu}_{\tilde G}$ by 2-automorphisms so that it acts naturally on $\cF,\cF'$ and trivially on $\cU,\cU'$. Let us show that $z\in Z(G)$ acts on $\IC^{\nu}$ as $\zeta(\frac{i(\nu)}{n}(z))$.

 Consider the open substack $\cH^{\nu}_{\tilde G}\subset \ov{\cH}^{\nu}_{\tilde G}$. There is a line bundle, say $\cB$, on $\cH^{\nu}_G$ such that $\cB^N$ is canonically the line bundle with fibre $\cL_{\cF}\otimes\cL_{\cF'}^{-1}$ at $(\cF,\cF',\beta)\in \cH^{\nu}_G$. The line bundle is uniquely defined, as the Picard group is torsion free. For $z\in Z(G)$ consider the 2-automorphism of $\cH^{\nu}_G$ acting as $z$ on $\cF,\cF'$. Then $z$ acts on $\cB$ as $\frac{i(\nu)}{n}(z^{-1})$.  Actually, $\frac{i(\nu)}{n}(z^{-1})\in \mu_n(k)$, because of Remark~\ref{Rem_about_roots_lattice} below. We have an isomorphism
$$
\eta: B(\mu_N)\times(\cH^{\nu}_G\times_{\Bun_G}\Bunt_G)\,\iso\, \cH^{\nu}_{\tilde G}
$$
where we used $h^{\ra}_G$ to define the fibred product in parentheses. It sends a collection $(\beta: \cF\,\iso\,\cF'\mid_{X-x}, \;\cU'^N\,\iso\, \cL_{\cF'}, \,\cU_0^N\,\iso\, k)$ to $(\beta, \cF,\cF', \cU,\cU')$, where $\cU=\cB_{(\cF,\cF',\beta)}\otimes \cU_0\otimes \cU'$ with the induced isomorphism $\cU^N\,\iso\, \cL_{\cF}$. The perverse sheaf $\IC^{\nu}$ is the intermediate extension of 
$$
\eta_*(\cL_{\zeta}\boxtimes \IC(\cH^{\nu}_G\times_{\Bun_G}\Bunt_G))
$$ 
So, our 2-automorphism $z\in Z(G)$ of $\ov{\cH}^{\nu}_{\tilde G}$ acts on $\IC^{\nu}$ as $\zeta(\frac{i(\nu)}{n}(z))$.
The group $Z(G)$ acts by the above 2-automorphisms on the diagram
$$
\Bunt_G\times X\getsup{\tilde h^{\la}_G\times\pi} \ov{\cH}^{\nu}_{\tilde G} \toup{\tilde h^{\ra}_G} \Bunt_G,
$$
our assertion follows. 
\end{Prf} 

\begin{Rem} 
\label{Rem_about_roots_lattice}
i) For any $\nu\in \Lambda$ the element $\iota(\nu)$ lies in the root lattice of $G$. So, if $z\in Z(G)$ then $\iota(\nu)(z)=1$.\\
ii) Recall that $G$ is simple, simply-connected. Write $Z(G)_n$ for the $n$-torsion subgroup of $Z(G)$. If $G$ is simply-laced then $\check{G}_n$ is isomorphic to the Langlands dual to $G/Z(G)_n$. This was also observed by Savin in \cite{Sa}. 
\end{Rem}

\subsubsection{} The group $T$ acts naturally on $\Bun_T$ by 2-automorphisms, under this action $t\in T$ acts on $\cL_T\mid_{\Bun_T^{\mu}}$ by the character $-\kappa(\mu)(t)$. In partucular, $Z(G)\subset T$ acts on $\Bun_T$ by 2-automorphisms, and acts trivially on $\cL_T$. We let $Z(G)$ act on $\Bunt_T$ by 2-automorphisms, so that $z\in Z$ acts on $(\cF_T,\cU)$ as $z$ on $\cF_T$ and trivially on $\cU$. We also let $Z(G)$ act on $\Bunb_{\tilde B}$ by 2-automorphisms so that $z\in Z$ acts on $(\cF_T,\cF, \cU^N\,\iso\, (\cL_T)_{\cF_T}, \cU_G^N\,\iso\, \cL_{\cF})$ as $z$ on $\cF_T, \cF_G$ and trivially on $\cU, \cU_G$. The diagram (\ref{diag_def_of_Eis}) is equivariant with respect to this 2-action. The group $Z(G)$ acts trivially on $\IC_{\zeta}$. So, given $\cS\in \D_{\zeta}(\Bunt_T)$, if $Z(G)$ acts on $\cS$ by a character $\chi: Z(G)\to\Qlb^*$ then $Z(G)$ acts on $\Eis(\cS)$ also by $\chi$. 

\begin{Lm} 
\label{Lm_2-action_for_Bunt_T}
Let $\nu\in \Lambda^{\sharp}, K\in \D_{\zeta}(\Bunt_T^{\nu})$. Then $z\in Z(G)$ acts on $K$ as $\zeta(\frac{\iota(\nu)}{n}(z^{-1}))$.
\end{Lm}
\begin{Prf}  
\Step 1 The group $T$ acts trivially on $\cL_T\mid_{\Bun_T^0}$. So, $T$ acts by 2-automorphisms of $\Bunt_T^0$, namely, $t\in T$ acts as $t$ on $\cF$ and trivially on $\cU$. For any $K\in \D_{\zeta}(\Bunt_T^0)$, $T$ acts trivially on $K$, so $Z(G)$ also acts trivially on $K$. 

\Step 2 Pick $x\in X$, let $_xm^{\nu}: \Bunt_T\to\Bunt_T$ be the restriction of $m^{\nu}$ to $x$. Then $_xm^{\nu}$ is an isomoprhism. From (\ref{cU'_for_def_of_mnu}) we see if $z\in Z(G)$ acts on some $\cS\in\D_{\zeta}(\Bunt_T)$ as $\chi(z)$ then $z$ acts on $(_xm^{\nu})^*\cS$ as $\chi(z)\zeta(\frac{\iota(\nu)}{n}(z))$.
\end{Prf}

\subsection{Towards the functional equation} 
\label{Section_Towards the functional equation}
The $W$-action on $\Lambda$ perserves $\Lambda^{\sharp}$, so $W$ acts on the left on $T^{\sharp}$ naturally. For $w\in W$ we denote also by $w: \Bun_{T^{\sharp}}\to \Bun_{T^{\sharp}}$, $\cF\mapsto {^w\cF}$ the extension of scalars map with respect to the left action map $w: T^{\sharp}\to T^{\sharp}$. We also let $W$ act on the left on $\check{\Lambda}$, so $w\check{\lambda}=\check{\lambda}\comp w^{-1}\in \check{\Lambda}$. For $\check{\lambda}\in\check{\Lambda}$, $\cF\in\Bun_T$ we get $\cL^{\check{\lambda}}_{(^w\cF)}\,\iso\, \cL^{w^{-1}\check{\lambda}}_{\cF}$ canonically. For $w_i\in W$, $\cF\in\Bun_{T^{\sharp}}$ one has $^{w_2}(^{w_1}\cF)\,\iso\, {^{w_2w_1}\cF}$ naturally. 
 
 We let also $W$ act on the left on $\check{T}^{\sharp}$. For a $\check{T}^{\sharp}$-local system $E$ on $X$ we denote by $^wE$ the extension of scalars of $E$ with respect to $w: \check{T}^{\sharp}\to \check{T}^{\sharp}$. So, for $\mu\in\Lambda^{\sharp}$ we get $(^wE)^{\mu}\,\iso\, (E)^{w^{-1}\mu}$ as local systems on $X$. For $w\in W$ and the map $w^{-1}: \Bun_{T^{\sharp}}\to \Bun_{T^{\sharp}}$ we then get $(w^{-1})^* AE\,\iso\, A(^wE)$ canonically. 
 
  Let $W$ act naturally on the left on $T$. For $w\in W$ write $^w\cF$ for the extension of scalars of $\cF\in \Bun_T$ under $w: T\to T$. Write $w: \Bun_T\to\Bun_T$ for the map $\cF\mapsto {^w\cF}$,  this is a left action of $W$ on $\Bun_T$. The line bundle $\cL_T$ is naturally $W$-equivariant. Write also $w: \Bunt_T\to \Bunt_T$ for the map $(\cF, \cU, \cU^N\,\iso\, (\cL_T)_{\cF})\mapsto (^w\cF, \cU)$ with the induced isomorphism $\cU^N\,\iso\, (\cL_T)_{^w\cF}$. This defines a left action of $W$ on $\Bunt_T$.  

 For a $\check{T}^{\sharp}$-local system $E$ on $X$ let $\cK_E\in \D_{\zeta}(\Bunt_T)$ be the $E$-Hecke eigensheaf constructed in (\cite{L}, Proposition~2.2), this is a local system over the components $\Bunt_T^{\mu}$, $\mu\in\Lambda^{\sharp}$. 
 
\begin{Pp} 
\label{Pp_action_of_W_on_cK_E}
For $w\in W$ there is an isomorphism $(w^{-1})^*\cK_E\,\iso\, \cK_{(^wE)}$. 
\end{Pp}
\begin{Prf}
In this proof we use the notations of (\cite{L}, Section~5.2.4). Recall that $K$ is the kernel of the natural map $T^{\sharp}\to T$, so $K\,\iso\, (\Lambda/\Lambda^{\sharp})\otimes\mu_n$. The group $\H^1(X, K)\,\iso\, \H^1(X, \mu_n)\otimes (\Lambda/\Lambda^{\sharp})$ is equipped with the skew-symmetric non-degenerate pairing 
$$
(\cdot, \cdot)_c: \H^1(X, K)\times \H^1(X, K)\to \mu_n(k)
$$ 
described in (\cite{L}, Proposition~5.1). Let $H_0\subset \H^1(X, \mu_n)$ be a maximal isotropic subgroup with respect to the natural pairing $\H^1(X, \mu_n)\times \H^1(X, \mu_n)\to\H^2(X, \mu_n^{\otimes 2})\,\iso\,\mu_n$. Set $H=H_0\otimes (\Lambda/\Lambda^{\sharp})$. So, $H\subset \H^1(X, K)$ is a $W$-invariant maximal isotropic subgroup with respect to $(\cdot, \cdot)_c$.

 Recall the stacks $'\Bun_{T^{\sharp}}$, $\Bun_{T^{\sharp}, H}$ from (\cite{L}, Section~5.2.4). The group $W$ acts naturally on $'\Bun_{T^{\sharp}}$, $\Bun_T$, and $'i_X: {'\Bun_{T^{\sharp}}}\to \Bun_T$ is $W$-equivariant. The line bundle $'\tau$ on $'\Bun_{T^{\sharp}}$ is naturally $W$-equivariant. The $W$-action on $'\Bun_{T^{\sharp}}$ induces a $W$-action on $\Bun_{T^{\sharp}, H}$ so that the diagram $'\Bun_{T^{\sharp}}\to \Bun_{T^{\sharp}, H}\to \Bun_T$ is $W$-equivariant. The $W$-actions on $\Bun_{T^{\sharp}, H}$ and on $\Bun_T$ naturally extend to $W$-actions on $\Bunt_{T^{\sharp}, H}$ and $\Bunt_T$. The map $\pi_H: \Bunt_{T^{\sharp}, H}\to\Bunt_T$ is $W$-equivariant. 
 
 Pick a local system $AE_H$ on $\Bunt_{T^{\sharp}, H}$, whose restriction to $'\Bun_{T^{\sharp}}$ is identified with $AE$, and such that $\mu_N(k)$ acts on it by $\zeta$. Recall that $\cK_E$ is defined as $\pi_{H !}(AE_H)$. Since the map $'\Bun_{T^{\sharp}}\to \Bunt_{T^{\sharp}, H}$ is $W$-equivariant, our claim follows.
\end{Prf} 

\subsubsection{}  Recall that $\rho_n$ denotes the half sum of positive roots of $\check{G}_n$. If $w\in W$ then $w(\rho_n)-\rho_n\in\Lambda^{\sharp}$. (For $w$ a simple reflection this is clear, the general case is obtained by induction on the length of the decomposition into simple reflections).  
  
  Define the twisted $W$-action on $\Bunt_T$ as follows. For $\mu\in \Lambda^{\sharp}$ write $\Omega^{\mu}$ for the $T^{\sharp}$-torsor induced from $\Omega$ via $\mu: \Gm\to T^{\sharp}$. By abuse of notation, the corresponding $T$-torsor is also denoted $\Omega^{\mu}$. Denote by $a: \Bun_{T^{\sharp}}\times\Bunt_T\to \Bunt_T$ the action of $\Bun_{T^{\sharp}}$ on $\Bunt_T$ defined in (\cite{L}, Section~5.2.3). 
 
For $(\cF,\cU, \cU^N\,\iso\,\cL_{\cF})\in \Bunt_T$ we set 
\begin{equation}
\label{def_twisted_W_action_on_Bunt_T}
w\ast (\cF,\cU)=a(\Omega^{w(\rho_n)-\rho_n}, (^w\cF, \cU))=(^w\cF\otimes \Omega^{w(\rho_n)-\rho_n}, \cU'),
\end{equation}
where $a$ is the above action map, $\cU'$ is the corresponding 1-dimensional space. 
\begin{Lm} The maps (\ref{def_twisted_W_action_on_Bunt_T})
define a left $W$-action on $\Bunt_T$. 
\end{Lm}
\begin{Prf}
Recall that the line bundles $\cL_T$ on $\Bun_T$ and $\tau$ on $\Bun_{T^{\sharp}}$ are $W$-equivariant. 

Let $W$ act as above on $\Bunt_T$, and 
on $\Bun_{T^{\sharp}}\times\Bunt_T$ as the product of the $W$-actions on $\Bun_{T^{\sharp}}$ and on $\Bunt_T$. Then $a$ is $W$-equivariant. 
\end{Prf}

\medskip
  
   For $w\in W$ denote by $K\mapsto w\ast K$ the direct image functor $\D_{\zeta}(\Bunt_T)\to \D_{\zeta}(\Bunt_T)$ for the new action map $w\ast : \Bunt_T\to\Bunt_T$. 
   
   For each simple root $\bar\alpha: \Gm\to T^{\sharp}$ of $\check{G}_n$ let $a_{\bar\alpha}: \Bun_1\times\Bunt_T\to \Bunt_T$ denote the restriction of $a$ under the push-out map $\Bun_1\toup{\bar\alpha} \Bun_{T^{\sharp}}$. Call $\cS\in \D_{\zeta}(\Bunt_T)$ regular if for each simple root $\bar\alpha$ of $\check{G}_n$ one has
\begin{equation}
\label{reg_complex_def}
(a_{\bar\alpha})_!\pr_2^*\cS=0
\end{equation}
This defines the full triangulated subcategory $\D_{\zeta}(\Bunt_T)^{reg}\subset \D_{\zeta}(\Bunt_T)$ of regular complexes. Equivalently, instead of (\ref{reg_complex_def}) one can require the property $\pr_{2 !}a_{\bar\alpha}^*\cS=0$ to define the regularity.

\begin{Con} 
\label{Con_functional_equation}
For $w\in W$ and $\cS\in \D_{\zeta}(\Bunt_T)^{reg}$ there is an  isomorphism
$$
\Eis(w\ast \cS)\,\iso\, \Eis(\cS)
$$
functorial in $\cS\in \D_{\zeta}(\Bunt_T)^{reg}$.
\end{Con}
\begin{Rem} Let $E$ be a $\check{T}^{\sharp}$-local system on $X$, $\cK_E\in\D_{\zeta}(\Bunt_T)$ be the Hecke eigen-sheaf associated to $E$ in (\cite{L}, Proposition~2.2). Then $\cK_E$ is regular if and only if $E^{\bar\alpha}$ is not trivial for each simple root $\bar\alpha$ of $\check{G}_n$.
\end{Rem}

\subsection{Action of $\Bun_{Z(G)}$} 
\label{Section_action_Bun_Z(G)}
The stack $\Bun_{Z(G)}$ is a group stack acting naturally on $\Bun_G$ by tensor product. For $\cT\in\Bun_{Z(F)}, \cF\in\Bun_G$ there is a canonical $\ZZ/2\ZZ$-graded isomorphism
$$
\cL_{\cF\otimes\cT}\,\iso\, \cL_{\cF}
$$
In particular, $\cL$ is canonically trivialized over $\Bun_{Z(G)}$. So, $\Bun_{Z(G)}$ acts on $\Bunt_G$, namely $\cT\in\Bun_{Z(G)}$ sends $(\cF, \cU_G)\in\Bunt_G$ to $(\cF\otimes\cT, \cU)$. 

 Set $(\Lambda^{\sharp})^{\check{}}=\Hom(\Lambda^{\sharp}, \ZZ)$. The map $\frac{\iota}{n}:\Lambda\otimes\Lambda^{\sharp}\to\ZZ$ yields a map 
$\frac{\iota}{n}:\Lambda\to (\Lambda^{\sharp})^{\check{}}$.
Consider for a moment $\check{T}^{\sharp}=\Gm\otimes(\Lambda^{\sharp})^{\check{}}$ as a torus over $\Spec k$. The map $\eta: T\to \check{T}^{\sharp}$ induced by $\frac{\iota}{n}$ gives the push-out map $\eta_X: \Bun_T\to \Bun_{\check{T}^{\sharp}}$. 
 
  Write $Z(\check{G}_n)_n$ for the $n$-torsor subgroup of $Z(\check{G}_n)$.
  
\begin{Lm} 
The map $\eta: T\to \check{T}^{\sharp}$ sends $Z(G)$ to $Z(\check{G}_n)_n$. 
\end{Lm}
\begin{Prf}
For $\bar\alpha\in \Lambda^{\sharp}$ one has $\bar\alpha\comp\eta=\frac{\iota(\bar\alpha)}{n}$. If $\bar\alpha$ is a simple root of $\check{G}_n$ then $\frac{\iota(\bar\alpha)}{n}$ lies in the roots lattice of $G$. So, for $z\in Z(G)$ we get $\frac{\iota(\bar\alpha)}{n}(z)=1$, and $\eta(z)\in Z(\check{G}_n)$. To see that $\eta(z)\in Z(\check{G}_n)_n$, note that if $\nu\in \Lambda^{\sharp}$ then $\iota(\nu)$ lies in the roots lattice of $G$. 
\end{Prf} 
 
\medskip

 For $\cT\in \Bun_{Z(G)}$ let $\cT_{\eta}:=\eta_X(\cT)$ denote the corresponding $Z(\check{G}_n)_n$-torsor on $X$. For $\nu\in\Lambda^{\sharp}$ denote by $\cT_{\eta, \bar\zeta}^{\nu}$ the $\Qlb$-local system on $X$ obtained from $\cT_{\eta}$ via the push-out by 
$$
Z(\check{G}_n)_n\toup{\nu} \mu_n(k)\toup{\bar\zeta} \Qlb^*
$$  
For $\cT\in\Bun_{Z(G)}$ denote by $\sigma_{\cT}: \Bunt_G\to\Bunt_G$ the automorphism $(\cF, \cU)\mapsto (\cF\otimes\cT, \cU)$.

\begin{Pp} Let $\nu\in\Lambda^{\sharp,+}$, $\cT\in\Bun_{Z(G)}$. The functors $\D_{\zeta}(\Bunt_G)\to \D_{\zeta}(X\times\Bunt_G)$ given by
$$
K\mapsto \pr_1^*\cT_{\eta, \bar\zeta}^{\nu}\otimes \H^{\nu}_G(\sigma_{\cT}^*K) \;\;\;\;\;\mbox{and by}\;\;\;\;\;
K\mapsto (\id\times \sigma_{\cT})^*\H^{\nu}_G(K)
$$
are naturally isomorphic.  
\end{Pp}
\begin{Prf}
Recall from Section~\ref{section_Hecke_functors_for_tilde_G} that $\ov{\cH}^{\nu}_{\tilde G}$ classifies $(\cF, \cF'\in\Bun_G, x\in X, \beta: \cF\,\iso\, \cF'\mid_{X-x}, \cU,\cU')$ such that $\cF'$ is in the position $\le \nu$ with respect to $\cF$ at $x$, $\cU^N\,\iso\, \cL_{\cF}$, $\cU'^N\,\iso\, \cL_{\cF'}$. 

 Let $\Bun_{Z(G)}$ act on $\ov{\cH}^{\nu}_{\tilde G}$ so that $\cT\in\Bun_{Z(G)}$ sends the above point to the collection 
$$
(\cF\otimes \cT, \cF'\otimes\cT, x, \beta, \cU, \cU')
$$ 
with the induced isomorphisms $\cU^N\,\iso\, \cL_{\cF\otimes\cT}$, $\cU'^N\,\iso \cL_{\cF'\otimes\cT}$. Write $\sigma_{\cT}: \ov{\cH}^{\nu}_{\tilde G}\to \ov{\cH}^{\nu}_{\tilde G}$ for this map for a given $\cT$. We get a commutative diagram
\begin{equation}
\label{diag_for_sigma_cT}
\begin{array}{ccccc}
X\times \Bunt_G & \getsup{\pi\times \tilde h^{\la}_G} &  \ov{\cH}^{\nu}_{\tilde G} & \toup{\tilde h^{\ra}_G} & \Bunt_G\\
\downarrow\lefteqn{\scriptstyle \id\times \sigma_{\cT}} && \downarrow\lefteqn{\scriptstyle  \sigma_{\cT}} && \downarrow\lefteqn{\scriptstyle  \sigma_{\cT}}\\
X\times \Bunt_G & \getsup{\pi\times \tilde h^{\la}_G} &  \ov{\cH}^{\nu}_{\tilde G} & \toup{\tilde h^{\ra}_G} & \Bunt_G
\end{array}
\end{equation}
Note that $\cH^{\nu}_{\tilde G}$ is preserved by $\sigma_{\cT}$. The Hecke functor $\H^{\nu}_G$ is defined by formula 
(\ref{formula_for_H^nu_G_Section04}). 

 Let us establish an isomorphism 
\begin{equation} 
\label{iso_related_to_sigma_cT}
 \sigma_{\cT}^*\IC^{\nu}\,\iso\, \IC^{\nu}\otimes \pi^*\cT_{\eta, \bar\zeta}^{\nu}
\end{equation} 
over $\ov{\cH}^{\nu}_{\tilde G}$. Since both are irreducible perverse sheaves, it suffices to establish it over $\cH^{\nu}_{\tilde G}$.

As in the proof of Proposition~\ref{Lm_great_about_gradings}, there is a line bundle $\cB$ on $\cH^{\nu}_G$ such that $\cB^N$ is canonically the line bundle with fibre $\cL_{\cF}\otimes\cL^{-1}_{\cF'}$ at $(\cF,\cF',\beta,x)\in \cH^{\nu}_G$. The line bundle is uniquely defined, as the Picard group is torsion free. We get an isomorphism
$$
\eta: B(\mu_N)\times(\cH^{\nu}_G\times_{\Bun_G}\Bunt_G)\,\iso\, \cH^{\nu}_{\tilde G}
$$
where we used $h^{\ra}_G$ to define the fibred product in parentheses. It sends a collection $(\beta: \cF\,\iso\,\cF'\mid_{X-x}, \;\cU'^N\,\iso\, \cL_{\cF'}, \,\cU_0^N\,\iso\, k)$ to $(\beta, \cF,\cF', \cU,\cU')$, where $\cU=\cB_{(\cF,\cF',\beta)}\otimes \cU_0\otimes \cU'$ with the induced isomorphism $\cU^N\,\iso\, \cL_{\cF}$. The perverse sheaf $\IC^{\nu}$ is the intermediate extension of 
$$
\eta_*(\cL_{\zeta}\boxtimes \IC(\cH^{\nu}_G\times_{\Bun_G}\Bunt_G))
$$ 
Viewing $\cT$ as a $T$-torsor on $X$, for $\check{\mu}\in\check{\Lambda}$ we get the line bundle $\cL^{\check{\mu}}_{\cT}$ on $X$. Over $\cH^{\nu}_G$ one has an isomorphism 
$$
\sigma_{\cT}^*\cB\,\iso\, \cB\otimes \pi^*\cL^{\frac{-\iota(\nu)}{n}}_{\cT},
$$ 
where $\pi: \cH^{\nu}_G\to X$ sends $(\cF,\cF',\beta, x)$ to $x$. Note that $\cL^{\frac{\iota(\nu)}{n}}_{\cT}$ is a $\mu_n$-torsor on $X$ that we see as a map $X\to B(\mu_n)$. The restriction of $\cL_{\bar\zeta}$ under the latter map identifies with $\cT_{\eta, \bar\zeta}^{\nu}$, because the composition $Z(G)\toup{\eta} Z(\check{G}_n)_n\toup{\nu} \mu_n$ equals $\frac{\iota(\nu)}{n}$. The isomorphism (\ref{iso_related_to_sigma_cT}) follows. Our claim follows now from the diagram (\ref{diag_for_sigma_cT}).
\end{Prf}

\medskip

This applies to Eisenstein series as follows. Recall that a point of $\Bunb_{\tilde B}$ is a collection $(\cF_T,\cF, \nu, \cU, \cU_G)$, where $\nu^{\check{\lambda}}: \cL^{\check{\lambda}}_{\cF_T}\hook{} \cV^{\check{\lambda}}_{\cF}$ are inclusions of coherent sheaves for each dominant weight $\check{\lambda}$. 

Let $\Bun_{Z(G)}$ act on $\Bunb_{\tilde B}$ so that $\cT\in\Bun_{Z(G)}$ sends $(\cF_T,\cF, \nu, \cU, \cU_G)$ to 
$(\cF_T\otimes\cT,\cF\otimes\cT, \nu, \cU, \cU_G)$. This action preserves the open substack $\Bun_{\tilde B}$. Let $\Bun_{Z(G)}$ act on $\Bunt_T$ so that $\cT$ sends $(\cF_T, \cU)$ to $(\cF_T\otimes\cT, \cU)$. The diagram
$$
\Bunt_T\getsup{\tilde\gq} \Bunb_{\tilde B} \toup{\tilde\gp} \Bunt_G
$$
is $\Bun_{Z(G)}$-equivariant. 

\begin{Lm} Let $K\in \D_{\zeta}(\Bunt_T)$, $\cT\in\Bun_{Z(G)}$. One has an isomorphism 
$$
\sigma_{\cT}^*\Eis(K)\,\iso\, \Eis(\sigma_{\cT}^*K)
$$ 
functorial in $K\in \D_{\zeta}(\Bunt_T)$.
\end{Lm}
\begin{Prf}
Write also $\sigma_{\cT}: \Bunb_{\tilde B}\to \Bunb_{\tilde B}$ for the above action map by $\cT$. One has $\sigma_{\cT}^*\IC_{\zeta}\,\iso\,\IC_{\zeta}$ canonically. Our claim follows.
\end{Prf}

\section{Parabolic geometric Eisenstein series}
\label{Section_Parabolic geometric Eisenstein series}

\subsection{Definitions} 
\label{Section_4.1}
Let $P\subset G$ be a parabolic containing $B$, let $M$ be its Levi factor, write $\cI_M\subset\cI$ for the corresponding subset. Let $\cL_M$ denote the restriction of $\cL$ under $\Bun_M\to\Bun_G$. Let $\Bunt_M$ denote the gerb of $N$-th roots of $\cL_M$. 

The notations $\Bunt_P, \Bunb_P$ are those of \cite{BG}. Let $\Bunt_{P,\tilde G}$, $\Bunb_{P,\tilde G}$ be obtained from $\Bunt_P$, $\Bunb_P$ by the base change $\Bunt_G\to\Bun_G$. Let $\Bunt_{\tilde P}$, $\Bunb_{\tilde P}$ be obtained from $\Bunt_{P,\tilde G}$, $\Bunb_{P,\tilde G}$ by the base change $\Bunt_M\to\Bun_M$. 

 The Eisenstein series functor $\Eis_M: \D_{\zeta}(\Bunt_M)\to\D_{\zeta}(\Bunt_G)$ is defined as follows. By abuse of notations, the diagram of projections is denoted 
\begin{equation} 
\label{projections_from_Bunt_tildeP}
\Bunt_M\getsup{\tilde\gq} \Bunt_{\tilde P} \toup{\tilde\gp} \Bunt_G
\end{equation}
A point of $\Bunt_{\tilde P}$ is given by a point $(\cF_M,\cF,\nu)\in\Bunt_P$, where 
$$
\nu^{\cV}: \cV^{U(P)}_{\cF_M}\hook{} \cV_{\cF}
$$ 
is a morphism of coherent sheaves for each representation $\cV$ of $G$; $\cU,\cU_G$ are $\ZZ/2\ZZ$-graded lines of parity zero equipped with $\cU^N\,\iso\, (\cL_M)_{\cF_M}$, $\cU_G^N\,\iso\, \cL_{\cF}$. The map $\tilde\gq$ sends the above point to $(\cF_M, \cU)$, and $\tilde\gp$ sends the above point to $(\cF, \cU_G)$.

Let $\Bun_{\tilde P}\subset \Bunt_{\tilde P}$ be the preimage of $\Bun_P$ in $\Bunt_{\tilde P}$. For a point of $\Bun_P$ we have canonically $(\cL_M)_{\cF_M}\,\iso\, \cL_{\cF}$. One defines $\Bun_{P,\tilde G}$ similarly. We get an isomorphism
\begin{equation}
\label{iso_for_Bun_P_tildeG}
B(\mu_N)\times\Bun_{P, \tilde G}\,\iso\, \Bun_{\tilde P}
\end{equation}
sending $(\cF_P, \cU_G, \cU_0\in B(\mu_N))$ with $\cU_0^N\,\iso\, k$ to $(\cF_P, \cU_G, \cU)$ with $\cU=\cU_G\otimes\cU_0^{-1}$.   
By definition,
$$
\tilde\gq(\cF_M,\cF,\nu, \cU,\cU_G)=(\cF_M, \cU)\;\;\mbox{and}\;\;\; \tilde\gp(\cF_M,\cF,\nu, \cU,\cU_G)=(\cF,\cU_G)
$$ 
View $\cL_{\zeta}\boxtimes \IC(\Bun_{P,\tilde G})$ as a perverse sheaf on $\Bun_{\tilde P}$ via (\ref{iso_for_Bun_P_tildeG}). We still denote by $\IC_{\zeta}$ the intermediate extension of this perverse sheaf to $\Bunt_{\tilde P}$. Write $j_{\tilde P}: \Bun_{\tilde P}\hook{} \Bunt_{\tilde P}$ for the natural open immersion.

\begin{Def} For $K\in \D_{\zeta}(\Bunt_M)$ set
$$
\Eis_M(K)=\tilde\gp_!(\tilde\gq^*K\otimes\IC_{\zeta})[-\dim\Bun_M]
$$  
This gives a functor $\D_{\zeta}(\Bunt_M)\to\D_{\zeta}(\Bunt_G)$.     
\end{Def}

\subsubsection{} 
\label{Section_4.1.1_some_vanishing}
Set $\Lambda_{M,0}=\{\lambda\in\Lambda\mid \<\lambda, \check{\alpha}_i\>=0\,\mbox{for all}\; i\in \cI_M\}$. Let $Z(M)^0=\Gm\otimes \Lambda_{M,0}$, this is the connected component of unity of the center of $M$. Denote by $\check{\Lambda}_{M,0}$ the lattice dual to $\Lambda_{M,0}$. To $\mu\in \Lambda_{G,P}$ we associate the character $\Lambda_{M,0}\to\ZZ, \lambda\mapsto \kappa(\lambda,\mu)$ denoted $\kappa_M(\mu)$. This is well-defined, because $\kappa(\alpha_i)\in \ZZ\check{\alpha}_i$, and gives a homomorphism $\kappa_M: \Lambda_{G,P}\to \check{\Lambda}_{M,0}$.

The group $Z(M)^0$ acts on $\Bun_M$ by 2-automorphisms naturally. As in Section~\ref{Section_202} for $T$ one checks the following. If $\theta\in\Lambda_{G,P}$, $\cF\in\Bun_M^{\theta}$ then $Z(M)^0$ acts on $(\cL_M)_{\cF}$ by the character $-\kappa_M(\theta)$. The following is a generalization of (\cite{L}, Proposition~2.1). 

\begin{proof}{Proof of Proposition~\ref{Pp_vanishing_over_Bunt_M^theta}} A $k$-point $\cF\in\Bun_M^{\theta}$ defines a map $f: B(Z(M)^0)\to \Bun_M^{\theta}$. Let $\wt B(Z(M)^0)$ be the restriction of the gerb $\Bunt_M^{\theta}\to\Bun_M^{\theta}$ under this map. As above, we get the category $\D_{\zeta}(\wt B(Z(M)^0))$. By (\cite{L}, Lemma~5.3), $\D_{\zeta}(\wt B(Z(M)^0))=0$ unless $\kappa_M(\theta)\in N\check{\Lambda}_{M,0}$. Our claim follows.
\end{proof}

\begin{Pp} 
\label{Pp_IC_zeta_is_ULA_for_P}
The complexes $\IC_{\zeta}$ and $(j_{\tilde P})_!j_{\tilde P}^*\IC_{\zeta}$
are ULA with respect to $\tilde\gq: \Bunt_{\tilde P}\to \Bunt_M$. 
\end{Pp}

\subsubsection{Proof of Proposition~\ref{Pp_IC_zeta_is_ULA_for_P}}
The argument from (\cite{BG}, Proposition~5.1.5) applies in our setting. For $\cF_M, \cF'_M\in\Bun_M(k)$ write $\cF_M\prec \cF'_M$ if there is $\lambda\in\Lambda^+$ such that $(\cF_M,\cF'_M)$ is the image of 
$$
\cH^{\lambda}_M\toup{h^{\la}_M\times h^{\ra}_M} \Bun_M\times\Bun_M
$$ 
of some $k$-point. Let $\sim$ be the equivalence relation on $\Bun_M(k)$ generated by $\prec$. Write $\Lambda^{\sharp}_{G,P}$ for the image of the natural map $\Lambda^{\sharp}\to\Lambda_{G,P}$.

 Let $\Lambda_{G,P}^{pos}$ be the $\ZZ_+$-span of $\alpha_i, i\in \cI-\cI_M$ in $\Lambda_{G,P}$. For $\theta\in\Lambda_{G,P}^{pos}$ we have the open substack $\Bunt_P^{\le\theta}$ defined in (\cite{BG}, Section~5.3.1). Recall that $\cup_{\theta} \Bunt_P^{\le\theta}=\Bunt_P$, so it is sufficient to show that $\IC_{\zeta}$ is ULA with respect to $\Bunt^{\le\theta}_{\tilde P}\to \Bunt_M$ for any $\theta\in\Lambda_{G,P}^{\theta}$. 

 Let $\oBun_M$ denote the biggest open substack of $\Bun_M$ such that for its preimage $\owtBun_M$ in $\Bunt_M$ the perverse sheaf $\IC_{\zeta}$ is ULA with respect to $\Bunt_{\tilde P}^{\le\theta}\times_{\Bun_M} \oBun_M\to \owtBun_M$.
 
\begin{Pp}
\label{Pp_property_of_oBun_M}
 If $\cF_M\prec\cF_M$ are $k$-points of $\Bun_M$ then $\cF'_M\in \oBun_M$ if and only if $\cF_M\in \oBun_M$.
\end{Pp}
\begin{proof} 
The argument from (\cite{BG}, Proposition~5.3.4) applies. One only needs to check the following. Pick $\lambda\in\Lambda^+$. Consider the stack $Z$ classifying 
$$
(x\in X, \cF_G, \cF_M, \cF'_G, \cF'_M, \kappa, \kappa', \beta, \beta_M, 
\cU, \cU',\cU_G,\cU'_G)
$$
where $\cU, \cU',\cU_G,\cU'_G$ are lines equipped with isomorphisms
$$
\cU^N\,\iso\, \cL_{\cF_M}, \cU'^N\,\iso\, \cL_{\cF'_M}, \cU^N_G\,\iso\, \cL_{\cF_G}, {\cU'}^N_G\,\iso\, \cL_{\cF'_G},
$$
$(\cF_G, \cF_M, \kappa, \cU,\cU_G)\in\Bunt_{\tilde P}$, $(\cF'_G, \cF'_M, \kappa', \cU',\cU'_G)\in\Bunt_{\tilde P}$, and 
$$
\beta: \cF_G\,\iso\, \cF'_G\mid_{X-x}, \; \beta_M: \cF_M\,\iso\, \cF'_M\mid_{X-x}
$$ 
such that
\begin{itemize}
\item $\cF'_G$ is in the $G$-position $\lambda$ with respect to $\cF_G$ at $x$;
\item $\cF'_M$ is in the $M$-position $\lambda$ with respect to $\cF_M$ at $x$;
\item the maps $\kappa: V^{U(P)}_{\cF_M}\to V_{\cF_G}$ have no zero at $x$ for any $V\in\Rep(G)$;
\item the maps $\kappa': V^{U(P)}_{\cF'_M}\to V_{\cF'_G}$ have no zero at $x$ for any $V\in\Rep(G)$.
\end{itemize}

We have two smooth projections $\Bunt_{\tilde P}\getsup{h^{\la}} Z\toup{h^{\ra}}\Bunt_{\tilde P}$, where $h^{\la}$ and $h^{\ra}$ sends the above point to 
$$
(\cF_G, \cF_M, \kappa, \cU,\cU_G)\;\;\;\;\mbox{and}\;\;\;\; (\cF'_G, \cF'_M, \kappa', \cU',\cU'_G)
$$ 
respectively. Then the line bundle on $Z$ with fibre $\cL_{\cF_G}\otimes\cL_{\cF'_M}\otimes\cL_{\cF_M}^*\otimes\cL_{\cF'_G}^*$ is canonically trivialized. So,
 the perverse sheaves $(h^{\ra})^*\IC_{\zeta}[\dimrel(h^{\ra})]$ and $(h^{\la})^*\IC_{\zeta}[\dimrel(h^{\la})]$ are locally isomorphic in the smooth topology on $Z$. The rest of the argument is exactly as in (\cite{BG}, Proposition~5.3.4). 
\end{proof}

 One finishes the proof of Proposition~\ref{Pp_IC_zeta_is_ULA_for_P} now 
as in (\cite{BG}, Theorem~5.1.5). Theorem~\ref{Th_Eis_P_commutes_with_Verdier} also follows from Proposition~\ref{Pp_IC_zeta_is_ULA_for_P} as in the case $M=T$. 

\subsection{Hecke functors for $M$} 
\label{Section_Hecke functors for M}
Let $\cL$ also denote the restriction of $\cL$ under $\Gr_M\to \Gr_G$. Let $\wt\Gr_M\to \wt\Gr_P$ be obtained from $\Gr_M\to \Gr_P$ by the base change $\wt\Gr_G\to \Gr_G$. Let $\PPerv_{M,G,n}$ be the category of $M(\bO)$-equivariant perverse sheaves  
on $\wt\Gr_M$, on which $\mu_N(k)$ acts by $\zeta$.   

 Let $\Lambda^+_M\subset\Lambda$ be the coweights of $T$ dominant for $M$. Set $\Lambda^{\sharp,+}_M=\Lambda^{\sharp}\cap \Lambda^+_M$. As in (\cite{FL}, Section~4.1.1) for $\nu\in \Lambda^{\sharp,+}_M$ 
 we get the perverse sheaf $\cA^{\nu}_{M,\cE}\in \PPerv_{M,G,n}$ on $\wt\Gr_M$. Here $\cE$ is the square root of $\Omega(\bO)$ that we picked in Section~\ref{section_Hecke_functors_for_tilde_G}. 
 
  Recall that $\check{M}_n\subset \check{G}_n$ is the standard Levi subgroup corresponding to $\cI_M$. Note that $\Lambda^{\sharp,+}_M$ are exactly the dominant weights of $\check{M}_n$. In (\cite{FL}, Section~4.2) we introduced a tensor category $\PPerv_{M,G,n}^{\natural}$ (obtained from $\PPerv_{M,G,n}$ by changing the commutativity constraint) and established an equivalence $\PPerv_{M,G,n}^{\natural}\,\iso\, \Rep(\check{M}_n)$. For $V\in \Rep(\check{M}_n)$ we denote by $\Loc(V)\in \PPerv_{M,G,n}$ the corresponding perverse sheaf on $\wt\Gr_M$. If $V$ is an irreducible $\check{M}$-module with highest weight $\nu$ then $\Loc(V)=\cA^{\nu}_{M,\cE}$. We write $\Loc_{\zeta}(V)$ if we need to express the dependence on $\zeta$. 
  
  The Hecke stack $\cH_M$, the diagram
$$
\begin{array}{ccccc}
\Bunt_M\times X & \getsuplong{\tilde h^{\la}_M\times\pi} & \cH_{\tilde M} &\touplong{\tilde h^{\ra}_M} & \Bunt_M\\
\downarrow && \downarrow &&\downarrow\\
\Bun_M\times X & \getsuplong{h^{\la}_M\times\pi} & \cH_M &\touplong{h^{\ra}_M} &\Bun_M
\end{array}
$$
and $\Gr_{M,X}$ are defined as for $G$. The stack $\cH_{\tilde M}$ classifies $M$-torsors $\cF_M, \cF'_M$ on $X$, $x\in X, \beta_M: \cF_M\,\iso\, \cF_M'\mid_{X-x}$ together with lines $\cU, \cU'$ equipped with 
$$ 
\cU^N\,\iso\, (\cL_M)_{\cF_M}, \;\;\;\; \cU'^N\,\iso\, (\cL_M)_{\cF'_M}
$$  

 The line bundle $\cL_X$ on $\Gr_{M,X}$ is the restriction of $\cL_X$ under $\Gr_{M,X}\to \Gr_{G,X}$. 
We similarly have the isomorphisms $\gamma^{\la}$ (resp., $\gamma^{\ra}$) 
$$
\Bun_{M,X}\times_{M_X} \Gr_{M,X}\,\iso\, \cH_M
$$ 
such that the projection to the first term corresponds to $h^{\la}_M$ (resp., $h^{\ra}_M$). Over the stack $\Bun_{M,X}\times_{M_X} \Gr_{M,X}$ we have canonically the isomorphism
$$
(\gamma^{\ra})^*(h^{\la}_M)^*\cL_M\,\iso\, \cL_M\tboxtimes \cL_X
$$

 Let $\PPerv_{M,G, n,X}$ be the category of $K\in \D(\wt\Gr_{M,X})$ such that $K[1]$ is perverse, $M_X$-equivariant, and $\mu_N(k)$ acts on $K$ by $\zeta$. We have the fully faithful functor defined in (\cite{FL}, Section~4.1.2)
$$
\tau^0: \PPerv_{M,G,n}\to\PPerv_{M,G,n,X}
$$
For $\nu\in \Lambda^{\sharp,+}_M$ set $\cA^{\nu}_M=\tau^0(\cA^{\nu}_{M,\cE})\in \PPerv_{M,G,n,X}$. 

Consider the $M_X$-torsor 
$$
\tilde\gamma^{\ra}: \Bunt_{M,X}\times_X \wt\Gr_{M,X}\to \cH_{\tilde M}
$$
For a $M_X$-equivariant perverse sheaf $\cS$ on $\wt\Gr_{M,X}$ and $\cT\in \D(\Bunt_M)$ one defines their twisted exterior product $(\cT\tboxtimes \cS)^r$ on $\cH_{\tilde M}$. This is the descent via $\tilde\gamma^{\ra}$, it is normalized to be perverse for $\cT,\cS$ perverse. Similarly, one gets $(\cT\tboxtimes \cS)^l$ on $\cH_{\tilde M}$.
  
 Now for $\nu\in \Lambda^{\sharp,+}_M$ let
$$
\H^{\nu}_M: \D_{\zeta}(\Bunt_M)\to \D_{\zeta}(\Bunt_M\times X)
$$
be given by 
$$
\H^{\nu}_M(\cT)=(\tilde h^{\la}_M\times\pi)_!((\cT\tboxtimes \cA_M^{-w_0^M(\nu)})^r)
$$ 

 As for $G$, one has the covariant functor $\star: \PPerv_{M,G,n, \zeta^{-1}}\to \PPerv_{M,G,n, \zeta}$. For $\nu\in\Lambda^{\sharp, +}_M$ it sends $\cA^{\nu}_{M,\cE}$ to $\cA^{-w_0^M(\nu)}_{M,\cE}$. For $\cS\in\PPerv_{M,G,n}$, $\cT\in\D_{\zeta}(\Bunt_M)$ we define
$$
\H^{\la}_M(\cS, \cT)=(\tilde h^{\la}_M\times\pi)_!(\cT\tboxtimes \tau^0(\star \cS))^r\;\;\;\mbox{and}\;\;\; \H^{\ra}_M(\cS,\cT)=(\tilde h^{\ra}_M\times\pi)_!(\cF\tboxtimes \tau^0(\cS))^l
$$

\subsection{Geometric restriction functors} 
\label{Section_Geometric restriction functors}
Write $\Lambda_{G,P}=\Lambda/\{\alpha_i, i\in \cI_M\}$ for the lattice of cocharacters of $M/[M,M]$. Let $\check{\Lambda}_{G,P}$ denote the dual lattice. For $\theta\in\Lambda_{G,P}$ the connected component $\Gr_M^{\theta}$ is the one containing $t^{\nu}M(\bO)$ for any $\nu\in \Lambda$ over $\theta$. For $\theta\in\Lambda_{G,P}$ denote by $\Gr_P^{\theta}$ the ind-scheme classifying $(\cF_G,\beta: \cF_G\,\iso\, \cF^0_G\mid_{D_x^*})\Gr_G$ such that for any $\check{\lambda}\in\check{\Lambda}^+\cap \check{\Lambda}_{G,P}$ the corresponding map
$$
\kappa^{\check{\lambda}}: \cL^{\check{\lambda}}_{\cF^0_T}(-\<\theta, \check{\lambda}\>x)\to \cV^{\check{\lambda}}_{\cF_G}
$$
is regular and has no zero. This ind-scheme was denoted $S^{\theta}_P$ in (\cite{BG}, Section~4.3.1). 

Let $\wt\Gr_M^{\theta}\to \wt\Gr_P^{\theta}$ be obtained from $\Gr_M^{\theta}\to\Gr_P^{\theta}$ by the base change $\wt\Gr_G\to\Gr_G$.
For $\theta\in\Lambda_{G,P}$ we have the diagram
$$
\wt\Gr_M^{\theta}\;\getsuplong{\tilde\gt^{\theta}_P}\; \wt\Gr_P^{\theta}\; \touplong{\tilde\gs^{\theta}_P} \;\wt\Gr_G
$$
The next result follows from \cite{FL}. 

\begin{Pp} 
\label{Pp_gRes}
There is a functor $\gRes: \PPerv^{\natural}_{G,n}\to\PPerv^{\natural}_{M,G,n}$ given by 
\begin{equation}
\label{def_functor_gRes_M_G_n}
K\mapsto \mathop{\oplus}\limits_{\theta\in\Lambda_{G,P}} (\tilde\gt_P^{\theta})_!(\tilde\gs^{\theta}_P)^* K[\<\theta, 2\check{\rho}-2\check{\rho}_M\>],
\end{equation}
and the following diagram of categories naturally commutes
$$
\begin{array}{ccc}
 \PPerv^{\natural}_{G,n} & \iso & \Rep(\check{G}_n)\\
 \downarrow\lefteqn{\scriptstyle \gRes} &&  \downarrow\lefteqn{\scriptstyle \Res}\\
\PPerv^{\natural}_{M,G,n} & \iso & \Rep(\check{M}_n)
\end{array}
$$
Here the horizontal equivalences are those constructed in \cite{FL}, and $\Res$ is the restriction functor for $\check{M}_n\hook{}\check{G}_n$. 
\end{Pp}

Let $\Lambda_{G,P}^{\sharp}$ denote the image of the map $\Lambda^{\sharp}\to \Lambda_{G,P}$ given by $\mu\mapsto\mu$. Note that $\Lambda_{G,P}^{\sharp}$ is a subgroup of finite index in the free abelian group $\Lambda_{G,P}$, so $\Lambda_{G,P}^{\sharp}$ is also free. In the formula (\ref{def_functor_gRes_M_G_n}) for $\gRes$ one may replace $\Lambda_{G,P}$ by $\Lambda_{G,P}^{\sharp}$. 
 
 The center $Z(\check{M}_n)$ is not connected in general under our assumptions. Write $C^*(\check{M}_n)$ for the cocenter of $\check{M}_n$, that is, the quotient of $\Lambda^{\sharp}$ by the roots lattice of $\check{M}_n$. 
We have canonically $\Hom(Z(\check{M}_n),\Gm)\,\iso\, C^*(\check{M}_n)$ by (\cite{Sp}, 2.15(b)). 
 The natural map $c_P: C^*(\check{M}_n)\to \Lambda^{\sharp}_{G,P}$ is surjective, but not injective in general. Its kernel is finite and coincides with the torsion subgroup in $C^*(\check{M}_n)$. Indeed, if $\sum_i a_i\alpha_i\in\Lambda$ vanishes in $\Lambda_{G,P}$ then it is of the form $\sum_{i\in \cI_M} a_i\alpha_i$, and a multiple of this element lies in the roots lattice $\oplus_{i\in \cI_M} \delta_i\alpha_i\ZZ$ of $\check{M}_n$. 
 
 Recall that for $\nu\in \Lambda^{\sharp,+}_M$ we denote by $U^{\nu}$ the irreducible representation of $\check{M}_n$ with highest weight $\nu$. 
 
 For $\nu\in\Lambda^+_M$ lying over $\theta\in\Lambda_{G,P}$ let $\tilde\gt^{\nu}_P: \wt\Gr_P^{\nu}\to \wt\Gr_M^{\nu}$ be the map obtained from
$\tilde\gt^{\theta}_P$ by the base change $\wt\Gr_M^{\nu}\to \wt\Gr_M^{\theta}$. For $\mu\in \Lambda^{\sharp,+}$ recall the local system $E^{\mu}_{\cE}$ over $\wt\Gr^{\mu}_G$. From Proposition~\ref{Pp_gRes} one gets the following.
 
\begin{Cor} 
\label{Cor_2}
Let $\nu\in\Lambda^+_M$ lying over $\theta\in\Lambda_{G,P}$. Let $\mu\in \Lambda^{\sharp,+}$. The complex
$$
(\tilde\gt^{\nu}_P)_!(E^{\mu}_{\cE})\mid_{\wt\Gr_P^{\nu}\cap\wt\Gr_G^{\mu}}[\<\mu, 2\check{\rho}\>+\<\nu, 2\check{\rho}-2\check{\rho}_M\>]
$$
is placed in perverse degrees $\le 0$, its 0-th perverse cohomology sheaf vanishes unless $\nu\in \Lambda^{\sharp,+}_M$, and in the latter case identifies canonically with $\cA^{\nu}_{M,\cE}\otimes \Hom_{\check{M}_n}(U^{\nu}, V^{\mu})$. The space $\Hom_{\check{M}_n}(U^{\nu}, V^{\mu})$ admits a base formed by those connected components $C$ of $\wt\Gr_P^{\nu}\cap \wt\Gr_G^{\mu}$ of dimension $\<\mu+\nu, \check{\rho}\>$ for which $E^{\mu}_{\cE}$ descends under the map $\tilde\gt^{\nu}_P: C\to \wt\Gr_M^{\nu}$. 
\end{Cor}
 
  Note that the descent property in Corollary~\ref{Cor_2} can be checked at the generic point of the component $C$. 
  
\subsection{Proof of Theorem~\ref{Th_Hecke_property_of_Eis^G_M}} The proof of (\cite{BG}, Theorem~2.3.7) essentially applies in our setting with some minor changes that we explain. To simplify notations, we prove a version of Theorem~\ref{Th_Hecke_property_of_Eis^G_M} with $x\in X$ fixed.

\subsubsection{} We have the stack $_{x,\infty}\Bunt_P$ defined in (\cite{BG}, Section~4.1.1). For $\nu\in\Lambda^+_M$ one also has the diagram $_{x,\nu}\Bunt_P\hook{}{_{x,\ge\nu}\Bunt_P}\hook{} {_{x,\infty}\Bunt_P}$ defined in (\cite{BG}, Sections~4.1.1 and 4.2.1), the first map is an open immersion, the second one is a closed immersion. Let $_{x,\nu}\Bunt_{\tilde P}\hook{}{_{x,\ge\nu}\Bunt_{\tilde P}}\hook{} {_{x,\infty}\Bunt_{\tilde P}}$ be obtained from the above by the base change $\Bunt_M\times \Bunt_G\to \Bun_M\times\Bun_G$. 
  
  Recall the stacks $_{x,\infty}Z_{P,M}$ and $_{x,\infty}Z_{P,G}$ defined in (\cite{BG}, Section~4.1.2 and 4.1.4). We similarly define the stacks $_{x,\infty}Z_{\tilde P, \tilde M}$ and $_{x,\infty}Z_{\tilde P,\tilde G}$ included into the diagrams
$$
\begin{array}{ccccc}
_{x,\infty}\Bunt_{\tilde P} & \getsup{'\tilde h^{\la}_M} & {_{x,\infty}Z_{\tilde P,\tilde M}} & \toup{'\tilde h^{\ra}_M} & _{x,\infty}\Bunt_{\tilde P}\\
\downarrow\lefteqn{\scriptstyle \tilde\gq_P} && \downarrow\lefteqn{\scriptstyle '\tilde\gq_P}&& \downarrow\lefteqn{\scriptstyle \tilde\gq_P}\\
\Bunt_M & \getsup{\tilde h^{\la}_M} & {_x\cH_{\tilde M}} & \toup{\tilde h^{\ra}_M} & \Bunt_M
\end{array}
$$
and
$$
\begin{array}{ccccc}
_{x,\infty}\Bunt_{\tilde P} & \getsup{'\tilde h^{\la}_G} & {_{x,\infty}Z_{\tilde P,\tilde G}} & \toup{'\tilde h^{\ra}_G} & _{x,\infty}\Bunt_{\tilde P}\\
\downarrow\lefteqn{\scriptstyle \tilde\gp_P} && \downarrow\lefteqn{\scriptstyle '\tilde\gp_P}&& \downarrow\lefteqn{\scriptstyle \tilde\gp_P}\\
\Bunt_G & \getsup{\tilde h^{\la}_G} & {_x\cH_{\tilde G}} & \toup{\tilde h^{\ra}_G} & \Bunt_G
\end{array}
$$
Both squares in each of the above diagrams are cartesian.

 A point of $_{x,\infty}\Bunt_{\tilde P}$ is $(\cF_M,\cF_G,\kappa)\in {_{x,\infty}\Bunt_P}$ and $(\cF_M,\cU)\in\Bunt_M, (\cF_G,\cU_G)\in\Bunt_G$. Let $\mu_N(k)\times\mu_N(k)$ act on $_{x,\infty}\Bunt_{\tilde P}$ by 2-automorphisms so that $(a,a_G)$ acts as $a$ on $\cU$, as $a_G$ on $\cU_G$ and trivially on $(\cF_M, \cF_G,\kappa)$. Denote by 
$$
\D_{\zeta}(_{x,\infty}\Bunt_{\tilde P})\subset \D(_{x,\infty}\Bunt_{\tilde P})
$$ 
the full subcategory of objects on which any $(a,a_G)\in \mu_N(k)\times\mu_N(k)$ acts by $\zeta(\frac{a_G}{a})$. 
  
  Now for $\cS\in \PPerv_{M,G,n}$, $\cT\in \D_{\zeta}(_{x,\infty}\Bunt_{\tilde P})$ one defines the functors 
$$
_x\H^{\la}_{P,M}, {_x\H^{\ra}_{P,M}}: \PPerv_{M,G,n}\times \D_{\zeta}(_{x,\infty}\Bunt_{\tilde P})\to \D_{\zeta}(_{x,\infty}\Bunt_{\tilde P})
$$ 
and 
$$
_x\H^{\la}_{P,G}, {_x\H^{\ra}_{P,G}}: \PPerv_{G,n}\times \D_{\zeta}(_{x,\infty}\Bunt_{\tilde P})\to \D_{\zeta}(_{x,\infty}\Bunt_{\tilde P})
$$   
as in (\cite{BG}, Sections~4.1.2-4.1.4). In particular, for $\nu\in\Lambda^{\sharp, +}_M$ we get 
$$
_x\H^{\nu}_{P,M}: \D_{\zeta}(_{x,\infty}\Bunt_{\tilde P})\to \D_{\zeta}(_{x,\infty}\Bunt_{\tilde P})
$$ 
given by
$
_x\H^{\nu}_{P,M}(\cT)={_x\H^{\la}_{P,M}}(\cA^{\nu}_{\cE}, \cT)
$. For $\lambda\in\Lambda^{\sharp, +}$ we get 
$$
_x\H^{\lambda}_{P,G}: \D_{\zeta}(_{x,\infty}\Bunt_{\tilde P})\to \D_{\zeta}(_{x,\infty}\Bunt_{\tilde P})
$$ 
given by $_x\H^{\lambda}_{P,G}(\cT)={_x\H^{\la}_{P,G}}(\cA^{\lambda}_{\cE}, \cT)$. 

 For $\nu\in \Lambda^{\sharp, +}_M$ we define the perverse sheaf $\IC_{\nu, \zeta}\in \D_{\zeta}(_{x,\ge\nu}\Bunt_{\tilde P})$ as follows. As in (\cite{BG}, Section~4.2.3) we denote by 
$$
j_{\nu}: {_{x,\nu}\Bunt_{\tilde P}}\hook{} {_{x,\ge\nu}\Bunt_{\tilde P}}
$$ 
the corresponding open immersion. The stack $_{x,\nu}\Bunt_{\tilde P}$ classifies $(\cF^1_M, \cF_G,\kappa)\in {_{x,0}\Bunt_P}$, $\cF_M\in\Bun_M$, $\beta_M: \cF_M\,\iso\cF^1_M\mid_{X-x}$ such that $\cF_M$ is in the position $\nu$ with respect to $\cF^1_M$ at $x$, $(\cF_M, \cU)\in\Bunt_M$, $(\cF_G, \cU_G)\in\Bunt_G$.  

 The projection $_{x,\nu}\Bunt_{\tilde P}\to {_{x,0}\Bun_P\times_{\Bun_G}\Bunt_G}$ sending the above point to $(\cF^1_M, \cF_G,\kappa, \cU_G)$ is a locally trivial fibration (in smooth topology) with typical fibre $\wt\Gr_{M,x}^{\nu}$. As in Section~\ref{Section_Hecke functors for M}, one gets the twisted exterior product 
$$
\IC(_{x,0}\Bun_P\times_{\Bun_G}\Bunt_G)\tboxtimes \cA^{\nu}_{M,\cE}
$$ 
on $_{x,\nu}\Bunt_{\tilde P}$. Then $\IC_{\nu, \zeta}$ is defined as its intermediate extension under $j_{\nu}$. In particular, $\IC_{\zeta}=\IC_{0,\zeta}$ on $_{x,\ge 0}\Bunt_{\tilde P}=\Bunt_{\tilde P}$. 
 
  The following are analogs of (\cite{BG}, Theorem~4.1.3 and 4.1.5). 
  
\begin{Pp} 
\label{Pp_analog_413}
For $\nu\in\Lambda^{\sharp,+}_M$ one has canonically $_x\H^{\nu}_{P,M}(\IC_{\zeta})\,\iso\, \IC_{-w_0^M(\nu), \zeta}$.
\end{Pp}
  
\begin{Pp} 
\label{Pp_analog_415}
For $\lambda\in\Lambda^{\sharp, +}$ there is a canonical isomorphism
$$
_x\H^{\lambda}_{P,G}(\IC_{\zeta})\,\iso\, \mathop{\oplus}\limits_{\nu\in\Lambda^{\sharp, +}_M} \IC_{\nu,\zeta}\otimes\Hom_{\check{M}_n}(U^{\nu}, V^{\lambda})\, .
$$
\end{Pp}  
  
\begin{Cor} The two functors $\PPerv_{G,n}\to \D_{\zeta}(_{x,\infty}\Bunt_{\tilde P})$
$$
\cS\mapsto {_x\H^{\la}_{P,G}(\cS, \IC_{\zeta})}\;\;\;\mbox{and}\;\;\; \cS\mapsto   
{_x\H^{\ra}_{P,M}(\gRes(\cS), \IC_{\zeta})}
$$
are canonically isomorphic. This isomorphism is compatible with the tensor structure on $\PPerv_{G,n}$ as in (\cite{BG}, Corollary~4.7).
\end{Cor}

 Now as in (\cite{BG}, Section~4.1.8) combination of Propositions~\ref{Pp_analog_413} and \ref{Pp_analog_415} implies Theorem~\ref{Th_Hecke_property_of_Eis^G_M}. The proof of Proposition~\ref{Pp_analog_413} is completely analogous to (\cite{BG}, Theorem~4.1.3).
  
\subsubsection{Proof of Proposition~\ref{Pp_analog_415}} The proof of (\cite{BG}, Theorem~4.1.5) applies in our situation with the role of (\cite{BG}, Corollary~4.3.5) replaced by our Corollary~\ref{Cor_2}. 

 For the convenience of the reader recall that in the proof of (\cite{BG}, Theorem~4.1.5) for $\nu,\nu',\eta\in\Lambda_M^+$, $\lambda'\in\Lambda^+_G$ the following locally closed substack $Z^{\nu,\nu',\eta,\lambda'}_{P,G}\hook{}{_{x,\infty}Z_{P,G}}$ plays a key role. It classifies
$$
(\cF^1_M, \cF_G,\kappa)\in {_{x,0}\Bunt_P}
, \; \cF_M\,\iso\, \cF^1_M\mid_{X-x}, \; \cF^2_M\,\iso\, \cF^1_M\mid_{X-x},\; \cF'_G\,\iso\, \cF_G\mid_{X-x}
$$ 
such that $\cF_M$ is in the position $\nu$ w.r.t $\cF^1_M$ at $x$, $\cF^2_M$ is in the position $\eta$ w.r.t. $\cF^1_M$ at $x$, $\cF_M$ is in the position $\nu'$ w.r.t. $\cF^2_M$ at $x$,  $(\cF^2_M, \cF'_G, \kappa)\in {_{x,0}\Bunt_P}$, $\cF'_G$ is in the position $\lambda'$ w.r.t. $\cF_G$ at $x$.

It is included into the diagram
$$
_{x,\nu}\Bunt_P\; \getsup{'h^{\la}_G}\;  Z^{\nu,\nu',\eta,\lambda'}_{P,G}\;\toup{'h^{\ra}_G}\; {_{x,\nu'}\Bunt_P},
$$
where $'h^{\la}_G$ sends the above point to $((\cF^1_M, \cF_G,\kappa)\in {_{x,0}\Bunt_P}, \cF_M\,\iso\, \cF^1_M\mid_{X-x})$. The map $'h^{\ra}_G$ sends the above point to 
$$
((\cF^2_M, \cF'_G, \kappa)\in {_{x,0}\Bunt_P}, \cF_M\,\iso\, \cF^2_M\mid_{X-x})
$$ 

 In our situation we define $Z^{\nu,\nu',\eta,\lambda'}_{\tilde P,\tilde G}$ by the base change ${_{x,\infty}Z_{\tilde P,\tilde G}}\to {_{x,\infty}Z_{P,G}}$.
Let $K^{\nu,\nu', \eta,\lambda'}$ denote the $!$-direct image under 
$$
'\tilde h^{\la}_G: Z^{\nu,\nu',\eta,\lambda'}_{\tilde P,\tilde G}\to {_{x,\nu}\Bunt_{\tilde P}}
$$ 
of the $*$-restriction of $(\cA^{-w_0(\lambda)}_{\cE}\tboxtimes \IC_{\zeta})^r$ to $Z^{\nu,\nu',\eta,\lambda'}_{\tilde P,\tilde G}$. As in (\cite{BG}, Section~4.3.8) one shows the following.\\
a) The complex $K^{\nu,\nu', \eta,\lambda'}$ is placed in perverse degrees $\le 0$, and the inequality is strict unless $\nu'=0$, $\lambda'=\lambda$ and $\nu=\eta$.\\
b) The $*$-restriction of $K^{\nu, 0,\nu,\lambda}$ to $_{x,\nu}\Bunt_{\tilde P}-{_{x,\nu}\Bun_{\tilde P}}$ is placed in structly negative perverse degrees.\\
c) The $0$-th perverse cohomology sheaf of $K^{\nu, 0,\nu,\lambda}$ over $_{x,\nu}\Bun_{\tilde P}$ identifies canonically with $\IC_{\nu,\zeta}\otimes\Hom_{\check{M}_n}(U^{\nu}, V^{\lambda})$. 

 Point c) here uses Corollary~\ref{Cor_2}. We are done.

\subsection{Description of $\IC_{\zeta}$} 
\label{Section_Description of IC_zeta}
In this section we give a description of $\IC_{\zeta}$ generalizing the main result of \cite{ICDC} to our twisted setting.

 We freely use some notations of \cite{ICDC}. In particular, $\Lambda_{G,P}^{pos}$ is the $\ZZ_+$-span of $\{\alpha_i, i\in \cI-\cI_M\}$ in $\Lambda_{G,P}$. If $\theta\in\Lambda_{G,P}^{pos}$ is the projection under $\Lambda\to\Lambda_{G,P}$ of $\tilde\theta\in Span(\alpha_j), j\in \cI-\cI_M$ then $\flat(\theta)=w_0^M(\tilde\theta)$. Here $w_0^M$ is the longest element of the Weyl group of $M$.
For $\theta\in \Lambda_{G,P}^{pos}$ the scheme $\Gr_M^{+,\theta}$ is defined in (\cite{ICDC}, Proposition~1.7).  

 Let $\check{\gu}_n$ denote the Lie algebra of the unipotent radical of the Borel subgroup $\check{B}_n\subset \check{G}_n$. More generally, let $\check{\gu}_n(P)$ denote the Lie algebra of the unipotent radical of the standard parabolic $\check{P}_n\subset \check{G}_n$ corresponding to $\cI_M\subset\cI$. 

 For $\nu\in C^*(\check{M}_n)$ and $V\in\Rep(\check{M}_n)$ write $V_{\nu}$ for the direct summand of $V$, on which $Z(\check{M}_n)$ acts by $\nu$. In particular, we have the $\check{M}_n$-module $(\check{\gu}_n(P))_{\nu}$. 
\begin{Lm} If $\nu\in C^*(\check{M}_n)$ and $(\check{\gu}_n(P))_{\nu}$ is not zero then it is an irreducible $\check{M}_n$-module. 
\end{Lm}
\begin{Prf} Each root space of $\check{T}_n$ in $\check{\gu}_n(P)$ is 1-dimensional, so our claim follows from Lemma~\ref{Lm_common_weight} below.
\end{Prf}

\begin{Lm} 
\label{Lm_common_weight}
Let $H$ be a connected reductive group over $k$, $\check{H}$ be the Langlands dual over $\Qlb$. Let $\nu_1,\nu_2$ be dominant coweights of $H$ such that $\nu_1=\nu_2$ in $\pi_1(H)$. Then the irreducible $\check{H}$-representations $V^{\nu_1}$, $V^{\nu_2}$ with highest weights $\nu_1,\nu_2$ admit a common weight.
\end{Lm}
\begin{Prf} Let $\theta$ be the image of $\nu_i$ in $\pi_1(H)$. If $\theta=0$ then they both admit the zero weight space. Assume $\theta\ne 0$. Let $\mu_i$ be a dominant coweight of $H$ satisfying $\mu_i\le \nu_i$ and minimal with this property. Then the orbit $\Gr_H^{\mu_i}$ in closed in the connected component $\Gr_H^{\theta}$ of the affine grassmanian of $H$. Since $\Gr_H^{\theta}$ admits a unique closed $H(\bO)$-orbit, $\mu_1=\mu_2$.
\end{Prf}

\medskip

 Recall the functor $\Loc: \Rep(\check{M}_n)\,\iso\, \PPerv^{\natural}_{M,G,n}$ from Section~\ref{Section_Hecke functors for M}. 

\begin{Lm} Let $\nu\in\Lambda^{\sharp,+}_M$ be such that the irreducible $\check{M}_n$-module $U^{\nu}$ appears in $\check{\gu}_n(P)$, let $\theta$ be the image of $\nu$ in $\Lambda^{pos}_{G,P}$. Then $\Loc(U^{\nu})$ over $\wt\Gr_M$ is the extension by zero from $\wt\Gr_M^{+,\theta}$.
\end{Lm}
\begin{Prf} Note that $\nu$ lies in $\ZZ_+$-span of positive coroots of $G$. Let $\tilde\theta$ be the unique element in $\ZZ_+$-span of $\{\alpha_i\mid  i\in \cI-\cI_M\}$ such that $\tilde\theta=\nu$ in $\Lambda_{G,P}$. So, $\nu=\tilde\theta+\nu_1$, where $\nu_1$ is in $\ZZ_+$-span of positive coroots of $M$. Now $w_0^M(\nu)$ is a positive root of $\check{G}_n$, and 
$w_0^M(\nu)=\nu$ in $\Lambda_{G,P}$. So, $\tilde\theta\le_M \, w_0^M(\nu)$. This implies $\nu\le_M \, w_0^M(\tilde\theta)=\flat(\theta)$. 
\end{Prf}

\medskip

Set 
$$
J=\{\nu\in C^*(\check{M}_n)\mid (\check{\gu}_n(P))_{\nu}\ne 0\}
$$  
\begin{Lm} The restriction of $c_P: C^*(\check{M}_n)\to \Lambda^{\sharp}_{G,P}$ to $J$ is injective.
\end{Lm}  
\begin{Prf} Let $C^*_r(\check{M}_n)\subset C^*(\check{M}_n)$ be the subgroup generated by roots of $\check{G}_n$. It is a free abelian group, so the intersection of $C^*_r(\check{M}_n)$ with the kernel of $c_P$ is $\{0\}$. The restriction of $c_P$ to $C^*_r(\check{M}_n)$ is injective, and one has $J\subset C^*_r(\check{M}_n)$.
\end{Prf}

\medskip

 Set $\Lambda_{G,P}^{\sharp, \, pos}=\Lambda_{G,P}^{\sharp}\cap \Lambda_{G,P}^{pos}$.  For $\theta\in\Lambda^{\sharp}_{G,P}$ and $V\in \Rep(\check{M}_n)$ set 
$$
V_{\theta}=\mathop{\oplus}\limits_{\nu\in C^*(\check{M}_n), \, c_P(\nu)=\theta}\;  V_{\nu}
$$ 
 
\begin{Rem} For $i\in \cI$ let $\delta_i$ denote the denominator of $\frac{\iota(\alpha_i,\alpha_i)}{2n}$. Recall that $\delta_i\alpha_i$ is the corresponding simple root of $\check{G}_n$. If $i\in \cI-\cI_M$ then $\delta_i \alpha_i\in J$. The set $J$ may contain other elements also.
\end{Rem}

\subsubsection{} 
\label{section_11}
 
Given $\theta\in \Lambda_{G,P}^{pos}$, let $\gU(\theta)$ be a decomposition of $\theta$ as in (\cite{ICDC}, Section~1.4), recall the isomorphism of (\cite{ICDC}, Proposition~1.9)
\begin{equation}
\label{stack_gU(theta)_stratum}
_{\gU(\theta)}\Bunt_P\,\iso\, \Bun_P\times_{\Bun_M} \cH^{+,\gU(\theta)}_M
\end{equation}
Let $_{\gU(\theta)}\Bunt_{\tilde P}$ be obtained from $_{\gU(\theta)}\Bunt_P$ by the base change $\Bunt_{\tilde P}\to \Bunt_P$. We will describe the $*$-restriction of $\IC_{\zeta}$ under $_{\gU(\theta)}\Bunt_{\tilde P}\hook{} \Bunt_{\tilde P}$. 

 Recall that if $\gU(\theta)$ is the decomposition 
\begin{equation}
\label{decomposition_gU_theta}  
  \theta=\sum_m n_m\theta_m
\end{equation}  
then $\mid\! \gU(\theta)\!\mid=\sum_m n_m$, $X^{\gU(\theta)}=\prod_m X^{(n_m)}$, and $\oX{}^{\gU(\theta)}$ is the complement to all the diagonals in $X^{\gU(\theta)}$. Recall that $X^{\theta}$ is the scheme of $\Lambda_{G,P}^{pos}$-valued divisors of degree $\theta$, and $\oX{}^{\gU(\theta)}\subset X^{\theta}$ is locally closed. Set $^{uns}\oX{}^{\gU(\theta)}=X^{\gU(\theta)}-\vartriangle$, where $\vartriangle$ is the divisor of all the diagonals. Here $uns$ stands for `unsymmetrized'. 
 
  The stack $\cH^{+,\gU(\theta)}_M$ classifies $D\in \oX{}^{\gU(\theta)}$, $\cF_M, \cF'_M$, an isomorphism $\beta_M: \cF_M\,\iso\, \cF'_M\mid_{X-D}$ such that for each $\cV\in\Rep(G)$ the induced map 
$$ 
\beta_M^{\cV}: \cV^{U(P)}_{\cF_M}\hook{} \cV^{U(P)}_{\cF'_M}
$$
is an inclusion, and $\beta_M$ induces an isomorphism 
$$
\cF_{M/[M,M]}\,\iso\, \cF'_{M/[M,M]}(-D)
$$ 
Note that the Plucker relations for $\beta_M^{\cV}$ hold automatically. So, here we think of $\cF'_M$ as the `background' $M$-torsor. The stack (\ref{stack_gU(theta)_stratum}) classifies the same data together with a $P$-torsor $\cF'_P$ and an isomorphism $\cF'_P\times_P M\,\iso\, \cF'_M$.

Let $\cH^{+,\gU(\theta)}_{\tilde M}$ be obtained from $\cH^{+,\gU(\theta)}_M$ by the base change $\Bunt_M\times\Bunt_M\to \Bun_M\times\Bun_M$. So, it classifies the same data together with lines $\cU, \cU'$ equipped with 
$$ 
\cU^N\,\iso\, (\cL_M)_{\cF_M}, \;\;\;\; \cU'^N\,\iso\, (\cL_M)_{\cF'_M}
$$  

We get 
\begin{equation}
\label{iso_stratum_gU(theta)Bunt_tildeP}
_{\gU(\theta)}\Bunt_{\tilde P}\,\iso\, \Bun_P\times_{\Bun_M} \cH^{+,\gU(\theta)}_{\tilde M}
\end{equation}
Note that $(a,a')\in \mu_N(k)\times\mu_N(k)\subset \Aut(\cU)\times \Aut(\cU')$ acts on $\IC_{\zeta}\mid_{_{\gU(\theta)}\Bunt_{\tilde P}}$ as $\zeta(\frac{a'}{a})$. 

Let $\Gr_M^{+, \gU(\theta)}$ be as in \cite{ICDC}, so it is obtained from $\cH^{+,\gU(\theta)}_M$ by the base change $\Spec k\to \Bun_M$ trivializing the $M$-torsor $\cF'_M$.  We also denote by $\cL_X$ the line bundle on $\Gr_M^{+, \gU(\theta)}$ whose fibre at $(\cF_M, \beta_M, D)$ is 
$$
\det\RG(X, \gg\otimes\cO)\otimes\det\RG(X, \gg_{\cF_M})^{-1}
$$ 
Let $\wt\Gr_M^{+, \gU(\theta)}$ be the gerb of $N$-th roots of $\cL_X$ over $\Gr_M^{+, \gU(\theta)}$. 

 As in (\cite{BG}, Section~6.2.3) we set $\Lambda^+_{M,G}=\Lambda^+_M\cap w_0^M(\Lambda_G^{pos})$. Say that $V\in \Rep(\check{M}_n)$ is \select{positive} if $\Loc(V)$ is the extension by zero from $\wt\Gr_M^+=\cup_{\theta\in\Lambda_{G,P}^{pos}} \; \wt\Gr_M^{+,\theta}$. In this case it is actually the extension by zero from $\cup_{\theta\in\Lambda_{G,P}^{\sharp, pos}}\; \wt\Gr_M^{+,\theta}$.
 
 In fact, $V\in \Rep(\check{M}_n)$ is positive if and only if for any irreducible subrepresentation $U^{\nu}\subset V$, $\nu\in\Lambda^{\sharp, +}_M$ there is $\mu\in\Lambda^+_{M,G}$ such that $\nu\le_M \mu$. 
 
  Let $^{uns}\wt\Gr_M^{+,\gU(\theta)}$ be obtained from $\wt\Gr_M^{+, \gU(\theta)}$ by the base change $^{uns}\oX{}^{\gU(\theta)}\to \oX{}^{\gU(\theta)}$.
 
 For $V\in \Rep(\check{M}_n)$ positive we denote by $\Loc^{\gU(\theta)}_{\zeta}(V)$ the perverse sheaf on $\wt\Gr_M^{+, \gU(\theta)}$, on which $\mu_N(k)$ acts by $\zeta$ and such that for $D=\sum\theta_k x_k\in \oX{}^{\cU(\theta)}$ its restriction to
\begin{equation}
\label{formula_product_of_Gr^+theta}
\prod_k \wt\Gr^{+,\theta_k}_{M, x_k}
\end{equation}
is
$$
(\boxtimes_k \Loc_{\zeta}(V_{\theta_k}))\otimes\Qlb[\mid\! \gU(\theta)\!\mid]
$$
To make this definition rigorous, we first define $\Loc_{\zeta}^{uns, \gU(\theta)}(V)$ on $^{uns}\wt\Gr_M^{+,\gU(\theta)}$. Write $^{uns}\wt{\wt\Gr}{}_M^{+,\gU(\theta)}$ for the stack over $^{uns}\oX{}^{\gU(\theta)}$ whose fibre over $D\in {^{uns}\oX{}^{\gU(\theta)}}$ is (\ref{formula_product_of_Gr^+theta}). The desired perverse sheaf is obtained by descent via the gerb
$$
^{uns}\wt{\wt\Gr}{}_M^{+,\gU(\theta)}\to {^{uns}\wt\Gr_M^{+,\gU(\theta)}}
$$

Over $^{uns}\wt{\wt\Gr}{}_M^{+,\gU(\theta)}$ this perverse sheaf is defined similarly using the fact (cf. \cite{FL}, Section~4.1.2) that every object of $\PPerv_{M,G,n}$ admits a unique $\Aut^0_2(\bO)$-equivariant structure. Here $\Aut^0_2(\bO)$ is a connected group scheme defined in (\cite{FL}, Section~2.3). More precisely, consider the torsor over $^{uns}\oX{}^{\gU(\theta)}$ whose fibre over $D$ is the set of isomorphisms $(\bO_x, \cE_X(\bO_x))\,\iso\, (\cO, \cE)$ for all $x\in D$. Here $\cE\in \Omega(\bO)^{\frac{1}{2}}$ is the object we picked in Section~\ref{section_Hecke_functors_for_tilde_G}. Then $^{uns}\wt{\wt\Gr}{}_M^{+,\gU(\theta)}$ is the twist of $\prod_k \wt\Gr^{+,\theta_k}_M$ by this torsor. So, the object $\Loc_{\zeta}^{uns, \gU(\theta)}(V)$ is well-defined, and moreover equivariant with respect to the Galois group of the covering $^{uns}\oX{}^{\gU(\theta)}\to \oX{}^{\gU(\theta)}$. Thus, it gives rise to the perverse sheaf $\Loc^{\gU(\theta)}_{\zeta}(V)$ on $\wt\Gr_M^{+, \gU(\theta)}$ defined up to a unique isomorphism. 

Note that $\Loc^{\gU(\theta)}_{\zeta}(V)$ vanishes unless in the decomposition (\ref{decomposition_gU_theta}) each $\theta_m$ lies in $\Lambda_{G,P}^{\sharp, \, pos}$. 

\subsubsection{} Let $\Bun_{M, \gU(\theta)}$ be the stack classifying $\cF_M\in\Bun_M$, $D\in \oX{}^{\cU(\theta)}$, and a trivialization of $\cF_M$ over the formal neighbourhood of $D$. Let $\Bunt_{M, \gU(\theta)}=\Bun_{M, \gU(\theta)}\times_{\Bun_M} \Bunt_M$. 
 
  Let $M_{\gU(\theta)}$ be the scheme classifying $D\in \oX{}^{\cU(\theta)}$ and a section of $M$ over the formal neighbourhood of $D$. This is a group scheme over $\oX{}^{\cU(\theta)}$. The group $M_{\gU(\theta)}$ acts diagonally on $\Bun_{M, \gU(\theta)}\times_{\oX{}^{\cU(\theta)}} \Gr_M^{+, \gU(\theta)}$, and the stack quotient is denoted
$$
\Bun_{M, \gU(\theta)}\times_{M_{\gU(\theta)}} \Gr_M^{+, \gU(\theta)} 
$$
There is an isomorphism $\gamma^{\ra}$ from the latter stack to $\cH^{+,\gU(\theta)}_M$ such that the projection to the first term corresponds to $h^{\ra}_M$. As above, one extends this $M_{\gU(\theta)}$-torsor to a $M_{\gU(\theta)}$-torsor
$$
\tilde \gamma^{\ra}: \Bunt_{M, \gU(\theta)}\times_{\oX{}^{\cU(\theta)}} \wt\Gr_M^{+, \gU(\theta)}\to \cH^{+,\gU(\theta)}_{\tilde M}
$$
So, for $\cT\in \D(\Bunt_M)$ and a $M_{\gU(\theta)}$-equivariant perverse sheaf $S$ on $\wt\Gr_M^{+, \gU(\theta)}$ we may form their twisted product $(\cT\tboxtimes S)^r$ on $\cH^{+,\gU(\theta)}_{\tilde M}$. For $V\in  \Rep(\check{M}_n)$ positive define 
\begin{equation}
\label{functor_Loc_third}
\Loc^{\gU(\theta)}_{\Bun_M, \zeta}(V)=(\IC(\Bunt_M)\tboxtimes \Loc^{\gU(\theta)}_{\zeta}(V))^r
\end{equation}
Similarly, applying for $\nu_P: \Bun_P\to\Bun_M$ the functor $\nu_P^*[\dimrel(\nu_P)]$ to (\ref{functor_Loc_third}), we get the perverse sheaf on (\ref{iso_stratum_gU(theta)Bunt_tildeP}) denoted
$$
\Loc^{\gU(\theta)}_{\Bun_P, \zeta}(V)
$$ 

\begin{Th} 
\label{Th_main_Section1}
The $*$-restriction of $\IC_{\zeta}$ under $_{\gU(\theta)}\Bunt_{\tilde P}\hook{} \Bunt_{\tilde P}$ vanishes unless in the decomposition (\ref{decomposition_gU_theta}) each $\theta_m$ lies in $\Lambda_{G,P}^{\sharp, \, pos}$. In the latter case it is isomorphic to
$$
\Loc^{\gU(\theta)}_{\Bun_P, \zeta^{-1}}(\;\mathop{\oplus}\limits_{i\ge 0} \Sym^i(\check{\gu}_n(P))[2i]\; ) \otimes \Qlb[-\mid\! \gU(\theta)\!\mid] 
$$
where $\mathop{\oplus}\limits_{i\ge 0} \Sym^i(\check{\gu}_n(P))[2i]$ is viewed as a cohomologically graded $\check{M}_n$-module.
\end{Th}

\section{Proof of Theorem~\ref{Th_main_Section1}}

\subsection{Zastava spaces} 
\label{Section_Zastava spaces}
Keep notations of Section~\ref{Section_Parabolic geometric Eisenstein series}. We also use some notations from \cite{ICDC}. For $\theta\in\Lambda_{G,P}^{pos}$ let $Z^{\theta}$ be as in \cite{ICDC}. Recall that $\Mod_M^{+,\theta}$ classifies $(D\in X^{\theta}, \cF_M\in\Bun_M, \beta_M)$, where $\beta_M: \cF_M\,\iso\, \cF^0_M\mid_{X-D}$ is an isomorphism such that for any $G$-module $\cV$, the map 
$$
\beta_M: \cV^{U(P)}_{\cF_M}\to \cV^{U(P)}_{\cF^0_M}
$$ 
is regular over $X$, and $\beta_M$ induces $\cF_{M/[M,M]}\,\iso\, \cF^0_{M/[M,M]}(-D)$. 

A point of $Z^{\theta}$ is given by 
\begin{equation}
\label{point_of_Z^theta}
(\cF_G,\cF_M,\beta,\beta_M, D),
\end{equation}
where $(\cF_M, \beta_M,D)\in \Mod_M^{+,\theta}$, and $\cF_G$ is a $G$-torsor on $X$ equipped with a trivialization $\beta: \cF_G\,\iso\, \cF_G^0\mid_{X-D}$ satisfying the conditions of (\cite{ICDC}, Section~2.2). We have a diagram $Z^{\theta}\toup{\pi_P} \Mod_M^{+,\theta}\to \Bun_M$, where the second map sends the above point to $\cF_M$.  By abuse of notations, the restrictions of $\cL_M$ under these maps are also denoted by $\cL_M$. Let 
$$
\wt Z^{\theta}\toup{\pi_P} \Mod_{\tilde M}^{+,\theta}\to \Bunt_M
$$ 
be obtained from the latter diagram by the base change $\Bunt_M\to\Bun_M$. A point of $\wt Z^{\theta}$ is given by (\ref{point_of_Z^theta}) together with a line $\cU$ equipped with $\cU^N\,\iso\, (\cL_M)_{\cF_M}$.

 The open subscheme $Z^{\theta}_{max}\subset Z^{\theta}$ is defined in (\cite{ICDC}, Section~2.2). Let $\wt Z^{\theta}_{max}=Z^{\theta}_{max}\times_{Z^{\theta}} \wt Z^{\theta}$. We have an isomorphism
$$
i_{max}: B(\mu_N)\times Z^{\theta}_{max}\,\iso\, \wt Z^{\theta}_{max}
$$
sending $\cU_0$, $\cU_0^N\,\iso\, k$ and $(\cF_G, \cF_M, \beta,\beta_M)$ to 
$(\cF_G, \cF_M, \beta,\beta_M,\cU)$, where $\cU=\cU_0^{-1}$ is equipped with the induced isomorphism
\begin{equation}
\label{iso_for_tilde_Z_max}
\cU^N\,\iso\, k\,\iso\, \cL_{\cF_G}\,\iso\, (\cL_M)_{\cF_M}
\end{equation}
Define $\IC_{Z^{\theta},\zeta}$ as the intermediate extension of $i_{max *}(\cL_{\zeta}\boxtimes \IC(Z^{\theta}_{max}))$ to $\tilde Z^{\theta}$. We underline that $a\in \mu_N(k)\subset\Aut(\cU)$ acts on $\IC_{Z^{\theta},\zeta}$ as $\zeta^{-1}(a)$. 

\subsubsection{Action of $M_{X^{\theta}}$} 
\label{section_122}
For $D\in X^{\theta}$ denote by $\bar D$ (resp., by $\bar D^0$) the formal (resp., the punctured formal) neighbourhood of $D$ in $X$. This means that we pick a homomorphism of semigroups $\Lambda_{G,P}^{pos}\to\ZZ_+$ sending each $\alpha_i$, $i\in \cI-\cI_M$ to a nonzero element, it yields a morphism $v: X^{\theta}\to X^{(d)}$, where $d$ is the image of $\theta$, and $\bar D$ is the formal neighbourhood of $v(D)$ in $X$. Similarly for $\bar D^0$.
 
 Let $M_{X^{\theta}}$ be the scheme classifying $D\in X^{\theta}$ and a section of $M$ over the formal neighbourhood of $D$ in $X$. 

  The space $\Mod^{+,\theta}_M$ can be rewritten as the space classifying $D\in X^{\theta}$, a $M$-torsor $\cF_M$ on $\bar D$, its trivialization $\beta_M: \cF_M\,\iso\, \cF^0_M\mid_{\bar D^0}$ such that for each representation $\cV$ of $M$ the map $\beta_M: \cV^{U(P)}_{\cF_M}\to \cV^{U(P)}_{\cF_M^0}$ is regular over $\bar D$, and $\beta_M$ induces an isomorphism $\cF_{M/[M,M]}\,\iso\, \cF^0_{M/[M,M]}(-D)$ over $\bar D$. 
In this incarnation $M_{X^{\theta}}$ acts on $\Mod^{+,\theta}_M$ over $X^{\theta}$ by changing the trivialization $\beta_M$. 

 Similarly, $Z^{\theta}$ can be seen as the scheme classifying $D\in X^{\theta}$, a $M$-torsor $\cF_M$ on $\bar D$, its trivialization $\beta_M: \cF_M\,\iso\, \cF^0_M\mid_{\bar D^0}$ such that for each representation $\cV$ of $M$ the map $\beta_M: \cV^{U(P)}_{\cF_M}\to \cV^{U(P)}_{\cF_M^0}$ is regular over $\bar D$, and $\beta_M$ induces an isomorphism $\cF_{M/[M,M]}\,\iso\, \cF^0_{M/[M,M]}(-D)$ over $\bar D$; a $G$-torsor $\cF_G$ over $\bar D$, an isomorphism $\beta: \cF_G\,\iso\, \cF^0_G\mid_{\bar D^0}$ such that for each $G$-module $V$ the map
$$
V_{\cF_G}\toup{\beta} V_{\cF^0_G}\to (V_{U(P^-)})_{\cF^0_M}
$$
is regular and surjective over $\bar D$, and the map
$$
V^{U(P)}_{\cF_M}\toup{\beta_M} V^{U(P)}_{\cF^0_M}\hook{} V_{\cF^0_G}\toup{\beta^{-1}} V_{\cF_G}
$$
is regular over $\bar D$. In this incarnation $M_{X^{\theta}}$ acts on $Z^{\theta}$ via its action on $\cF^0_M$. Namely, if $g$ is automorphism of $\cF^0_M$ over $\bar D$, it sends the above point to the collection $(\cF_G,\cF_M, g\beta_M, g\beta)$.  

 The line bundle $\cL_M$ on $Z^{\theta}$ is naturally $M_{X^{\theta}}$-equivariant. Namely, the fibre of $\cL_M$ at $(\cF_G,\cF_M, \beta_M, \beta)$ is $\det\RG(X, \gg_{\cF^0_M})\otimes \det\RG(X, \gg_{\cF_M})^{-1}$, and $M_{X^{\theta}}$ acts via its action on $\cF^0_M$. So, $M_{X^{\theta}}$ acts on $\Mod^{+,\theta}_{\tilde M}$ and on $\wt Z^{\theta}$, and the maps $\gs^{\theta}: \Mod^{+,\theta}_{\tilde M}\to \wt Z^{\theta}$ and $\pi_P: \wt Z^{\theta}\to \Mod^{+,\theta}_{\tilde M}$ are $M_{X^{\theta}}$-equivariant. Note that $\IC_{Z^{\theta},\zeta}$ is naturally $M_{X^{\theta}}$-equivariant.

\subsubsection{The relation between $\IC_{Z^{\theta},\zeta}$ and $\IC_{\zeta}$} 
\label{section_relation between_IC_Ztheta_and_IC_zeta}
Write $\Bun_{M,X^{\theta}}$ for the stack classifying $\cF_M\in\Bun_M, D\in X^{\theta}$ and a trivialization of $\cF_M$ over the formal neighbourhood of $D$ in $X$. Let $\Bunt_{M,X^{\theta}}$ be obtained from it by the base change $\Bunt_M\to\Bun_M$. 

 Recall that $Z^{\theta}_{\Bun_M}$ is defined as $Z^{\theta}$ by replacing $\cF_M^0$ by the `background' $M$-torsor $\cF'_M\in\Bun_M$. Let $\wt Z^{\theta}_{\Bunt_M}$ be obtained from $Z^{\theta}_{\Bun_M}$ by adding lines $\cU,\cU'$ equipped with isomorphisms $\cU^N\,\iso\, (\cL_M)_{\cF_M}$, $\cU'^n\,\iso\, (\cL_M)_{\cF'_M}$, where $\cF'_M$ is the background $M$-torsor. 
   
   Let $M_{X^{\theta}}$ act diagonally on $\wt Z^{\theta}\times_{X^{\theta}} \Bunt_{M, X^{\theta}}$. As n Section~\ref{section_Hecke_functors_for_tilde_G}, we have a $M_{X^{\theta}}$-torsor
$$
\gamma_Z: \wt Z^{\theta}\times_{X^{\theta}} \Bunt_{M, X^{\theta}}\to \wt Z^{\theta}_{\Bunt_M}
$$   
We form the twisted external product 
\begin{equation}
\label{complex_gamma_Z_direct_image}
(\IC_{Z^{\theta},\zeta}\tboxtimes \IC(\Bunt_M))
\end{equation}
on $\wt Z^{\theta}_{\Bunt_M}$, which is the descent with respect to $\gamma_Z$.

 Let $\Bun_{P^-}^r$ be defined as in (\cite{ICDC}, Section~3.6) so that $\Bun_{P^-}^r\to\Bun_G$ is smooth. By (\cite{ICDC}, Propostion~3.2), $Z^{\theta}_{\Bun_M}\subset \Bunt_P\times_{\Bun_G}\Bun_{P^-}$ is open. This extends naturally to an open immersion
$$
\wt Z^{\theta}_{\Bunt_M}\hook{} \Bunt_{\tilde P}\times_{\Bun_G}\Bun_{P^-}
$$  
The restriction of $\pr_1^*(\IC_{\zeta})[\dimrel(\pr_1)]$ under this open immersion identifies with (\ref{complex_gamma_Z_direct_image}) over the intersection with 
$$
\Bunt_{\tilde P}\times_{\Bun_G}\Bun_{P^-}^r
$$ 
So, as in \cite{ICDC}, $\IC_{Z^{\theta},\zeta}$ is a local model for $\IC_{\zeta}$. 

\subsubsection{} The natural extensions of $\pi_P$ and $\gs^{\theta}$ are still denoted $\gs^{\theta}: \Mod^{+,\theta}_{\tilde M}\to \wt Z^{\theta}$ and $\pi_P: \wt Z^{\theta}\to \Mod^{+,\theta}_{\tilde M}$.  
 
 If we pick a $\Gm$-action on $Z^{\theta}$ as in (\cite{ICDC}, Section~5.3) then the line bundle $\cL_M$ and its trivialization over $Z^{\theta}_{max}$ are $\Gm$-equivariant, as $\Gm$ is a subgroup in $M_{X^{\theta}}$. So, $\IC_{Z^{\theta},\zeta}$ is $\Gm$-equivariant, and the analog of (\cite{ICDC}, Proposition~5.2) holds, there is a canonical isomorphism 
$$
\gs^{\theta !}(\IC_{Z^{\theta},\zeta})\,\iso\, \pi_{P !}(\IC_{Z^{\theta},\zeta})
$$ 
over $\Mod^{+,\theta}_{\tilde M}$. Since all our objects are already defined over a suitable finite subfield of $k$, the analog of (\cite{ICDC}, Corollary~5.5) holds, the latter complex is a direct sum of shifted perverse sheaves.

 Recall that for a fixed $x\in X$ one denotes by $\SSS^{\theta}$ the preimage of $\Spec k\toup{x} X\to X^{\theta}$ under $\pi_G: Z^{\theta}\to X^{\theta}$. The corresponding preimage under $\pi_G: \wt Z^{\theta}\to X^{\theta}$ is denoted $\wt \SSS^{\theta}$. 
 
 For $\theta'\in\Lambda_{G,P}^{pos}$ with $\theta-\theta'\in\Lambda_{G,P}^{pos}$ one has the locally closed subschemes $_{\theta'}Z^{\theta}\hook{}Z^{\theta}$ and $_{\theta'}\SSS^{\theta}\hook{}\SSS^{\theta}$ defined in (\cite{ICDC}, Section~3.5). Restricting the gerb $\tilde Z^{\theta}$, one gets the locally closed substacks $_{\theta'}\wt Z^{\theta}\hook{} \wt Z^{\theta}$ and $_{\theta'}\wt \SSS^{\theta}\hook{} \wt \SSS^{\theta}$. 
 
\subsubsection{} Recall that $\Gr_M^{\theta}$ is the connected component containing $t^{\nu}M(\bO)$ for any $\nu\in\Lambda$ over $\theta$. By virtue of (\cite{ICDC}, Propostion~2.6) for $\theta\in\Lambda_{G,P}^{pos}$ the isomorphism $i_{max}$ restricts to an isomorphism
$$
i_{max}^x: B(\mu_N)\times\Gr_P^{\theta}\cap \Gr_{U(P^-)}\,\iso\,{_0\wt\SSS^{\theta}}
$$
sending $\cU_0$, $\cU_0^N\,\iso\, k$, $(\cF_G,\cF_M, \beta,\beta_M)$ to $(\cF_G,\cF_M, \beta,\beta_M, \cU)$, where $\cU=\cU_0^{-1}$ is equipped with the induced isomorphism (\ref{iso_for_tilde_Z_max}). 
The map $\pi_P$ restricts to a morphism
$$
\tilde\gt^{\theta}: {_0\wt\SSS^{\theta}}\to \wt\Gr_M^{+,\theta}
$$
sending $(\cF_G,\cF_M, \beta,\beta_M, \cU)$ to $(\cF_M, \beta_M, \cU)$. 
First, we prove the following analog of (\cite{ICDC}, Theorem~5.9).

\begin{Th} 
\label{Th_2}
1) For $\theta\in\Lambda_{G,P}^{pos}$ the complex 
\begin{equation}
\label{complex_for_Th2}
\tilde\gt^{\theta}_!(i^x_{max})_*(\cL_{\zeta}\boxtimes\Qlb)
\end{equation} 
is placed in perverse degrees $\le \<\theta, 2(\check{\rho}-\check{\rho}_M)\>$.\\
2) Its $\<\theta, 2(\check{\rho}-\check{\rho}_M)\>$-th perverse cohomology sheaf 
vanishes unless $\theta\in \Lambda^{\sharp}_{G,P}$. In the latter case it identifies canonically with $\Loc_{{\zeta}^{-1}}(U(\check{\gu}_n(P))_{\theta})$. 
\end{Th}

 The map $\tilde\gt^{\theta}$ is $M(\bO_x)$-equivariant, so each perverse cohomology sheaf of (\ref{complex_for_Th2}) is of the form $\Loc_{\zeta^{-1}}(V)$ for some $V\in\Rep(\check{M}_n)$. 

\subsubsection{Proof of Theorem~\ref{Th_2} for $P=B$} By (\cite{ICDC}, Section~6.3), for any $\nu\in \Lambda$ one has $\dim(\Gr_B^{\nu}\times \Gr_{U(B^-)})\le \<\nu, \check{\rho}\>$. This implies part 1). Moreover, (\ref{complex_for_Th2}) vanishes unless $\theta\in \Lambda^{\sharp}$ because of the description of $\PPerv_{T,G,n}$.  

 Assume $\mu\in \Lambda^+$ is deep enough in the dominant chamber so that $\Gr_B^{\nu-\mu}\cap \Gr_{B^-}^{-\mu}\subset \Gr_G^{-w_0(\mu)}$ by (\cite{ICDC}, Proposition~6.4). By \select{loc.cit.}, the inclusion 
$$
a: \wt\Gr_B^{\nu-\mu}\cap \wt\Gr_{B^-}^{-\mu}\subset \wt\Gr_G^{-w_0(\mu)}\cap \wt\Gr_B^{\nu-\mu}
$$ 
yields a bijection between the irreducible components of dimension $\<\nu, \check{\rho}\>$ of both schemes. Recall that the multiplication by $t^{\mu}$ gives an isomorphism 
$$
\Gr_B^{\nu-\mu}\cap \Gr_{B^-}^{-\mu}\,\iso\, \Gr_B^{\nu}\cap \Gr_{B^-}^0
$$ 
Recall that $\det(\gg(\bO): \gg(\bO)^{t_x^{\mu}})\,\iso\, \Omega_{\bar c}^{\check{h}\iota(\mu,\mu)}$. Assume in addition that $\mu\in 2\Lambda^{\sharp}$. In this case we get a $T(\bO)$-equivariant diagram
$$
\begin{array}{ccc}
\wt\Gr_B^{\nu-\mu}\cap \wt\Gr_{B^-}^{-\mu} & \toup{b} & \wt\Gr_B^{\nu}\cap \wt\Gr_{B^-}^0\\
\downarrow\lefteqn{\scriptstyle\tilde\gt} && \downarrow\lefteqn{\scriptstyle\tilde\gt}\\
\wt\Gr_T^{\nu-\mu} & \toup{b} & \wt\Gr_T^{\nu},
\end{array}
$$ 
where for the top row $b$ is the isomorphism sending $(\cU, \cU^N\,\iso\, \det(\gg(\bO_x): \gg(\bO)^g), \; gG(\bO_x))$ to $(\cU', t^{\mu}gG(\bO_x))$, where $\cU'=\cU\otimes \Omega_{\bar c}^{\frac{\iota(\mu,\mu)}{2n}}$ is equipped with the induced isomorphism
$$
\cU'^N\,\iso\, \det(\gg(\bO): \gg(\bO)^{t_x^{\mu}})\otimes \det(\gg(\bO): \gg(\bO)^g)\,\iso\, \det(\gg(\bO): \gg(\bO)^{t_x^{\mu}g})
$$
For the low row $b$ is defined similarly. Using Lemma~\ref{Lm_for_canonicity_Lm7} below, one gets canonically 
$$
(i^x_{max})^*b_*a^*\cA^{-w_0(\mu)}_{\cE}[-\<\mu, 2\check{\rho}\>]\,\iso\, (\cL_{\zeta}^*\boxtimes\Qlb)
$$
From Proposition~\ref{Pp_description_weight_space} we see that 
$$
\tilde \gt^{\nu-\mu}_!(\gs^{\nu-\mu})^*\cA^{-w_0(\mu)}_{\cE}[-\<\mu, 2\check{\rho}\>]
$$ 
identifies with $\Loc(V^{-w_0(\mu)}(\nu-\mu))[-\<\nu, 2\check{\rho}\>]$. If $\mu$ is large enough in the dominant chamber compared to $\nu$ then the latter identifies with $U(\check{\gu}_n)_{\nu}[-\<\nu, 2\check{\rho}\>]$. Here $V^{\mu_1}(\mu_2)$ denotes $\mu_2$-weight space of $\check{T}^{\sharp}$ in the irreducible representation $V^{\mu_1}$ of $\check{G}_n$ with highest weight $\mu_1$. 

\subsubsection{Proof of Theorem~\ref{Th_2} for general $P$} Let $\theta\in \Lambda_{G,P}^{pos}$. For $\nu\in\Lambda$ dominant for $M$ with $\Gr_M^{\nu}\subset \Gr_M^{+,\theta}$ consider the map $\tilde\gt^{\nu}: \wt\Gr_P^{\nu}\cap  \wt\Gr_{U(P^-)}\to\wt\Gr_M^{\nu}$. It suffices to prove that for each such $\nu$ the complex
$$
\tilde\gt_!^{\nu}(i^x_{max})_*(\cL_{\zeta}\boxtimes\Qlb)\mid_{\wt\Gr_P^{\nu}\cap  \wt\Gr_{U(P^-)}}
$$
is placed in perverse degrees $\le \<\nu, 2\check{\rho}-2\check{\rho}_M\>$, and its $\<\nu, 2\check{\rho}-2\check{\rho}_M\>$-th perverse cohomology sheaf vanishes unless $\nu\in \Lambda^{\sharp,+}_M$, and in the latter case identifies with 
$$
\cA^{\nu}_{M,\cE, \zeta^{-1}}\otimes \Hom_{\check{M}_n}(U^{\nu}, U(\check{\gu}_n(P)))
$$

Pick $\nu\in\Lambda$ dominant for $M$ with $\Gr_M^{\nu}\subset \Gr_M^{+,\theta}$. We have a diagram
$$
\begin{array}{ccc}
\wt\Gr_P^{\theta}\cap \wt\Gr_{U(P^-)} & \toup{\tilde\gt^{\theta}} & \wt\Gr_M^{+,\theta}\\
\uparrow && \uparrow\\
\wt\Gr_P^{\nu}\cap  \wt\Gr_{U(P^-)} & \toup{\tilde\gt^{\nu}} & \wt\Gr_M^{\nu},
\end{array}
$$
where vertical arrows are natural closed immersions. Let $Z(M)^0$ be the connected component of unity of the center $Z(M)$ of $M$. Pick $\mu\in \Lambda$ satisfying $\<\mu,\check{\alpha}_i\>=0$ for $i\in \cI_M$, and $\<\mu,\check{\alpha}_i\>$ positive and large enough compared to $\nu$ for $i\in \cI-\cI_M$. So, $\mu\in \Hom(\Gm, Z(M)^0)$. The multiplication by $t^{\mu}$ yields a diagram, where the horizontal arrows are isomorphisms
\begin{equation}
\label{diag_for_Th_2_general_P}
\begin{array}{ccc}
\Gr_P^{\nu-\mu}\cap \Gr_{P^-}^{-\mu} & \iso & \Gr_P^{\nu}\cap \Gr_{U(P^-)}\\
\downarrow\lefteqn{\scriptstyle \gt^{\nu-\mu}} && \downarrow\lefteqn{\scriptstyle \gt^{\nu}}\\
\Gr_M^{\nu-\mu} & \iso & \Gr_M^{\nu}
\end{array}
\end{equation}
We may and do assume by (\cite{ICDC}, Proposition~6.6) that $\Gr_P^{\nu-\mu}\cap \Gr_{P^-}^{-\mu}\subset \Gr_G^{-w_0(\mu)}$. 

 Assume in addition that $\mu\in 2\Lambda^{\sharp}$. Then (\ref{diag_for_Th_2_general_P}) extends to a diagram of $M(\bO)$-equivariant morphisms
\begin{equation}
\label{diag_for_general_P_third}
\begin{array}{ccc}
\wt\Gr_P^{\nu-\mu}\cap \wt\Gr_{P^-}^{-\mu} & \toup{b} & \wt\Gr_P^{\nu}\cap \wt\Gr_{U(P^-)}\\
\downarrow\lefteqn{\scriptstyle \wt\gt^{\nu-\mu}} && \downarrow\lefteqn{\scriptstyle \wt\gt^{\nu}}\\
\wt\Gr_M^{\nu-\mu} & \toup{b} & \wt\Gr_M^{\nu},
\end{array} 
\end{equation}
where for the top row $b$ is the isomorphism sending $(\cU, \cU^N\,\iso\, \det(\gg(\bO_x): \gg(\bO)^g), \; gG(\bO_x))$ to $(\cU', t^{\mu}gG(\bO_x))$, where $\cU'=\cU\otimes \Omega_{\bar c}^{\frac{\iota(\mu,\mu)}{2n}}$ is equipped with the induced isomorphism
$$
\cU'^N\,\iso\, \det(\gg(\bO): \gg(\bO)^{t_x^{\mu}})\otimes \det(\gg(\bO): \gg(\bO)^g)\,\iso\, \det(\gg(\bO): \gg(\bO)^{t_x^{\mu}g})
$$
For the low row $b$ is defined similarly. 

 
 Consider the inclusion $a: \wt\Gr_P^{\nu-\mu}\cap \wt\Gr_{P^-}^{-\mu}\hook{} \wt\Gr_P^{\nu-\mu}\cap \wt\Gr_G^{-w_0(\mu)}$. Using Lemma~\ref{Lm_for_canonicity_Lm7} below, one gets a canonical isomorphism
$$
b_*a^*\cA_{\cE}^{-w_0(\mu)}[-\<\mu, 2\check{\rho}\>]\,\iso\, (i^x_{max})_*(\cL_{\zeta}^*\boxtimes\Qlb)
$$
By (\cite{BG}, Proposition~4.3.3), the fibres of the left vertical arrow $\tilde\gt^{\nu-\mu}$ in (\ref{diag_for_general_P_third}) are of dimension $\le \<\nu, \check{\rho}-2\check{\rho}_M\>$, so $\tilde\gt^{\nu-\mu}_!(a^*E^{-w_0(\mu)}_{\cE})$ is placed in usual cohomological degrees $\le \<\nu, 2\check{\rho}-4\check{\rho}_M\>$, and this complex is $M(\bO)$-equivariant. So, it is placed in perverse degrees $\le \<\nu, 2\check{\rho}-2\check{\rho}_M\>$. The natural map
$$
\tilde\gt^{\nu-\mu}_!(E^{-w_0(\mu)}_{\cE})\to \tilde\gt^{\nu-\mu}_!(a_*a^*E^{-w_0(\mu)}_{\cE})
$$
induces an isomorphism between the corresponding $\<\nu, 2\check{\rho}-2\check{\rho}_M\>$-th perverse cohomology sheaves over $\wt\Gr_M^{\nu-\mu}$, which identify by Corollary~\ref{Cor_2} with $\cA^{\nu-\mu}_{M,\cE}\otimes \Hom_{\check{M}_n}(U^{\nu-\mu}, V^{-w_0(\mu)})$ for $\nu\in \Lambda^{\sharp,+}_M$ and vanish otherwise. Assume $\nu\in \Lambda^{\sharp,+}_M$. Since $\mu$ is large enough on the corresponding wall of the Weyl chamber compared to $\nu$, 
we have 
$$
\Hom_{\check{M}_n}(U^{\nu-\mu}, V^{-w_0(\mu)})\,\iso\, \Hom_{\check{M}_n}(U^{\nu}, U(\check{\gu}_n(P)))
$$ 
Theorem~\ref{Th_2} is proved. $\square$

\bigskip

For $\mu\in \Lambda^{\sharp,+}$ we have the line bundle $\cL_{\mu,\cE}$ on $\Gr_G^{\mu}$ defined in (\cite{FL}, Section~2.1, p. 723). Its analog for the Levi $M$ is the following line bundle $\cL_{\mu, M, \cE}$. For $\mu\in \Lambda^{\sharp,+}_M$ let $\cB^{\mu}_M=M/P_M^{\mu}$ be the $M$-orbit through $t^{\mu}M(\bO)$ in $\Gr_M$ as in (\cite{FL}, Section~4.1.1). Denote by $\tilde\omega_{M,\mu}: \Gr^{\mu}_M\to \cB^{\mu}_M$ the projection. Set 
$$
\cL_{\mu, M, \cE}=\cE_{\bar c}^{\iota(\mu,\mu)/n}\otimes \tilde\omega_{M,\mu}\cO(\iota(\mu)/n)
$$ 
Note that for $\mu\in 2\Lambda^{\sharp}$ it does not depend on $\cE$.
Over $\Gr_M^{\mu}$ one has the isomorphism
\begin{equation}
\label{iso_root_over_Gr_M^mu}
\cL\mid_{\Gr_M^{\mu}}\,\iso\, \cL_{\mu, M, \cE}^N
\end{equation}

\begin{Lm} 
\label{Lm_for_canonicity_Lm7}
Let $\mu\in 2\Lambda^{\sharp,+}$ be orthogonal to all roots of $M$. Consider the map $\gt^{-\mu}: \Gr^{-\mu}_{P^-}\to \Gr^{-\mu}_M$. 
There is a natural isomorphism
\begin{equation}
\label{iso_for_canonicity_Lm7}
((\gt_{P^-}^{-\mu})^*\cL_{-\mu, M, \cE})\mid_{\Gr^{-\mu}_{P^-}\cap \Gr_G^{-w_0(\mu)}} \,\iso\, \cL_{-w_0(\mu), \cE}\mid_{\Gr^{-\mu}_{P^-}\cap \Gr_G^{-w_0(\mu)}}
\end{equation}
compatible with the isomorphisms (\ref{iso_root_over_Gr_M^mu}) for $M$ and $G$. 
\end{Lm}
\begin{Prf} The intersection of $\cB^{-w_0(\mu)}_G\cap \Gr^{-\mu}_{P^-}$ is the point $t^{-\mu}G(\bO)$ fixed by $M$. So, over $\Gr^{-\mu}_{P^-}\cap \Gr_G^{-w_0(\mu)}$ both line bundles in (\ref{iso_for_canonicity_Lm7}) are
constant, and it suffices to establish the desired isomorphism at the point $t^{-\mu}G(\bO)$. The fibres of both line bundles at this point identify with $\cE_{\bar c}^{\iota(\mu,\mu)/n}$ in a way compatible with (\ref{iso_root_over_Gr_M^mu}) for $M$ and $G$. 
\end{Prf}

\subsection{Main technical step} The purpose of this section is to formulate Theorem~\ref{Th_3}, which is 
an analog of (\cite{ICDC}, Theorem~4.5) and the main technical step in the proof of Theorem~\ref{Th_main_Section1}. 

Define $\Lambda_{G,P}^{pos,\, pos}$ as the free abelian semigroup with base $J$. Recall the map $c_P$ from Section~\ref{Section_Geometric restriction functors}. Let 
$$
\bar c_P: \Lambda_{G,P}^{pos,\, pos}\to \Lambda_{G,P}^{\sharp, \, pos}
$$
be the morphism of semigroups, which on the base of $\Lambda_{G,P}^{pos,\, pos}$ is given by $c_P$. For $\theta\in\Lambda_{G,P}^{\sharp, \, pos}$ we will denote by $\gB(\theta)$ the elements of $\Lambda_{G,P}^{pos,\, pos}$ sent by $\bar c_P$ to $\theta$. 
 
  Let $\theta\in\Lambda_{G,P}^{\sharp, pos}$. Let $\gB(\theta)=\mathop{\sum}\limits_{\nu\in J} n_{\nu} \nu$ be an element of $\Lambda_{G,P}^{pos,pos}$ over $\theta$. Set $\mid\! \gB(\theta)\! \mid=\sum_{\nu} n_{\nu}$. Write $X^{\gB(\theta)}$ for the moduli scheme of  $\Lambda_{G,P}^{pos,pos}$-valued divisors of degree $\gB(\theta)$, so
$$
X^{\gB(\theta)}\,\iso\, \prod_{\nu\in J} X^{(n_{\nu})}
$$ 
To a point $(D_{\nu}, \nu\in J)$ of the latter scheme there corresponds the divisor $D=\sum_{\nu\in J} D_{\nu}\nu$. The map $\bar c_P$ yields a finite morphism $X^{\gB(\theta)}\to X^{\theta}$. Let $\oX{}^{\gB(\theta)}\subset X^{\gB(\theta)}$ be the complement to all the diagonals. 
 
If $D=\sum_k x_k\theta_k\in X^{\theta}$ and $x_k$ are pairwise different then the fibre of $\Mod_M^{+,\theta}\to X^{\theta}$ over $D$ is $\prod_k \Gr_M^{+,\theta_k}$. Let $\IC^{\gB(\theta),0}_{\zeta}$ be the perverse sheaf on $\oX{}^{\gB(\theta)}\times_{X^{\theta}} \Mod_{\tilde M}^{+,\theta}$ on which $\mu_N(k)$ acts by $\zeta^{-1}$ and such that for $D=\sum_k x_k\nu_k\in \oX{}^{\gB(\theta)}$ with $\theta_k=c_P(\nu_k)$ its restriction to 
$$
\prod_k \wt\Gr_{M, x_k}^{+, \theta_k}
$$
is $\boxtimes_k \Loc_{\zeta^{-1}}((\check{\gu}_n(P))_{\nu_k})[\mid\! \gB(\theta)\! \mid]$. One makes this definition rigorous as in Section~\ref{section_11}. This perverse sheaf is defined up to a unique isomorphism and irreducible. Let $\IC^{\gB(\theta)}_{\zeta}$ be the intermediate extension of $\IC^{\gB(\theta),0}_{\zeta}$ under
$$
\oX{}^{\gB(\theta)}\times_{X^{\theta}} \Mod_{\tilde M}^{+,\theta}
\hook{}X^{\gB(\theta)}\times_{X^{\theta}} \Mod_{\tilde M}^{+,\theta}
$$ 
Denote by $i_{\gB(\theta)}: X^{\gB(\theta)}\times_{X^{\theta}} \Mod_{\tilde M}^{+,\theta}\to \Mod_{\tilde M}^{+,\theta}$ the second projection.

 For $\nu\in J$ let $\theta_{\nu}=c_P(\nu)$. We get a decomposition $\gU(\theta)$ given by $\theta=\sum_{\nu\in J} n_{\nu}\theta_{\nu}$, and $X^{\gB(\theta)}\,\iso\, X^{\gU(\theta)}$ naturally. It follows that $i_{\gB(\theta) *} (\IC^{\gB(\theta)}_{\zeta})$ is the intermediate extension from $\oX{}^{\gU(\theta)}\times_{X^{\theta}} \Mod_{\tilde M}^{+,\theta}$ and is $M_{X^{\theta}}$-equivariant. We used the fact that $\oX{}^{\gU(\theta)}\subset X^{\theta}$ is locally closed.

\begin{Th}
\label{Th_3} Let $\theta\in\Lambda_{G,P}^{pos}$. For the map $\gs^{\theta}: \Mod_{\tilde M}^{+,\theta}\to \wt Z^{\theta}$ there is a $M_{X^{\theta}}$-equivariant isomorphism 
$$
\gs^{\theta !}(\IC_{Z^{\theta}, \zeta})\,\iso\, \mathop{\oplus}\limits_{\gB(\theta)} i_{\gB(\theta) *} (\IC^{\gB(\theta)}_{\zeta})[-\mid\! \gB(\theta)\!\mid]
$$
In particular, this complex vanishes unless $\theta\in \Lambda_{G,P}^{\sharp, \, pos}$. 
\end{Th} 

 The following is proved exactly as (\cite{ICDC}, Lemma~4.3).

\begin{Lm} 
\label{Lm_11_resricting_to_diag}
The $*$-restriction of $\IC^{\gB(\theta)}_{\zeta}$ to $X\times_{X^{\theta}}\Mod^{+,\theta}_{\tilde M}\,\iso\, \wt\Gr_{M, X}^{+,\theta}$ 
identifies canonically with
$$
\Loc_{X, \zeta^{-1}}(\mathop{\otimes}\limits_{\nu\in J} \Sym^{n_{\nu}}(\check{\gu}_n(P)_{\nu}))[-1+\mid\! \gB(\theta)\!\mid]
$$
\end{Lm}

 The functor $\Loc_X=(\tau^0\Loc)[1]$ used in Lemma~\ref{Lm_11_resricting_to_diag} takes values in $M_X$-equivariant perverse sheaves on $\wt\Gr_{M,X}$. 

\subsubsection{} 
\label{Section_521}
Let $\theta\in\Lambda_{G,P}^{pos}$. Recall that $\Mod^{+, \theta}_{\Bun_M}$ classifies $D\in X^{\theta}$, $\cF_M,\cF'_M\in\Bun_M$ and an isomorphism $\beta_M: \cF_M\,\iso\, \cF'_M\mid_{X-D}$ such that for each $\cV\in\Rep(G)$ the map $\beta_M: \cV^{U(P)}_{\cF_M}\to \cV^{U(P)}_{\cF'_M}$ is regular, and $\beta_M$ induces an isomorphism
$$
\cF_{M/[M,M]}\,\iso\, \cF'_{M/[M,M]}(-D)
$$
Consider the projection $\Mod^{+, \theta}_{\Bun_M}\to\Bun_M\times\Bun_M$ sending the above point to $(\cF_M, \cF'_M)$. Denote by $\Modt^{+,\theta}_{\Bunt_M}$
the stack obtained from $\Mod^{+, \theta}_{\Bun_M}$ by the base change $\Bunt_M\times\Bunt_M\to \Bun_M\times\Bun_M$.

 Recall the stacks $\Bun_{M, X^{\theta}}$ and $\Bunt_{M, X^{\theta}}$ from Section~\ref{section_relation between_IC_Ztheta_and_IC_zeta}. As we have seen in Section~\ref{section_122}, $M_{X^{\theta}}$ acts on $\Mod_M^{+,\theta}$ over $X^{\theta}$. Let $M_{X^{\theta}}$ act diagonally on $\Bun_{M, X^{\theta}}\times_{X^{\theta}} \Mod_M^{+,\theta}$, the corresponding stack quotient is denoted $\Bun_{M, X^{\theta}}\times_{M_{X^{\theta}}} \Mod_M^{+,\theta}$. Let 
$$
\gamma^{\ra}: \Bun_{M, X^{\theta}}\times_{M_{X^{\theta}}} \Mod_M^{+,\theta}\,\iso\, \Mod^{+,\theta}_{\Bun_M}
$$
be the natural isomorphism such that the projection to the first factor corresponds to $h^{\ra}_M$. As in Section~\ref{section_11}, we get a $M_{X^{\theta}}$-torsor
$$
\tilde\gamma^{\ra}: \Bunt_{M, X^{\theta}}\times_{X^{\theta}} \Mod_{\tilde M}^{+,\theta}\,\iso\, \Modt^{+,\theta}_{\Bunt_M}
$$
This allows to introduce for each $\gB(\theta)$ the relative version $\IC^{\gB(\theta)}_{\Bun_M, \zeta}$ of $\IC^{\gB(\theta)}_{\zeta}$, which is a perverse sheaf on $X^{\gB(\theta)}\times_{X^{\theta}} \Modt^{+,\theta}_{\Bunt_M}$, the intermediate extension from $\oX{}^{\gB(\theta)}\times_{X^{\theta}} \Modt^{+,\theta}_{\Bunt_M}$ of 
$$
(\IC(\Bunt_M)\tboxtimes \IC^{\gB(\theta),0}_{\zeta})^r
$$
The latter is the descent of $\IC(\Bunt_{M, X^{\theta}})\boxtimes \IC^{\gB(\theta),0}_{\zeta}[-\dim X^{\theta}]$ under the $\oX{}^{\gB(\theta)}\times_{X^{\theta}} M_{X^{\theta}}$-torsor
$$
\tilde\gamma^{\ra}: \Bunt_{M, X^{\theta}}\times_{X^{\theta}} \oX{}^{\gB(\theta)}\times_{X^{\theta}} \Mod_{\tilde M}^{+,\theta}\to \oX{}^{\gB(\theta)}\times_{X^{\theta}} \Modt^{+,\theta}_{\Bunt_M}
$$

Write also $\nu_P$ for the projection $\Bun_P\times_{\Bun_M} \Modt^{+,\theta}_{\Bunt_M}\to \Modt^{+,\theta}_{\Bunt_M}$, where we used $h^{\ra}_M$ to define the corresponding fibred product. We have the locally closed embedding 
$$
\Bun_P\times_{\Bun_M} \Modt^{+,\theta}_{\Bunt_M}\hook{} \Bunt_{\tilde P}
$$ 
sending $(D, \cF_M,\cF'_M, \beta_M: \cF_M\,\iso\, \cF'_M\mid_{X-D}, \cU,\cU')\in \Modt^{+,\theta}_{\Bunt_M}$, $\cF_P\in\Bun_P$ with $\cF'_P\times_P M\,\iso\, \cF'_M$ to $(\cF'_G, \cF_M, \kappa, \cU,\cU')$, where $\cF'_G=\cF_P\times_P G$. Its image is denoted $_{\theta}\Bunt_{\tilde P}$.

Translating Theorem~\ref{Th_3} to $\Bunt_{\tilde P}$ we obtain the following.

\begin{Cor} 
\label{Cor_3}
The $*$-restriction of $\IC_{\zeta}$ to $_{\theta}\Bunt_{\tilde P}\,\iso\, \Bun_P\times_{\Bun_M} \Modt^{+,\theta}_{\Bunt_M}$ vanishes unless $\theta\in \Lambda_{G,P}^{\sharp, pos}$. In the latter case it identifies with $\nu_P^*[\dimrel(\nu_P)]$ applied to the complex
$$
\mathop{\oplus}\limits_{\gB(\theta)} (i_{\Bun_M, \gB(\theta)})_*\IC^{\gB(\theta)}_{\Bun_M, \zeta}[\mid\! \gB(\theta)\! \mid]
$$

Here we denoted by $i_{\Bun_M, \gB(\theta)}: X^{\gB(\theta)}\times_{X^{\theta}} \Modt^{+,\theta}_{\Bunt_M}\to \Modt^{+,\theta}_{\Bunt_M}$ the projection. 
\end{Cor}
 
\subsection{Proof of Theorem~\ref{Th_main_Section1}} As (\cite{ICDC}, Theorem~1.12) it is derived from Corollary~\ref{Cor_3} and Lemma~\ref{Lm_11_resricting_to_diag}. The proof uses the following combinatorial identity. Given $\theta\in\Lambda_{G,P}^{pos}$ and $m\ge 0$ the space $(\Sym^m \gu_n(P))_{\theta}$ vanishes unless $\theta\in \Lambda_{G,P}^{\sharp, pos}$, and in the latter case
$$
(\Sym^m \gu_n(P))_{\theta}=\sum_{\gB(\theta)=\mathop{\sum}\limits_{\nu\in J} n_{\nu}\nu, \;\;\mid\gB(\theta)\mid=m } \;\; (\mathop{\otimes}\limits_{\nu\in J} \Sym^{n_{\nu}}(\check{\gu}_n(P)_{\nu}))
$$
This implies 
$$
\mathop{\oplus}\limits_{\gB(\theta)=\mathop{\sum}\limits_{\nu\in J} n_{\nu}\nu} \;\; (\mathop{\otimes}\limits_{\nu\in J} \Sym^{n_{\nu}}(\check{\gu}_n(P)_{\nu}))[2\mid\! \gB(\theta)\!\mid]\,\iso\, \mathop{\oplus}\limits_{m\ge 0} (\Sym^m\gu_n(P))_{\theta}[2m]
$$
Theorem~\ref{Th_main_Section1} is proved (modulo Theorem~\ref{Th_3}).

\subsection{Induction} Now we prove Theorem~\ref{Th_3} by induction on $\theta$. Recall that for $\theta=\theta_1+\theta_2$ with $\theta_i\in\Lambda_{G,P}^{pos}$ the factorization property yields a natural $\mu_N(k)$-gerb
$$
\wt Z^{\theta_1}\times \wt Z^{\theta_2}\times_{(X^{\theta_1}\times X^{\theta_2})}(X^{\theta_1}\times X^{\theta_2})_{disj}\to \wt Z^{\theta}\times_{X^{\theta}} (X^{\theta_1}\times X^{\theta_2})_{disj},
$$
and the restriction of $\IC_{Z^{\theta},\zeta}$ under this map is canonically identified with $\IC_{Z^{\theta_1},\zeta}\boxtimes \IC_{Z^{\theta_2},\zeta}$. Similarly, we have
$$
\sqcup_{\gB(\theta_1), \gB(\theta_2)} \;\;(X^{\gB(\theta_1)}\times X^{\gB(\theta_2)})\times_{X^{\theta_1}\times X^{\theta_2}} (X^{\theta_1}\times X^{\theta_2})_{disj}\,\iso\, \sqcup_{\gB(\theta)}\;\; X^{\gB(\theta)}\times_{X^{\theta}} (X^{\theta_1}\times X^{\theta_2})_{disj}
$$
and the perverse sheaves $\IC^{\gB(\theta)}_{\zeta}$ also naturally factorise. 
So, by the induction hypothesis locally over $X^{\theta}-\vartriangle_X$ we get an isomorphism
$$
\gs^{\theta !}(\IC_{Z^{\theta}, \zeta})\,\iso\, \mathop{\oplus}\limits_{\gB(\theta), \mid\gB(\theta)\mid\ne 1} i_{\gB(\theta) *} (\IC^{\gB(\theta)}_{\zeta})[-\mid\! \gB(\theta)\!\mid]
$$
As in \cite{ICDC}, globally we could have a nontrivial monodromy for $\beta(\theta)=2\nu$ with $\nu\in J$. So, there is a rank one and order at most 2 local system $\cE_{\gB(\theta)}$ on $X^{\gB(\theta)}\times_{X^{\theta}}(X^{\theta}-\vartriangle_X)$ and an isomorphism
$$
\gs^{\theta !}(\IC_{Z^{\theta}, \zeta})\,\iso\, \mathop{\oplus}\limits_{\gB(\theta), \mid\gB(\theta)\mid\ne 1} i_{\gB(\theta) *} (\IC^{\gB(\theta)}_{\zeta}\otimes\pr_1^*\cE_{\gB(\theta)})[-\mid\! \gB(\theta)\!\mid]
$$
over $\Mod_{\tilde M}^{+,\theta}\times_{X^{\theta}}(X^{\theta}-\vartriangle_X)$. 

 Let $\cK^{\gB(\theta)}$ be the intermediate extension of $\IC^{\gB(\theta)}_{\zeta}\otimes\pr_1^*\cE_{\gB(\theta)}$ to $X^{\gB(\theta)}\times_{X^{\theta}} \Mod_{\tilde M}^{+,\theta}$. We get an isomorphism
\begin{equation}
\label{eq_key_one_with_cK_theta}
\gs^{\theta !}(\IC_{Z^{\theta}, \zeta})\,\iso\, \mathop{\oplus}\limits_{\gB(\theta), \mid\gB(\theta)\mid\ne 1} i_{\gB(\theta) *}(\cK^{\gB(\theta)})[-\mid\! \gB(\theta)\!\mid]\oplus \cK^{\theta},
\end{equation}
where $\cK^{\theta}$ is a pure complex supported over $\Mod_{\tilde M}^{+,\theta}\times_{X^{\theta}} X\;\iso\; \wt\Gr_{M,X}^{+,\theta}$. 

 As in \cite{ICDC}, let $\vartriangle_x$ denote the closed subscheme $\Spec k\toup{x} X\toup{\vartriangle_X} X^{\theta}$. The analogs of (\cite{ICDC}, Proposition~5.7 and 5.8) are as follows.

\begin{Pp} 
\label{Pp_resricting_to_X}
The complex $\gs^{\theta !}(\IC_{Z^{\theta}, \zeta})\mid_{\vartriangle_x}$ over $\wt\Gr_M^{+,\theta}$ is placed in perverse degrees $\le 0$. Its $0$-th perverse cohomology sheaf vanishes unless $\theta\in \Lambda^{\sharp}_{G,P}$, in the latter case it is identified with $\Loc_{\zeta^{-1}}(U(\check{\gu}_n(P))_{\theta})$. 
\end{Pp}
 
 Recall the map $\pi_P: \wt Z^{\theta}\to \Mod_{\tilde M}^{+,\theta}$.
 
\begin{Pp} 
\label{Pp_alongs_strata}
Assume $\theta', \theta-\theta'\in\Lambda_{G,P}^{pos}$. \\
1) The complex $\pi_{P !}(\IC_{Z^{\theta}, \zeta}\mid_{_{\theta'}\wt\SSS^{\theta}})$ is placed in strictly negative perverse degrees for $\theta'\ne 0$.\\
2) The complex $\pi_{P !}(\IC_{Z^{\theta}, \zeta}\mid_{_0\wt\SSS^{\theta}})$ is placed in perverse degrees $\le 0$. \\
3)The $0$-th perverse cohomology sheaf of $\pi_{P !}(\IC_{Z^{\theta}, \zeta}\mid_{_0\wt\SSS^{\theta}})$ vanishes unless $\theta\in \Lambda^{\sharp}_{G,P}$, in the latter case it identifies with $\Loc_{\zeta^{-1}}(U(\check{\gu}_n(P))_{\theta})$.
\end{Pp}

  Let $\Conv_M$ denote the convolution diagram for the affine grassmanian of $M$ at $x$. This is the scheme classifying $\cF_M, \cF'_M\in\Bun_M$ with isomorphisms $\tilde\beta_M: \cF_M\,\iso\, \cF'_M\mid_{X-x}$ and $\beta'_M: \cF'_M\,\iso\, \cF^0_M\mid_{X-x}$. Let $\wt\Conv_{\tilde M}$ be obtained from $\Conv_M$ by adding two lines $\cU, \cU'$ and isomorphisms 
$$
\cU^N\,\iso\, (\cL_M)_{\cF_M}, \;\;\; \cU'^N\,\iso\, (\cL_M)_{\cF'_M}
$$ 
Write $\pr': \wt\Conv_{\tilde M}\to \wt\Gr_M$ for the projection sending the above point to $(\cF'_M, \beta'_M, \cU')$. It makes $\wt\Conv_{\tilde M}$ a fibration over $\wt\Gr_M$ with typical fibre isomorphic to $\wt\Gr_M$. Now given a $M(\OO_x)$-equivariant perverse sheaf $\cS$ on $\wt\Gr_M$ on which $\mu_N(k)$ acts as $\zeta^{-1}$, and any complex $\cS'$ on $\wt\Gr_M$, on which $\mu_N(k)$ acts as $\zeta^{-1}$, we can form their twisted external product $\cS\tboxtimes \cS'$, which is $\cS'$ along the base and $\cS$ along the fibre. It is normalized to be perverse for $\cS'$ perverse. As in \cite{Ga}, one proves the following. Let $\pr: \wt\Conv_{\tilde M}\to\wt\Gr_M$ be the map sending the above point to $(\cF_M, \beta'_M\tilde\beta_M, \cU)$. The convolution of $\cS$ with $\cS'$ is defined as $\pr_!(\cS\tboxtimes \cS')$.

\begin{Lm} 
\label{Lm_great_perversity}
If $\cS$ is a $M(\OO_x)$-equivariant perverse sheaf on $\wt\Gr_M$ on which $\mu_N(k)$ acts as $\zeta^{-1}$, $\cS'$ is a perverse sheaf on $\wt\Gr_M$, on which $\mu_N(k)$ acts as $\zeta^{-1}$ then their convolution is a perverse sheaf on $\wt\Gr_M$. \QED
\end{Lm}

\begin{Prf}\select{of Proposition~\ref{Pp_alongs_strata}}

\smallskip\noindent
As in the proof of (\cite{ICDC}, Proposition~5.8), one has $\dim Z^{\theta}=\<\theta, 2(\check{\rho}-\check{\rho}_M)\>$, and points 2),3) follow from Theorem~\ref{Th_2}. 

 Assume now $\theta'\ne\theta$. As in (\cite{ICDC}, Section~3.5), let $\Conv_M^{+,\theta'}$ denote the closed subscheme of the convolution diagram $\Conv_M$ at $x$ given by the property $(\cF_M, \cF'_M, \tilde\beta_M)\in\Mod_{\Bun_M}^{+,\theta'}$. As in \select{loc.cit.}, one has an isomorphism
$$
_0\SSS^{\theta-\theta'}\times_{\Gr_M} \Conv_M^{+,\theta'}\;\iso\; {_{\theta'}\SSS^{\theta}},
$$
here the map $\Conv^{+,\theta'}_M\to \Gr_M$ used to define the fibred product sends the above point of $\Conv^{+,\theta'}_M$ to $(\cF'_M, \beta'_M)$. Let $\wt\Conv_{\tilde M}^{+,\theta'}$ be obtained from $\Conv_M^{+,\theta'}$ by the base change $\wt\Conv_{\tilde M}\to \Conv_M$.
We get a natural $\mu_N$-gerb 
$$
_0\wt\SSS^{\theta-\theta'}\times_{\wt\Gr_M} \wt\Conv_{\tilde M}^{+,\theta'}\to {_{\theta'}\wt\SSS^{\theta}}
$$ 
given by forgetting $\cU'$. As in the proof of (\cite{ICDC}, Proposition~5.8), the $*$-restriction of $\IC_{Z^{\theta}, \zeta}\mid_{_{\theta'}\wt\SSS^{\theta}}$ under this gerb is described by the induction hypothesis. Namely, by Corollary~\ref{Cor_3} and Lemma~\ref{Lm_11_resricting_to_diag}, it identifies with
$$
\mathop{\oplus}\limits_{\gB(\theta')} \Loc_{\zeta^{-1}}(\otimes_{\nu\in J} \Sym^{n_{\nu}}(\check{\gu}_n(P)_{\nu})[2\mid \!\gB(\theta')\!\mid]\tboxtimes (\IC_{Z^{\theta-\theta'},\zeta}\mid_{_0\wt\SSS^{\theta-\theta'}})
$$
Here it is understood that $\gB(\theta')=\sum_{\nu\in J} n_{\nu}\nu$, the sum being taken over all elements $\gB(\theta')$ over $\theta'$. Now by 2), $\pi_{P !}(\IC_{Z^{\theta-\theta'}, \zeta}\mid_{_0\wt\SSS^{\theta-\theta'}})$ is placed in perverse degrees $\le 0$. So, by Lemma~\ref{Lm_great_perversity}, $\pi_{P !}(\IC_{Z^{\theta}, \zeta}\mid_{_0\wt\SSS^{\theta}})$ is placed in perverse degrees $< 0$.

 In the case $\theta'=\theta$ the complex $\pi_{P !}(\IC_{Z^{\theta}, \zeta}\mid_{_{\theta'}\wt\SSS^{\theta}})$ is placed in strictly negative perverse degrees, as $\pi_P: {_\theta\wt\SSS^{\theta}}\to\wt\Gr_M^{+,\theta}$ is an isomorphism. We are done.
\end{Prf}

\medskip

 Now as in (\cite{ICDC}, Section~5.11) one checks that all the local systems $\cE_{\gB(\theta)}$ are trivial.

 To finish the proof of Theorem~\ref{Th_3} it remains to analyze the complex $\cK^{\theta}$ from (\ref{eq_key_one_with_cK_theta}). There is at most one $\gB(\theta)$ with $\mid\!\gB(\theta)\!\mid=1$, which we denote $\gB(\theta)^0$ as in \cite{ICDC}. If it exists, that is, $\theta=c_P(\nu)$ for some $\nu\in J$, we have to show that $\cK^{\theta}\,\iso\, (i_{\gB(\theta)^0})_*\IC^{\gB(\theta)^0}_{\zeta}[-1]$. Otherwise, we have to show that $\cK^{\theta}=0$. 
 
  By definition of $\IC$, as $\cK^{\theta}$ is a direct summand of $\gs^{\theta !}(\IC_{Z^{\theta}, \zeta})$, it is placed in perverse degrees $\ge 1$. Restrict both sides of (\ref{eq_key_one_with_cK_theta}) to $\Mod^{+,\theta}_{\tilde M}\mid_{\vartriangle_X}$ and apply the perverse cohomological truncation $\tau_{\ge 1}$. Using Lemma~\ref{Lm_11_resricting_to_diag} and Proposition~\ref{Pp_resricting_to_X}, we get
$$
\Loc_{X, \zeta^{-1}}(U(\check{\gu}_n(P))_{\theta})[-1]\,\iso\, \mathop{\oplus}\limits_{\gB(\theta), \mid\gB(\theta)\mid\ne 1} \Loc_{X, \zeta^{-1}}(\mathop{\otimes}\limits_{\nu\in J} \Sym^{n_{\nu}}(\check{\gu}_n(P)_{\nu}))[-1]\oplus \cK^{\theta}\mid_{\wt\Gr^{+,\theta}_{M,X}}
$$
As in (\cite{ICDC}, Section~5.12), this implies the desired result. We used here that $U(\check{\gu}_n(P))$ and $\Sym(\check{\gu}_n(P))$ are non-canonically isomorphic as $\check{M}_n$-modules. Theorem~\ref{Th_main_Section1} is proved. 

\section{Composing Eisenstein series}
\label{section_composing}

\subsubsection{} In this section we prove Theorem~\ref{Th_composing_Eis}, which is an analog of (\cite{BG}, Theorem~2.3.10) in our setting.

Keep notations of Section~\ref{Section_4.1}. Let $B(M)\subset M$ be the Borel subgroup corresponding to the roots $\check{\alpha}_i, i\in\cI_M$. As in (\cite{BG}, Section~7.1) set $\Bunt_{B,P}=\Bunt_P\times_{\Bun_M} \Bunb_{B(M)}$. Set also $\Bunt_{\tilde B, \tilde P}=\Bunt_{\tilde P}\times_{\Bunt_M}\Bunb_{\tilde B(M)}$. We have the cartesian square
$$
\begin{array}{ccc}
\Bunt_{\tilde B, \tilde P} & \toup{\tilde\gq'_P} &\Bunb_{\tilde B(M)}\\
\downarrow\lefteqn{\scriptstyle \tilde\gp'_M} && \downarrow\lefteqn{\scriptstyle \tilde\gp}\\
\Bunt_{\tilde P} & \toup{\tilde\gq} & \Bunt_M
\end{array}
$$
Write also $\Bunt_{\tilde B, P}=\Bunt_{B,P}\times_{\Bun_T\times \Bun_G}(\Bunt_T\times\Bunt_G)$. Recall that $\Bun_B\subset \Bun_{B,P}$ is naturally an open substack. The preimage of $\Bun_B$ in $\Bunt_{\tilde B, P}$ identifies with the stack $\Bun_{\tilde B}$ from Section~\ref{Section_3.1}. 
Recall the perverse sheaf $\cL_{\zeta}\boxtimes \IC(\Bun_{B,\tilde G})$ viewed as a perverse sheaf on $\Bun_{\tilde B}$ via (\ref{eq_one}). Denote by $\IC_{B,P,\zeta}$ its intermediate extension under $\Bun_{\tilde B}\hook{} \Bunt_{\tilde B,P}$. 

\begin{Pp} 
\label{Pp_IC_over_Bun_BP}
The restriction of $\IC_{B,P,\zeta}$ under the projection $\Bunt_{\tilde B, \tilde P}\to \Bunt_{\tilde B,P}$ identifies canonically with
$$
(\tilde\gp'_M)^*\IC_{\zeta}\otimes (\tilde\gq'_P)^*\IC_{\zeta}[-\dim\Bun_M]
$$
\end{Pp}
\begin{proof}
Since $\Bun_P\to\Bun_M$ is smooth, $\IC_{\zeta}$ and $(j_{\tilde P})_!j_{\tilde P}^*\IC_{\zeta}$ are ULA with respect to $\tilde\gq: \Bunt_{\tilde P}\to \Bunt_M$ by Proposition~\ref{Pp_IC_zeta_is_ULA_for_P}, the proof of (\cite{BG}, Theorem~7.1.2) applies in our setting. Using (\cite{L5}, Proposition~4.8.5), one gets the desired isomorphism.
\end{proof}

\begin{Rem} There are two a priori different definitions of the ULA property. For a morphism $p_1: Y_1\to S_1$ and an object $L\in D(Y_1)$ the first definition of $L$ being ULA with respect to $p_1$ is (\cite{SGA4demi}, Definition~2.12), and the second is (\cite{BG}, Definition~5.1). In the latter one requires that the local acyclicity holds after any smooth base change $q: S\to S_1$, while in the former one requires it to hold after any base change 
$q: S\to S_1$. In the proof of Proposition~\ref{Pp_IC_over_Bun_BP} it was used that the two definitions are equivalent. 
\end{Rem}

 Recall the definition of the natural map $\gr_P: \Bunt_{B,P}\to\Bunb_B$ from(\cite{BG}, Proposition~7.1.5). A point of $\Bunt_{B,P}$ is a collection $(\cF_G, \cF_M, \kappa_P)\in\Bunt_P, (\cF_M, \cF_T, \kappa_M)\in\Bunb_{B(M)}$. Here $\kappa_M$ is the collection of embeddings 
$$
\kappa_M^{\check{\nu}}: \cL^{\check{\nu}}_{\cF_T}\hook{} \cU^{\check{\nu}}_{\cF_M},\;\; \check{\nu}\in\Lambda^+_M,
$$
and $\cU^{\check{\nu}}$ denotes the corresponding Weyl module for $M$. Then $\gr_P$ sends this point to $(\cF_G, \cF_T,\kappa)$, where for $\check{\lambda}\in\check{\Lambda}^+$ the map $\kappa^{\check{\lambda}}$ is the composition
$$
\cL^{\check{\lambda}}_{\cF_T}\;\hook{\kappa_M^{\check{\lambda}}} \;\cU^{\check{\nu}}_{\cF_M}\;\to\; (\cV^{\check{\lambda}})^{U(P)}_{\cF_M}\;\hook{\kappa_P^{\check{\lambda}}} \;\cV^{\check{\lambda}}_{\cF_G}
$$
The map $\gr_P$ is representable and proper, it extends naturally to $\wt\gr_P: \Bunt_{\tilde B, P}\to\Bunb_{\tilde B}$. Recall that $\gr_P$ is an isomorphism over the open substack $\Bun_B\subset \Bunb_B$. 

The following is an analog of (\cite{BG}, Theorem~7.1.6) in our setting.
\begin{Th} 
\label{Th_gr_P_direct_image}
One has canonically $\wt\gr_{P !}\IC_{B,P,\zeta}\,\iso\, \IC_{\zeta}$.
\end{Th}

\subsubsection{Proof of Theorem~\ref{Th_gr_P_direct_image}} Once our Theorem~\ref{Th_main_Section1} is established, one easily adopts the proof of (\cite{BG}, Theorem~7.1.6) to our setting. We indicate the corresponding notation and changes for the convenience of the reader.
 
 Recall that $\Lambda^+_{M,G}=\Lambda^+_M\cap w_0^M(\Lambda_G^{pos})$. Given a collection $\ov{\mu,\nu}$ consisting of $n_1,\ldots, n_k\in \ZZ_{>0}$ and pairwise different elements $(\mu_1, \nu_1),\ldots(\mu_k,\nu_k)\in (\Lambda_M^{pos}\times \Lambda^+_{M,G})-0$
one sets $X^{\ov{\mu,\nu}}=X^{(n_1)}\times\ldots\times X^{(n_k)}-\vartriangle$, where $\vartriangle$ is the divisor of all the diagonals. Write $D=(D_1,\ldots, D_k)\in X^{\ov{\mu,\nu}}$ for a point of this scheme.

 The Hecke stack $\cH^{\ov{\mu,\nu}}_M$ classifies $(D\in X^{\ov{\mu,\nu}}, \cF'_M,\cF_M\in\Bun_M, \beta: \cF'_M\,\iso\, \cF_M\mid_{X-D})$ such that $\cF'_M$ is in the position $\nu_i$ with respect to $\cF_M$ at points of $D_i$. 
 
 By definition, $_{\ov{\mu,\nu}}\Bun_{B,P}$ is the image of the locally closed embedding 
$$
\Bun_P\times_{\Bun_M}\cH_M^{\ov{\mu,\nu}}\times_{\Bun_M} \Bun_{B(M)}\hook{} \Bunt_{B,P}
$$
 The first stack classifies 
$(D\in X^{\ov{\mu,\nu}},  \cF_M, \cF'_M, \beta)\in \cH^{\ov{\mu,\nu}}_M, \cF_P\in\Bun_P$ with an isomorphism $\cF_P\times_P M\,\iso\, \cF_M$, a $B(M)$-torsor $\cF_{B(M)}$ with $\cF_T:=\cF_{B(M)}\times_{B(M)} T$, an isomorphism $\cF_{B(M)}\times_{B(M)} M\,\iso\, \cF'_M$. Its image in $\Bunt_{B,P}$ is the collection $(\cF_G, \cF'_M,\kappa)\in\Bunt_P, (\cF'_M, \cF'_T, \kappa_M)\in\Bunb_{B(M)}$, where $\cF'_T=\cF_T(-\sum_i \mu_i D_i)$ and $\cF_G=\cF_P\times_P G$. 

  For $\lambda=\sum_i m_i\alpha_i\in\Lambda_G^{pos}$ the locally closed substack $\Bunb_B^{\lambda}\subset \Bunb_B$ is defined as the image of the locally closed immersion 
$$
\prod_{i\in\cI} X^{(m_i)}\times\Bun_B\to \Bunb_B
$$ 
sending $((D_i)_{i\in \cI}, \cF_B)$ to $(\cF'_T, \cF_G, \kappa)$, where $\cF'_T=\cF_T(-\sum_i \alpha_i D_i)$ for $\cF_T=\cF_B\times_B T$, $\cF_G=\cF_B\times_B G$. 

 For $\ov{\mu,\nu}$ as above, $\lambda\in\Lambda_G^{pos}$ one sets 
$$
 _{\ov{\mu,\nu}, \lambda}\Bun_{B,P}={_{\ov{\mu,\nu}}\Bun_{B,P}}\cap \gr_P^{-1}(\Bunb_B^{\lambda})
$$ 
Write $_{\ov{\mu,\nu}, \lambda}\Bun_{\tilde B,P}$ for the preimage of $_{\ov{\mu,\nu}, \lambda}\Bun_{B,P}$ under $\Bunt_{\tilde B,P}\to\Bunt_{B,P}$. 
 
 We set $\mid\!\ov{\mu,\nu}\!\mid=\sum_{i=1}^k n_i(\nu_i-\mu_i)\in\Lambda$. Theorem~\ref{Th_gr_P_direct_image} is reduced to the following.

\begin{Pp} For $\ov{\mu,\nu}$, $\lambda\in\Lambda_G^{pos}$ as above the following holds:\\
(i) The $*$-restriction of $\IC_{B,P,\zeta}$ to $_{\ov{\mu,\nu}, \lambda}\Bun_{\tilde B,P}$ lives in perverse degrees $\le -\<\lambda+\mid\!\ov{\mu,\nu}\!\mid, \check{\rho}_M\>$, and the inequality is strict unless $\lambda=0$.\\
(ii) The fibres of $\gr_P: {_{\ov{\mu,\nu}, \lambda}\Bun_{\tilde B,P}}\to \Bunb_B^{\lambda}$ are of dimension $\le \<\lambda+\mid\!\ov{\mu,\nu}\!\mid, \check{\rho}_M\>$.
\end{Pp}
\begin{proof}
The proof of (\cite{BG}, Proposition~7.1.8) applies. The only change is that one uses Theorem~\ref{Th_main_Section1} (and Proposition~\ref{Pp_IC_over_Bun_BP}) to garantee that the $*$-restriction of $\IC_{B,P,\zeta}$ to $_{\ov{\mu,\nu}}\Bun_{\tilde B,P}$ has smooth cohomology sheaves.
\end{proof}
 
\begin{proof}[Proof of Theorem~\ref{Th_composing_Eis}] For $K\in \D_{\zeta}(\Bunt_T)$ one has
$$
\Eis_M^G\Eis^M_T(K)\,\iso\, \tilde\gp_!(\IC_{\zeta}\otimes\tilde\gq^*(\tilde\gp_M)_!(\tilde\gq_M^*K\otimes \IC_{\zeta}))[-\dim\Bun_T-\dim\Bun_M]
$$
Using the projection formula and Proposition~\ref{Pp_IC_over_Bun_BP}, this identifies with 
\begin{equation}
\label{complex_for_Th_composing_Eis}
\tilde\gp_!(\tilde\gp'_M)_!(\IC_{B,P,\zeta}\otimes(\tilde\gq'_P)^*\tilde\gq_M^*)[-\dim\Bun_T]
\end{equation}
The composition $\tilde\gq'_P\tilde\gq_M$ coincides with $\Bunt_{\tilde B,\tilde P}\to \Bunt_{\tilde B, P}\toup{\tilde\gr_P}\Bunb_{\tilde B}\toup{\tilde\gq}\Bunt_T$, and the composition $\tilde\gp\tilde\gp'_M$ coincides with $\Bunt_{\tilde B,\tilde P}\to \Bunt_{\tilde B, P}\toup{\tilde\gr_P}\Bunb_{\tilde B}\toup{\tilde\gp}\Bunt_G$. So, (\ref{complex_for_Th_composing_Eis}) identifies with 
$$
\tilde\gp_!((\tilde\gr_P)_!\IC_{B,P,\zeta}\otimes\tilde\gq^*K)[-\dim\Bun_T]\,\iso\, \Eis_T^G(K)
$$
by Theorem~\ref{Th_gr_P_direct_image}. 
\end{proof}

\section{The case of $G=\SL_2$}
\label{Section_The case of SL_2}

\subsection{Precisions} In this section we get some more precise results for $G=\SL_2$. Keep notations of Section~\ref{Section_Main_results}. Let $e=n$ for $n$ odd (resp., $e=\frac{n}{2}$ for $n$ even). Then $\Lambda^{\sharp}=e\Lambda$. The unique simple coroot of $G$ is denoted $\alpha$, the simple root of $\check{G}_n$ is $n\alpha$. For $n$ even one gets $\check{G}_n\,\iso\,\SL_2$, and $\check{G}_n\,\iso\,\PSL_2$ for $n$ odd. Recall that $\check{h}=2$. 

 Let $\cL_c$ be the line bundle on $\Bun_G$ with fibre $\det\RG(X, \cO^2)\otimes\det\RG(X, M)^{-1}$ at $M\in\Bun_G$. The restriction of $\cL_c$ to  $\Bun_T$ is also denoted $\cL_c$.
 
 Identify $T$ with $\Gm$ via the coroot $\alpha: \Gm\,\iso\, T$, so $\Bun_1\,\iso\, \Bun_T$. The isomorphism $\ZZ\,\iso\, \Lambda^{\sharp}$, $1\mapsto e$ yields $\Gm\,\iso\, T^{\sharp}$, so that $i_X: \Bun_1=\Bun_{T^{\sharp}}\to \Bun_T=\Bun_1$ sends $\cE$ to $\cE^e$. The line bundle $\tau$ on $\Bun_{T^{\sharp}}$ is chosen as in (\cite{L}, 5.2.6, Example (1)). Namely, if $n$ is odd then 
$$
\tau_{\cE}=\det\RG(X, \cO)^{2n}\otimes \det\RG(X, \cE)^{-n}\otimes \det\RG(X, \cE^{-1})^{-n}
$$ 
for $\cE\in\Bun_1=\Bun_{T^{\sharp}}$. If $n$ is even then we first pick a super line bundle $\cL_1$ on $\Bun_1$ equipped with $\cL_1^2\,\iso\,\cL_c$ on $\Bun_T$. Then $\tau=\cL_1^e$ on $\Bun_1$. 
   
 Recall the action of $Z(G)=\mu_2$ on $\Bunt_G$ by 2-automorphisms (see Section~\ref{section_Some_gradings}). Denote by $\D_{\zeta,+}(\Bunt_G)$ and by $\D_{\zeta,-}(\Bunt_G)$ the full subcategory of $\D_{\zeta}(\Bunt_G)$, where $-1\in\mu_2$ acts as 1 and $-1$ respectively. As in \cite{L2}, we get 
\begin{equation}
\label{decomp_of_D_zeta(Bunt_G)}
\D_{\zeta}(\Bunt_G)\,\iso\, \D_{\zeta,+}(\Bunt_G)\times \D_{\zeta,-}(\Bunt_G)
\end{equation}
For $n$ even the category $\Rep(\check{G}_n)$ is $\ZZ/2\ZZ$-graded according to the action of the center of $\check{G}_n\,\iso\, \SL_2$. Lemma~\ref{Lm_great_about_gradings} says in this case that the Hecke functors are compatible with these gradings of 
$\D_{\zeta}(\Bunt_G)$ and $\Rep(\check{G}_n)$. 

\subsubsection{} \label{Section_211}
Take $P=B$. Let us reformulate Corollary~\ref{Cor_3} more explicitly in this case.
The stack $\Bun_B$ classifies exact sequences 
\begin{equation}
\label{ext_cE^-1_by_cE}
0\to \cE\to M\to \cE^{-1}\to 0
\end{equation}
on $X$ with $\cE\in\Bun_1$. The stack $\Bunb_B$ classifies $M\in \Bun_G$ and a subsheaf $\cE\hook{} M$, where $\cE\in\Bun_1$. The line bundle $\cL_T$ on $\Bun_T=\Bun_1$ is such that its fibre at $\cE$ is 
\begin{equation}
\label{iso_over_Bun_T_for_SL_2}
\frac{\det\RG(X,\cO)^2}
{\det\RG(X, \cE^2)\otimes\det\RG(X, \cE^{-2})}=\frac{\det\RG(X,\cO)^8}{\det\RG(X, \cE)^4\otimes\det\RG(X, \cE^{-1})^4}
\end{equation}  
There is an isomorphism $\cL_c^4\,\iso\, \cL$ over $\Bun_G$, whose restriction to $\Bun_T$ is compatible with the isomorphism (\ref{iso_over_Bun_T_for_SL_2}). 

The stack $\Bunt_T$ is the gerb of $4n$-th roots of $\cL_T$ over $\Bun_T$.
The map $\bar c_P: \Lambda^{pos,pos}_{G,B}\to \Lambda^{\sharp, pos}_{G,B}$ is injective, its image equals $n\alpha\ZZ_+$. 

 For $\theta=m\alpha\in\Lambda^{pos}_{G,B}$ the stack $_{\theta}\Bunb_B$ classifies $D\in X^{(m)}$ and an exact sequence on $X$
$$
0\to \cE(D)\to M\to \cE^{-1}(-D)\to 0
$$
with $\cE\in\Bun_1$. The $*$-restriction of $\IC_{\zeta}$ to $_{\theta}\Bunb_{\tilde B}$ vanishes unless $m\in n\ZZ$, in the latter case $\theta$ admits a unique $\gB(\theta)=\frac{m}{n}\nu$, where $\nu=n\alpha\in J$, and the map $X^{\gB(\theta)}\to X^{\theta}$ becomes $X^{(\frac{m}{n})}\to X^{(m)}$, $D\mapsto nD$. 
By Corollary~\ref{Cor_3},
$$
\IC_{\zeta}\mid_{_{\theta}\Bunb_{\tilde B}}\,\iso\, \nu_B^*\IC^{\gB(\theta)}_{\Bun_M, \zeta}[\dimrel(\nu_B)+\frac{m}{n}],
$$
where $\nu_B: \Bun_B\times_{\Bun_T}\wt\Mod^{+,\theta}_{\wt\Bun_T}\to 
\wt\Mod^{+,\theta}_{\wt\Bun_T}$ is the projection. The perverse sheaf $\IC^{\gB(\theta)}_{\Bun_M, \zeta}$ is described as follows. 

\begin{Lm} 
\label{Lm_calculation_detRG_for_cE(nD)}
Let $\cE\in\Bun_1$, let $D$ be an effective divisor on $X$. Then there is a canonical $\ZZ/2\ZZ$-graded isomorphism
$$
\frac{\det\RG(X, \cE)\otimes \det\RG(X, \cE^{-1})}{\det\RG(X, \cE(nD)\otimes\det\RG(X, \cE^{-1}(-nD))}\,\iso\, \left(\frac{\det\RG(X, \cO_D)}{
\det\RG(X, \cE^2(nD)\mid_D)}\right)^n
$$
\end{Lm}
\begin{Prf} One has canonically $\det\RG(X, \cE(nD)/\cE)\,\iso\, \otimes_{r=1}^n \det\RG(X, \cE(rD)\mid_D)$ and 
$$
\det\RG(X, \cE^{-1}/\cE^{-1}(-nD))\,\iso\, \otimes_{r=1}^n \det\RG(X, \cE^{-1}((r-n)D)\mid_D)
$$ 
Using \ref{Section_612} below, we get for $1\le r\le n$
$$
\frac{\det\RG(X, \cE^{-1}((r-n)D)\mid_D)}{\det\RG(X, \cE(rD)\mid_D)}\,\iso\, \frac{\det\RG(X, \cO_D)}{
\det\RG(X, \cE^2(nD)\mid_D)}
$$
Our claim follows.
\end{Prf}

\medskip

 The stack $X^{\gB(\theta)}\times_{X^{\theta}} \wt\Mod^{+,\theta}_{\wt\Bun_T}$ classifies $D\in X^{(\frac{m}{n})}$, $\cE\in\Bun_1$, and two lines $\cU, \cU_G$ equipped with isomorphisms $\cU^N_G\,\iso\, (\cL_T)_{\cE(nD)}$, $\cU^N\,\iso\, (\cL_T)_{\cE}$. The stack $_{\theta}\Bunb_{\tilde B}\times_{X^{\theta}} X^{\gB(\theta)}$ classifies 
the same data together with an exact sequence 
$$
0\to \cE(nD)\to M\to \cE^{-1}(-nD)\to 0,
$$
Recall that $N=4n$. One has an isomorphism
\begin{equation}
\label{iso_for_D_divisible_by_n}
X^{\gB(\theta)}\times_{X^{\theta}} \Mod^{+,\theta}_{\wt\Bun_T}\times B(\mu_N)\;\iso\; X^{\gB(\theta)}\times_{X^{\theta}} \wt\Mod^{+,\theta}_{\wt\Bun_T}
\end{equation}
sending $(D, \cE, \cU_G, \cU^N_G\,\iso\, (\cL_T)_{\cE(nD)}, \cU_0^N\,\iso\, k)$ to $(D, \cE, \cU_G, \cU)$, where 
$$
\cU=\cU_G\otimes \cU_0^{-1}\otimes \frac{\det\RG(X, \cE^2(nD)\mid_D)}{\det\RG(X, \cO_D)}
$$ 
is equipped with the isomorphism
$\cU^N\,\iso\, (\cL_T)_{\cE}$ given by Lemma~\ref{Lm_calculation_detRG_for_cE(nD)}. The perverse sheaf $\IC^{\gB(\theta)}_{\Bun_M, \zeta}$ via (\ref{iso_for_D_divisible_by_n}) identifies non-canonically with 
$$
\IC(X^{\gB(\theta)}\times_{X^{\theta}} \Mod^{+,\theta}_{\wt\Bun_T})\boxtimes \cL_{\zeta}
$$ 

\subsubsection{} 
\label{Section_612}
We need the following. Let $D$ be an effective divisor on $X$, $\cA,\cB\in\Bun_1$. There is a canonical $\ZZ/2\ZZ$-graded isomorphism
$$
\frac{\det\RG(X, \cA\mid_D)}{\det\RG(X, \cB\mid_D)}\;\iso\; \frac{\det\RG(X, \cA\otimes\cB^{-1}\mid_D)}{\det\RG(X,\cO_D)}
$$

\subsubsection{} 
\label{Section_713_Zastava}
Let $\theta=m\alpha\in \Lambda^{pos}_{G,B}$. For $\cE\in\Bun_1=\Bun_T$ write $Z^{\theta}_{\cE}$ for the Zastava space with the `background' $T$-torsor $\cE^{-1}$. Then $Z^{\theta}_{\cE}$ is a vector bundle over $X^{\theta}$ whose fibre at $D\in X^{(m)}$ is $\cE^2(D)/\cE^2\,\iso\, \Ext^1(\cE^{-1}/\cE^{-1}(-D), \cE)$. It is understood that a point of $Z^{\theta}_{\cE}$ gives rise to a diagram on $X$
\begin{equation}
\label{diag_point_of_Z^theta}
\begin{array}{cccc}
0\to \cE\to M\! &\to & \cE^{-1} & \to 0\\
& \nwarrow & \uparrow\\
&&\cE^{-1}(-D)
\end{array}
\end{equation}

 The group scheme $M_{X^{\theta}}$ acts trivially on $\Mod_M^{+,\theta}=X^{\theta}$. If $D\in X^{(m)}$ is given by $D=\sum n_k x_k$ with $x_k$ pairwise different then the fibre of $M_{X^{\theta}}$ at $D$ is $\prod_k \cO_{x_k}^*$.
The action of $M_{X^{\theta}}$ on $Z^{\theta}$ from Section~\ref{section_122} becomes as follows. The element $g=(g_k)\in \prod_k \cO_{x_k}^*$ acts on $v=(v_k)\in \prod_k \cE^2(n_k x_k)/\cE^2$ as $g^2v=(g_k^2v_k)$.  

 Let $\GG^{\theta}$ denote the group scheme over $X^{\theta}$ whose fibre at $D$ is $(\cO/\cO(-D))^*$. The action of $M_{X^{\theta}}$ on $Z^{\theta}$ factors through an action of $\GG^{\theta}$.
 
  Write $\bar\cL$ for the line bundle over $X^{\theta}$, whose fibre at $D$ is
\begin{equation}
\label{line_bundle_over_tilde_Z^theta_cE}
\frac{(\cL_c)_{\cE(D)}}{(\cL_c)_{\cE}}=\frac{\det\RG(X, \cE)\otimes\det\RG(X, \cE^{-1})}{\det\RG(X, \cE(D))\otimes\det\RG(X, \cE^{-1}(-D))}\;\iso\; \frac{\det\RG(X, \cO_D)}{\det\RG(X, \cE^2(D)\mid_D)}
\end{equation}
Here the second isomorphism is given by Lemma~\ref{Lm_calculation_detRG_for_cE(nD)}. The restriction of $\bar\cL$ to $Z^{\theta}_{\cE}$ is also denoted $\bar\cL$. Then $\tilde Z^{\theta}_{\cE}$ is the gerb of $4n$-th roots of $\bar\cL^4$. Write also $\tilde Z^{\theta}_{\cE, c}$ for the gerb of $n$-th roots of $\bar\cL$. We have a natural map $\tilde Z^{\theta}_{\cE, c}\to \tilde Z^{\theta}_{\cE}$ making  $\tilde Z^{\theta}_{\cE}$ a trivial $\mu_4$-gerb over $\tilde Z^{\theta}_{\cE, c}$. Let $\IC_{Z^{\theta}_c, \bar\zeta}$ denote the restriction of $\IC_{Z^{\theta}, \zeta}$ to $\tilde Z^{\theta}_{\cE, c}$. 

 For a point $(D, v\in \cE^2(D)/\cE^2, \cU)\in \tilde Z^{\theta}_{\cE, c}$ with $\cU^n\,\iso\,\bar\cL_D$ note that $a\in \mu_n(k)\subset\Aut\cU$ acts on $\IC_{Z^{\theta}_c, \bar\zeta}$ as $\bar\zeta(a)^{-1}$. 

 The open subscheme $Z^{\theta}_{max}\subset Z^{\theta}_{\cE}$ classifies $(D,\sigma\in \cE^2(D)/\cE^2)$ such that for any $0\le D'<D$ we have $\sigma\notin \cE^2(D')/\cE^2$. Over $Z^{\theta}_{max}$ we have a canonical section of $\bar\cL$ given by the isomorphisms
$$
\det\RG(X, \cE)\otimes\det\RG(X, \cE^{-1})\,\iso\,\det\RG(X,M)\,\iso\, \det\RG(X, \cE(D))\otimes\det\RG(X, \cE^{-1}(-D))
$$
 
  Let $\check{Z}^{\theta}_{\cE}\to X^{\theta}$ denote the dual vector bundle, so its fibre over $D$ is $\cE^{-2}\otimes\Omega/\cE^{-2}\otimes\Omega(-D)$. Denote by $\check{\tilde Z}^{\theta}_{\cE, c}$ the gerb of $n$-th roots of $\bar\cL$ over $\check{Z}^{\theta}_{\cE}$. 

\subsection{Fourier coefficients} The purpose of this section is to establish some results about the Fourier transform of $\IC_{Z^{\theta}_c, \bar\zeta}$ over $\check{\tilde Z}^{\theta}_{\cE, c}$. This is important in view of a relation with the theory of Weyl group multiple Dirichlet series (see \cite{BBF}, \cite{B} for a survey).
 
\subsubsection{} We need the following observation. Let $\cX$ denote the stack classifying a 1-dimensional $k$-vector spaces $L$,  $U$ together with $U^n\,\iso\, L$, and $v\in L$. This is a vector bundle over the stack $B(\Gm)$ classifying a line $U$. Let $\check{\cX}$ denote the dual vector bundle over $B(\Gm)$, this is the stack classifying $U, L, U^n\,\iso\, L$ and $v^*\in L^*$. Let $\cX^0\subset \cX$ be the open substack given by $v\ne 0$. 
We have an isomorphism $\tau_{\cX}: \cX^0\,\iso\, B(\mu_n)$ sending the above point to $U$ equipped with the trivialization $U^n\,\iso\, k$ obtained from $k\,\iso\, L$, $1\mapsto v$. For the natural map $a: \Spec k\to B(\mu_n)$ let $\cL_{\bar\zeta}$ denote the direct summand in $a_*\Qlb$ on which $\mu_n(k)$ acts by $\bar\zeta$. Let $\cL_{\bar\zeta, ex}$ denote the intermediate extension of $\tau_{\cX}^*\cL_{\bar\zeta}$ under $\cX^0\hook{}\cX$. 

 Denote by $\check{\cX}^0\subset \check{\cX}$ the open substack given by $v^*\ne 0$. Let $\tau_{\check{\cX}}: \check{\cX}^0\,\iso\, B(\mu_n)$ be the isomorphism sending the above point to $U$ equipped with the composition $U^n\,\iso\, L\toup{v^*} k$. Write $\cL_{\bar\zeta, \check{ex}}$ for the intermediate extension of $\tau_{\check{\cX}}^*\cL_{\bar\zeta}$ to $\check{\cX}$. For $n\ge 2$ there is a 1-dimensional $\Qlb$-vector space $\cC_0$ and a canonical isomorphism 
\begin{equation} 
\label{iso_over_cX}
\Four_{\psi}(\cL_{\bar\zeta^{-1}, \check{ex}})\,\iso\, \cC_0\otimes \cL_{\bar\zeta^{-1}, ex}
\end{equation}

\subsubsection{Example $\theta=\alpha$}  In this case $Z^{\theta}_{\cE}$ is the total space of the line bundle $\cE^2\otimes\Omega^{-1}$ over $X$. The line bundle $\bar\cL$ over $X$ identifies with $\cE^{-2}\otimes\Omega$. We have a map $p_{\check{\cX}}: \tilde Z^{\theta}_{\cE,c}\to \check{\cX}$ given by $L=\cE^{-2}\otimes \Omega$. Then 
$$
p_{\check{\cX}}^*\cL_{\bar\zeta^{-1}, \check{ex}}[2]\,\iso\, \IC_{Z^{\theta}_c, \bar\zeta}
$$
canonically. We also have the natural map $p_{\cX}: \check{\tilde Z}^{\theta}_{\cE, c}\to \cX$ defined by the same formula. Set $\IC_{\check{Z}^{\theta}_c, \bar\zeta}=p_{\cX}^*\cL_{\bar\zeta^{-1}, ex}[2]$, this is an irreducible perverse sheaf.
Now from (\ref{iso_over_cX}) we get an isomorphism 
$$
\Four_{\psi}(\IC_{Z^{\theta}_c, \bar\zeta})\,\iso\, \cC_0\otimes \IC_{\check{Z}^{\theta}_c, \bar\zeta}
$$
\subsubsection{Generalization} 
\label{Section_Generalization}
It is natural to consider the following generalization of $\IC_{Z^{\theta}_c, \bar\zeta}$. Let $L$ be a line bundle on $X$. Let $\theta=m\alpha$, $m\ge 0$, so $X^{(m)}\,\iso\, X^{\theta}$ via the map $D\mapsto D\alpha$. Let $X^{\theta, rss}\subset X^{\theta}$ be the open subscheme of reduced divisors. Write $\Sign$ for the sign local system on $X^{\theta, rss}$. Let $_LZ^{\theta}$ be the vector bundle over $X^{\theta}$ with fibre $L(D)/L$ at $D\in X^{(m)}$. Let $\bar\cL$ be the line bundle over $X^{\theta}$ with fibre 
\begin{equation}
\label{fibre_of_barcL_for_Sect214}
\frac{\det\RG(X, \cO_D)}{\det\RG(X, L(D)/L)}
\end{equation}
at $D$. Let $_L\tilde Z^{\theta}$ be the gerb of $n$-th roots of $\bar\cL$ over $_LZ^{\theta}$. Write $_LZ^{\theta}_{max}\subset {_LZ^{\theta}}$ for the open subscheme given by $v\in L(D)/L$ such that for any $0\le D'<D$, $v\notin L(D')/L$. Let 
$$
_LZ^{\theta, rss}_{max}\subset  {_LZ^{\theta}_{max}}
$$ 
be the open subscheme given by the property that $D$ is multiplicity free.  
 
  If $D=\sum_i x_i\in X^{\theta}$ with $x_i$ pairwise different then the fibre of $\bar\cL$ at $D$ is $\otimes_i (L^{-1}\otimes\Omega)_{x_i}$, where each 
$(L^{-1}\otimes\Omega)_{x_i}$ is of parity zero, so the order does not matter. Besides, $L(D)/L=\oplus_i L(x_i)/L$. So, a point $v\in {_LZ^{\theta, rss}_{max}}$ is a collection $0\ne v_i\in L(x_i)/L$ for all $i$. This gives a trivialization of each line $L(x_i)/L$, hence also a trivialization of $\bar\cL_D$ as the tensor product thereof. So, we get a trivialization of $\bar\cL$ over $_LZ^{\theta, rss}_{max}$.
  
  We denote by $_L\tilde Z^{\theta, rss}_{max}$ the restriction of the gerb $_L\tilde Z^{\theta}$ to this open subscheme. The above trivialization yields an isomorphism
\begin{equation}
\label{iso_for_Section214}
_LZ^{\theta, rss}_{max}\times B(\mu_n)\,\iso\, {_L\tilde Z^{\theta, rss}_{max}}
\end{equation}
Consider $\IC\boxtimes \cL_{\bar\zeta^{-1}}$ as a perverse sheaf on $_L\tilde Z^{\theta, rss}_{max}$ via (\ref{iso_for_Section214}). Its intermediate extension to $_L\tilde Z^{\theta}$ is denoted $\IC_{_LZ^{\theta}, \bar\zeta}$.
For a point 
$$
(D, v\in L(D)/L, \cU)\in {_L\tilde Z^{\theta}}
$$ 
with $\cU^n\,\iso\, \bar\cL_D$ the element $a\in \mu_n(k)\subset \Aut(\cU)$ acts on $\IC_{_LZ^{\theta}, \bar\zeta}$ as $\bar\zeta^{-1}(a)$. 

 The dual vector bundle $_L\check{Z}^{\theta}\to X^{\theta}$ has the fibre $L^{-1}\otimes\Omega/L^{-1}\otimes\Omega(-D)$ at $D$. Let $_L\check{\tilde Z}^{\theta}$ be the gerb of $n$-th roots of $\bar\cL$ over $_L\check{Z}^{\theta}$. We define $_L\check{Z}^{\theta}_{max}$ similarly. 

 Define the open subscheme $_L\check{Z}^{\theta, rss}_{max}\subset {_L\check{Z}^{\theta}}$ and the gerb $_L\check{\tilde Z}^{\theta, rss}_{max}$ similarly. As above, we get a trivialization of $\bar\cL$ over $_L\check{Z}^{\theta, rss}_{max}$, hence an isomorphism
\begin{equation}
\label{iso_second_for_Section214}
_L\check{Z}^{\theta, rss}_{max}\times B(\mu_n)\,\iso\, {_L\check{\tilde Z}^{\theta, rss}_{max}}
\end{equation}
Let $\IC(\Sign)$ denote the $\IC$-sheaf of $_L\check{Z}^{\theta, rss}_{max}$ tensored by the inverse image of $\Sign$ from $X^{\theta, rss}$. 
Define $\IC_{_L\check{Z}^{\theta}, \bar\zeta}$ as the intermediate extension of $\IC(\Sign)\boxtimes \cL_{\bar\zeta^{-1}}$ to $_L\check{\tilde Z}^{\theta}$ using (\ref{iso_second_for_Section214}). Now the isomorphism (\ref{iso_over_cX}) yields an isomorphism
\begin{equation}
\label{iso_third_for_Section214}
\Four_{\psi}(\IC_{_LZ^{\theta}, \bar\zeta})\,\iso\, \cC_0^m\otimes \IC_{_L\check{Z}^{\theta}, \bar\zeta}
\end{equation}

 The schemes $_L Z^{\theta}$ for various $L$ are locally isomorphic in Zariski topology, so the description of $\IC_{_LZ^{\theta}, \bar\zeta}$ (and of $\IC_{_L\check{Z}^{\theta}, \bar\zeta}$) is independent of $L$. The fibres of $\IC_{_LZ^{\theta}, \bar\zeta}$ are completely described by Corollary~\ref{Cor_3}. 
 
 As $L$ varies in $\Bun_1$, the schemes $_LZ^{\theta}$ form a family $_{\Bun_T}Z^{\theta}\to \Bun_1\times X^{\theta}$, whose fibre over $(L, D)$ is $L(D)/L$. We still denote by $\bar\cL$ the line bundle over $\Bun_1\times X^{\theta}$ with fibre (\ref{fibre_of_barcL_for_Sect214}) over $(L,D)$. Denote by $_{\Bun_T}\wt Z^{\theta}$, $_{\Bun_T}\check{Z}^{\theta}$ and $_{\Bun_T}\check{\tilde Z}^{\theta}$ the corresponding relative  versions over $\Bun_T$.  
 
  We have an automorphism $\tau_Z$ of $\Bun_T\times X^{\theta}$ sending $(L,D)$ to $(L':=L^{-1}\otimes\Omega(-D), D)$. It lifts to a diagram of isomorphisms $$
\begin{array}{ccc}  
_{\Bun_T}\check{Z}^{\theta} & \iso &  {_{\Bun_T}Z^{\theta}}\\
\downarrow && \downarrow\\
\Bun_T\times X^{\theta} & \toup{\tau_Z} & \Bun_T\times X^{\theta}
\end{array}
$$
sending $(L, D, v\in L^{-1}\otimes\Omega/L^{-1}\otimes\Omega(-D))$ to $(L', D, v\in L'(D)/L')$, where $L'$ is as above. Let $\vartriangle\subset X^{\theta}$ denote the divisor of diagonals. One has canonically $\bar\cL\otimes\tau_Z^*(\bar\cL)\,\iso\, \pr_2^*\cO(-\!\vartriangle)$, where $\pr_2: \Bun_T\times X^{\theta}\to X^{\theta}$ is the projection. Recall that $\cO(-\!\vartriangle)$ identifies canonically with the line bundle whose fibre at $D\in X^{\theta}$ is $\det\RG(X,\cO_D)^2$. For $n=2$ this yields an isomorphisms
$$
\bar\tau_Z: {_{\Bun_T}\check{\tilde Z}^{\theta}}\,\iso\,  {_{\Bun_T}\tilde Z^{\theta}}
$$
and
\begin{equation}
\label{iso_forSection_214_in_families}
\bar\tau_Z^*\IC_{_{\Bun_T}Z^{\theta}, \bar\zeta}\,\iso\, \IC_{_{\Bun_T}\check{Z}^{\theta}, \bar\zeta}
\end{equation}
for the corresponding relative versions. Thus, for $n=2$ the description of $\IC_{_L\check{Z}^{\theta}, \bar\zeta}$ is reduced to that of $\IC_{_LZ^{\theta}, \bar\zeta}$, it was studied in \cite{L2}. However, for $n\ge 3$ the situation is very different.

 For the rest of Section~\ref{Section_Generalization} assume $L=\Omega$. Then $\bar\cL$ is canonically identified with $\cO(-\!\!\vartriangle)$ over $X^{\theta}$. Let $_{\Omega}\tilde X^{\theta}$ be the gerb of $n$-th roots of $\bar\cL$ over $X^{\theta}$. The fibre of ${_{\Omega}\check{Z}^{\theta}}\to X^{\theta}$ over $D\in X^{\theta}$ is $\cO_D$. Let 
$$
\pi_n: {_{\Omega}\check{Z}^{\theta}}\to {_{\Omega}\check{Z}^{\theta}}
$$ 
be the map sending $(D, v\in \cO_D)$ to $(D, v^n\in \cO_D)$. Over $_{\Omega}\check{Z}^{\theta}_{max}$ this map is finite. Let $\GG^{\theta}_n$ denote the kernel of the homomorphism $\GG^{\theta}\to\GG^{\theta}$, $g\mapsto g^n$. This is a group scheme over $X^{\theta}$. Let $\GG^{\theta}_n$ act on $_{\Omega}\check{Z}^{\theta}$ so that $g\in (\cO/\cO(-D))^*$ sends $(D, v\in \cO_D)$ to $(D, gv)$. The map $\pi_n$ is $\GG^{\theta}_n$-invariant. The restriction 
\begin{equation}
\label{torsor_pi_n}
\pi_n: {_{\Omega}\check{Z}^{\theta}_{max}}\to {_{\Omega}\check{Z}^{\theta}_{max}}
\end{equation}
is a $\GG^{\theta}_n$-torsor. We have the line bundle on $X^{\theta}$ with fibre $\det\RG(X, \cO_D)$, the group scheme $\GG^{\theta}_n$ acts on this line bundle by a character that we denote $\eta_n: \GG^{\theta}_n\to \Gm$. It actually takes values in $\mu_n$. Let $\check{W}_{max}$ denote the local system on $_{\Omega}\check{Z}^{\theta}_{max}$ obtained from the torsor (\ref{torsor_pi_n}) as the extension of scalars via $\GG^{\theta}_n\toup{\eta_n} \mu_n(k)\toup{\bar\zeta}\Qlb^*$. 
 
 Let $_{\Omega}\tilde X^{\theta, rss}$ denote the restriction of the gerb $_{\Omega}\tilde X^{\theta}$ to $X^{\theta, rss}$. Since $\bar\cL$ over $X^{\theta, rss}$ is canonically trivialized, one has a canonical isomorphism
\begin{equation}
\label{triv_of_Omega_tildeXtheta_rss}
X^{\theta, rss}\times B(\mu_n)\,\iso\, {_{\Omega}\tilde X^{\theta, rss}}
\end{equation}  
Viewing $(\IC(X^{\theta, rss})\otimes\Sign)\boxtimes \cL_{\bar\zeta^{-1}}$ as a perverse sheaf on $_{\Omega}\tilde X^{\theta, rss}$ via (\ref{triv_of_Omega_tildeXtheta_rss}), let $\IC_{_{\Omega}\tilde X^{\theta}, \bar\zeta}$ denote its intermediate extension to $_{\Omega}\tilde X^{\theta}$. Let $\check{\pi}: {_{\Omega}\check{\tilde Z}^{\theta}}\to {_{\Omega}\tilde X^{\theta}}$ denote the projection sending $(D, v\in \cO/\cO(-D),\cU)$ to $(D,\cU)$. 

\begin{Pp} 
\label{Lm_6.2_about_Fourier_of_IC_zeta}
There is an isomorphism over $_{\Omega}\check{\tilde Z}^{\theta}_{max}$
$$
\check{W}_{max}\otimes \check{\pi}^*\IC_{_{\Omega}\tilde X^{\theta}, \bar\zeta}[\dimrel(\check{\pi})]\,\iso\, \IC_{_{\Omega}\check{Z}^{\theta}, \bar\zeta}
$$
\end{Pp}
\begin{Prf}
The intermediate extension commutes with a smooth base change. So, it suffices to establish this isomorphism over $_{\Omega}\check{\tilde Z}^{\theta, rss}_{max}$, where it is easy.
\end{Prf}

\medskip
  
\begin{Rem} The restriction of $\bar\cL$ to the principal diagonal $X\subset X^{\theta}$ identifies with $\Omega^{m(m-1)}$. For $x\in X$ the group $\Aut(\Omega_x)$ acts on the fibre of $\bar\cL$ at $D=mx$. So, if the $*$-fibre of $\IC_{_{\Omega}\tilde X^{\theta}, \bar\zeta}$ at $mx$ does not vanish then $n$ divides $m(m-1)$. In particular, if $n$ is big enough compared to $m$ then $\IC_{_{\Omega}\tilde X^{\theta}, \bar\zeta}$ is the extension by zero under ${_{\Omega}\tilde X^{\theta, rss}}\hook{} {_{\Omega}\tilde X^{\theta}}$.
\end{Rem}

\begin{Rem} i) The perverse sheaf $\IC_{_{\Omega}\tilde X^{\theta}, \bar\zeta}$ has been studied in \cite{BFS} (see also \cite{FS}, Section~5.1). It satisfies the natural factorization property, and all its fibres are described in \cite{BFS}. The version of $\IC_{_{\Omega}\tilde X^{\theta}, \bar\zeta}$ in the world of twisted $\cD$-modules is exactly the sheaf denoted by $\cL^{\mu}_{\emptyset}$ in (\cite{Ga2}, Section~3.4) for $G=\SL_2$, $\mu=-m\alpha$. 

\smallskip\noindent
ii) The perverse sheaf $\IC_{_{\Omega}\tilde X^{\theta}, \bar\zeta}$ also appears in \cite{L5} under the name $\cL^{\mu}_{\emptyset}$  for $G=\SL_2$, $\mu=-m\alpha$. Note that for $n>1$ the so-called subtop cohomology property is satisfied for our metaplectic data for $\SL_2$ by (\cite{L5}, Theorem~1.1.6). So, for $n>1$ the perverse sheaf $\IC_{_{\Omega}\tilde X^{\theta}, \bar\zeta}$ identifies with the $!$-direct image under $_{\Omega}{\wt Z}^{\theta}\to {_{\Omega}\tilde X^{\theta}}$
by (\cite{L5}, Proposition~4.11.1). Here $_{\Omega}{\wt Z}^{\theta}$ is the stack from Section~\ref{Section_Generalization}.
\end{Rem}

\subsubsection{} For any $\cE\in\Bun_1$ taking $L=\cE^2$ we get $Z^{\theta}_{\cE}={_LZ^{\theta}}$ and $\tilde Z^{\theta}_{\cE}={_L\tilde Z^{\theta}}$. For $L=\cE^2$ set $\IC_{\check{Z}^{\theta}_c, \bar\zeta}=\IC_{_L\check{Z}^{\theta}, \bar\zeta}$. By (\ref{iso_third_for_Section214}) one has
\begin{equation}
\label{iso_Four_transform_forSect215}
\Four_{\psi}(\IC_{Z^{\theta}_c, \bar\zeta})\,\iso\, \cC^m_0\otimes\IC_{\check{Z}^{\theta}_c, \bar\zeta}
\end{equation}

\subsubsection{Global calculation} Let $\cS_B$ denote the stack classifying $\cE\in\Bun_1$ and $s_2: \cE^2\to\Omega$. Let $\nu_B:\Bun_B\to\Bun_G$ be the natural map. Let $\Bun_{\tilde B, c}$ be the stack classifying (\ref{ext_cE^-1_by_cE}) and a line $\cU_G$ equipped with 
\begin{equation}
\label{iso_for_cU_G^n_for_Section216}
\cU_G^n\,\iso\, (\cL_c)_{\cE}
\end{equation}

 Let $\cS_{\tilde B, c}$ be the stack classifying $(\cE, s_2)\in\cS_B$, and a line $\cU_G$ equipped with (\ref{iso_for_cU_G^n_for_Section216}). We have the Fourier transform 
$$
\Four_{\psi}: \D(\Bun_{\tilde B, c})\,\iso\, \D(\cS_{\tilde B, c})
$$
The map $\nu_B$ lifts to a map $\nu_{\tilde B}: \Bun_{\tilde B, c}\to \Bunt_{G, \cL_c}$ sending $\cU_G$ and (\ref{ext_cE^-1_by_cE}) to $(M,\cU_G)$.  

 Recall that $Z^{\theta}_{\Bun_T}$ classifies $D\in X^{\theta},\cE\in\Bun_1, v\in \cE^2(D)/\cE^2$ giving rise to the diagram (\ref{diag_point_of_Z^theta}). We have the dual vector bundles $Z^{\theta}_{\Bun_T}\to X^{\theta}\times\Bun_T\gets \check{Z}^{\theta}_{\Bun_T}$. Let 
$$
f_B: Z^{\theta}_{\Bun_T}\to\Bun_B\times X^{\theta}
$$ 
be the map sending (\ref{diag_point_of_Z^theta}) to the exact sequence (\ref{ext_cE^-1_by_cE}) together with $D\in X^{(m)}=X^{\theta}$. This is a morphism of generalized vector bundles over $\Bun_T\times X^{\theta}$ given by $\cE^2(D)/\cE^2\to \H^1(X, \cE^2)$. The dual map over $\Bun_T\times X^{\theta}$ is denoted
$$
\check{f}_B: \cS_B\times X^{\theta}\to \check{Z}^{\theta}_{\Bun_T},
$$
it sends $(\cE, s_2, D)$ to $(D, \cE, v\in \cE^{-2}\otimes\Omega/\cE^{-2}\otimes\Omega(-D))$, where $v$ is the image of $s_2$ under the transpose map
$$
\H^0(X, \cE^{-2}\otimes\Omega)\to  \cE^{-2}\otimes\Omega/\cE^{-2}\otimes\Omega(-D)
$$

 Let $\cS^0_B\subset\cS_B$ be the open substack given by $s_2\ne 0$. Let $\cS^0_{\tilde B, c}$ be its preimage in $\cS_{\tilde B, c}$. Let $\Bunt_{T,c}$ be the gerb of $n$-th roots of $\cL_c$ over $\Bun_T$. For $K\in \D_{\bar\zeta}(\Bunt_{T, c})$ let us describe
\begin{equation}
\label{the_complex_afer_Four_we_look_Sect216}
\Four_{\psi}\nu_{\tilde B}^*\Eis(K)[\dimrel(\nu_{\tilde B})]\mid_{\cS^0_{\tilde B, c}}
\end{equation} 
Write $(\cE_1\hook{}M, \cU_G,\cU)$ for a point of $\Bunb_{\tilde B, c}$, here $\cU_G^n\,\iso\, (\cL_c)_M$ and $\cU^n\,\iso\, (\cL_c)_{\cE_1}$. 
  
  Denote by $\tilde Z^{\theta}_{\Bun_T}$ the stack classifying a point $(\cE, D, v\in \cE^2(D)/\cE^2)\in Z^{\theta}_{\Bun_T}$ and a line $\bar\cU$ equipped with 
\begin{equation}
\label{line_bar_cU_for_Sect216}
\bar\cU^n\,\iso\, \frac{(\cL_c)_{\cE(D)}}{(\cL_c)_{\cE}}
\end{equation}
Let $\check{\tilde Z}^{\theta}_{\Bun_T}$ be the stack classifying $(\cE, D, v'\in \cE^{-2}\otimes\Omega/\cE^{-2}\otimes\Omega(-D))\in \check{Z}^{\theta}_{\Bun_T}$ and a line $\bar\cU$ equipped with (\ref{line_bar_cU_for_Sect216}).
The perverse sheaves $\IC_{Z^{\theta}_c, \bar\zeta}$ as $\cE$ varies in $\Bun_T$ naturally form a family, which is a perverse sheaf on $\tilde Z^{\theta}_{\Bun_T}$ still denoted $\IC_{Z^{\theta}_c, \bar\zeta}$ by abuse of notations. 
Similarly, we denote by $\IC_{\check{Z}^{\theta}_c, \bar\zeta}$ the corresponding perverse sheaf over $\check{\tilde Z}^{\theta}_{\Bun_T}$. The isomorphism (\ref{iso_Four_transform_forSect215}) naturally extends to the stack $\check{\tilde Z}^{\theta}_{\Bun_T}$. 

 Write $(\cS_B\times X^{\theta})^{\tilde{}}$ for the stack classifying $(\cE, s_2)\in \cS_B$ and lines $\cU_G,\cU$ equipped with (\ref{iso_for_cU_G^n_for_Section216}) and 
\begin{equation}
\label{iso_for_cU^n_for_Sect216}
\cU^n\,\iso\, (\cL_c)_{\cE(D)}
\end{equation}
The map $\check{f}_B$ extends to a morphism $\check{f}_{\tilde B}: (\cS_B\times X^{\theta})^{\tilde{}}\to \check{\tilde Z}^{\theta}_{\Bun_T}$ given by $\bar\cU=\cU\otimes\cU_G^{-1}$. We have a diagram of projections
\begin{equation}
\label{diag_for_Pp6.2}
\Bunt_{T, c}\;\;\getsup{\pr_{T,c}}\;\;\;(\cS_B\times X^{\theta})^{\tilde{}}  \;\;\;\toup{\pr_{B,c}} \;\cS_{\tilde B, c},
\end{equation}
where $\pr_{B,c}$ sends the above point to $(\cE, s_2, \cU_G)$, and the map $\pr_{T,c}$ sends the above point to $(\cE^{-1}(-D), \cU)$ equipped with (\ref{iso_for_cU^n_for_Sect216}). From the standard properties of the Fourier transform one gets the following.

\begin{Pp} 
\label{Pp_nondegenerate_Whit_coeff_of_Eis}
Let $m\ge 0$, $\theta=m\alpha$. Over the connected component of $\cS^0_{\tilde B, c}$ given by fixing $\deg\cE$, the contribution of the connected component of $\Bunb_{\tilde B, c}$ given by $\deg\cE_1=-\deg\cE-m$ to the complex (\ref{the_complex_afer_Four_we_look_Sect216}) identifies with
\begin{equation}
\label{complex_for_Lm15}
(\pr_{B,c})_!(\pr_{T,c}^*K\otimes \cC_0^m\otimes
(\check{f}_{\tilde B})^*\IC_{\check{Z}^{\theta}_c, \bar\zeta})[\dimrel(\check{f}_B)-\dim\Bun_T]
\end{equation}
This complex vanishes unless $e$ divides $m+\deg\cE$. $\square$
\end{Pp}

 Proposition~\ref{Pp_nondegenerate_Whit_coeff_of_Eis} implies the following description of the first Whittaker coefficients of $\Eis(K)$, $K\in \D_{\bar\zeta}(\Bunt_{T,c})$. Write $\Cov_B\subset \cS_B$ for the open substack given by requiring that $s_2:\cE^2\to\Omega$ is an isomorphism. 
Let $\Cov_{\tilde B, c}$ denote the restriction of the gerb $\cS_{\tilde B, c}$ to $\Cov_B$. The stack $\Cov_{\tilde B, c}$ is the base, on which the first Whittaker coefficient lives. Recall that the line bundle on $X^{\theta}\times\Cov_B$ whose fibre at $(D, \cE, s_2)\in X^{\theta}\times\Cov_B$ is
$$
\frac{(\cL_c)_{\cE(D)}}{(\cL_c)_{\cE}}
$$
identifies canonically with $\pr_1^*\cO(-\!\vartriangle)$. So, one gets the open immersion $\Cov_{\tilde B, c}\times(_{\Omega}\tilde X^{\theta}) \hook{} (\cS_B\times X^{\theta})^{\tilde{}}$ sending $(D, \cE, s_2, \cU_G, \bar\cU)$ with (\ref{iso_for_cU_G^n_for_Section216}) and (\ref{line_bar_cU_for_Sect216}) to $(\cE, s_2, D, \cU_G,\cU)$, where $\cU=\cU_G\otimes\bar\cU$. We get the diagram
$$
\Bunt_{T,c}\;\getsup{\pr_{\Cov,\theta}}\;
\Cov_{\tilde B, c}\times(_{\Omega}\tilde X^{\theta}) \;\toup{\pr_1}\; \Cov_{\tilde B, c}
$$
obtained by restricting (\ref{diag_for_Pp6.2}). 

\begin{Cor} 
\label{Cor_6.1}
The restriction of (\ref{complex_for_Lm15}) to $\Cov_{\tilde B, c}$ identifies with
\begin{equation}
\label{complex_Witten_conf_blocks}
(\pr_1)_!(\pr_{\Cov,\theta}^*K\otimes \cC_0^m\otimes \IC_{_{\Omega}\tilde X^{\theta}, \bar\zeta})[-\dim\Bun_T]
\end{equation}
\end{Cor}

\begin{Cor} 
\label{Cor_6.2}
Let $E$ be a $\check{T}^{\sharp}$-local system on $X$, $\cK_E\in \D_{\zeta}(\Bunt_T)$ be the $E^*$-Hecke eigen-sheaf as in Corollary~\ref{Cor_very_first}. The first Whittaker coefficient of $\Eis(\cK_E)$, that is, the 
complex 
$$
\Four_{\psi}\nu_{\tilde B}^*\Eis(\cK_E)[\dimrel(\nu_{\tilde B})]\mid_{\Cov_{\tilde B, c}}
$$
identifies with 
\begin{equation}
\label{complex_Witten_for_Eis(K_E)}
\mathop{\oplus}\limits_{\theta} (\pr_1)_!(\pr_{\Cov,\theta}^*\cK_E\otimes \cC_0^m\otimes \IC_{_{\Omega}\tilde X^{\theta}, \bar\zeta})
\end{equation}
Here $\theta=m\alpha$, and the sum is over $m\ge 0$ such that $m+g-1\in e\ZZ$. 
\end{Cor}

\begin{Rem} 
\label{Rem_Witten_conf_blocks}
i) The complex (\ref{complex_Witten_conf_blocks}) is an $\ell$-adic analog of the space of conformal blocks in Wess-Zumino-Witten model that was studied in \cite{BFS}.

\smallskip\noindent
ii) Assume $E^{n\alpha}$ nontrivial then $\cK_E$ is regular. The conjectural functional equation for $\Eis(\cK_E)$ should be reflected in  
the property of (\ref{complex_Witten_for_Eis(K_E)}) saying that the summand indexed by $\theta=m\alpha$ for $E$ identifies (up to tensoring by some 1-dimensional space) with the summand indexed by $\theta'=m'\alpha$ for $^{\sigma}E$, where $m+m'=(n-1)(2g-2)$. Here $^{\sigma}E$ is the extension of scalars of $E$ under $w_0:\check{T}^{\sharp}\to \check{T}^{\sharp}$. In particular, in this case the sum in (\ref{complex_Witten_for_Eis(K_E)}) should be over $0\le m\le (n-1)(2g-2)$ such that $m+g-1\in e\ZZ$.

\smallskip\noindent
iii) In the case $n=2$ the complex (\ref{complex_Witten_for_Eis(K_E)}) is calculated in (\cite{L2}, Theorem~4). In this case it is given by the `geometric central value' of some $L$-function. Recall the line bundle $\cE_X\in\Cov_B$ from Section~\ref{Section_Notations}. A point $(\cE, s_2)\in\Cov_B$ gives rise to the $\mu_2$-torsor on $X$ given by $s_2: (\cE\otimes\cE_X^{-1})^2\,\iso\, \cO_X$. Write $\cE_0$ for the $\Qlb$-local system on $X$ obtained from this $\mu_2$-torsor via the extension $\mu_2\to \Qlb^*$. For $n=2$ 
the fibre of (\ref{complex_Witten_for_Eis(K_E)}) over $(\cE, \cU_G)\in\Cov_{\tilde B, c}$ identifies with
$$
\oplus_{m\ge 0}\RG(X^{(m)}, (E^{-\alpha}\otimes\cE_0)^{(m)})[m]
$$
tensored by some 1-dimensional space. If $E^{-2\alpha}$ is nontrivial then the above identifies with 
$$
\mathop{\oplus}\limits_{0\le m\le 2g-2} \wedge^m \H^1(X, E^{-\alpha}\otimes\cE_0)
$$ 
tensored by some 1-dimensional space, and this agrees with ii). The compatibility with the functional equation then comes from the isomorphism $\H^1(X, E^{-\alpha}\otimes\cE_0)^*\,\iso\, \H^1(X, E^{\alpha}\otimes\cE_0)$. So, for $n>1$ we may think of (\ref{complex_Witten_for_Eis(K_E)}) as a generalization of the notion of the central value of an abelian $L$-function.
\end{Rem}

\subsection{Constant term of $\Eis$} 
\label{Section_Constant term of Eis}
Recall that $\Bun_{B, \tilde G}$ classifies an exact sequence (\ref{ext_cE^-1_by_cE}) and a line $\cU_G$ equipped with $\cU_G^N\,\iso\, \cL_{M}$. Write $\Bun_{B,\tilde G}^d$ for the connected component of $\Bun_{B,\tilde G}$ given by $\deg \cE=d$. We have the diagram of projections
$$
\Bunt_T \getsup{\gq} \Bun_{B,\tilde G} \toup{\gp} \Bunt_G,
$$
where $\gq$ sends $(\cE\hook{} M, \cU_G)$ to $(\cE,\cU_G)$. Write $\Bunt_T^d$ for the component of $\Bunt_T$ classifying $(\cE,\cU)\in\Bunt_T$ with $\deg\cE=d$. The constant term functor $\CT: \D_{\zeta}(\Bunt_G)\to \D_{\zeta}(\Bunt_T)$ is defined by $\CT=\gq_!\gp^*[\dimrel(\gp)]$. 

 Recall that $\Bunb_B$ classifies $M\in \Bun_G$ together with a subsheaf $\cE\hook{} M$, $\cE\in\Bun_1$. Write $\Bunb^d_B$ for the connected component of $\Bunb_B$ given by $\deg \cE=d$. The stack $\Bunb^d_B$ is smooth irreducible of dimension $2g-2-2d$.  
 
  Let $\sigma: \Bun_T\to \Bun_T$ be the map $\cE\mapsto \cE^{-1}$. We also denote by $\sigma: \Bunt_T\to\Bunt_T$ the map $(\cE,\cU)\mapsto (\cE^{-1},\cU)$.
 
\begin{Def} For $\theta=m\alpha\in \Lambda^{pos}_{G,B}$ with $m\in n\ZZ_+$ we define the following \select{integral Hecke functor} $\IH^{\theta}: \D_{\zeta}(\Bunt_T)\to \D_{\zeta}(\Bunt_T)$. One has $\gB(\theta)=\frac{m}{n}\nu$, where $\nu=n\alpha\in J$, and $X^{\gB(\theta)}\,\iso\, X^{(\frac{m}{n})}$. Recall the stack $X^{\gB(\theta)}\times_{X^{\theta}} \Modt^{+,\theta}_{\Bunt_T}$ from Section~\ref{Section_211}. Its point is a collection $(\cE\in \Bun_1, D\in X^{(\frac{m}{n})}, \cU,\cU_G)$ together with isomorphisms $\cU_G^N\,\iso\, (\cL_T)_{\cE(nD)}$, $\cU^N\,\iso\, (\cL_T)_{\cE}$. Here $\cE(nD)$ is the `background' $T$-torsor. Consider the diagram
$$
\Bunt_T \;\getsup{h^{\la}_T}\; X^{\gB(\theta)}\times_{X^{\theta}} \Modt^{+,\theta}_{\Bunt_T} \;\toup{h^{\ra}_T}\; \Bunt_T,
$$
where $h^{\la}_T$ sends the above point to $(\cE,\cU)$, and $h^{\ra}_T$ sends the above point to $(\cE(nD), \cU_G)$. Set 
$$
\IH^{\theta}(K)=(h^{\ra}_T)_!((h^{\la}_T)^*K\otimes \IC^{\gB(\theta)}_{\Bun_M, \zeta})[-\dim\Bun_T]
$$ 
\end{Def}

\begin{Pp} 
\label{Pp_CT_of_Eis}
Let $d_1\in e\ZZ$ and $K\in \D_{\zeta}(\Bunt^{d_1}_T)$. The complex 
$K_{d,d_1}:=\CT(\Eis(K))\mid_{\Bunt^d_T}$ vanishes unless $d-d_1\in n\ZZ$. In the latter case it is described as follows. 
\begin{itemize}
\item[1)] If $d_1> max\{d, -d\}$ then $K_{d,d_1}=0$. 
\item[2)] If $d<d_1\le -d$ then for $\theta:=-(d+d_1)\alpha\in n\alpha\ZZ$
$$
K_{d,d_1}\,\iso\, \sigma_!\IH^{\theta}(K)[-\mid\!\gB(\theta)\!\mid]
$$ 

\item[3)] If $d\ge d_1>-d$ then for $\theta:=(d-d_1)\alpha\in n\alpha\ZZ$
$$
K_{d, d_1}\,\iso\, \IH^{\theta}(K)[2-2g+\mid\!\gB(\theta)\!\mid]
$$
\item[4)] If $d_1\le \min\{d, -d\}$ then there is a distinguished triangle
$$
\sigma_!\IH^{\theta}(K)[-\mid\!\gB(\theta)\!\mid]\to K_{d,d_1}\to \IH^{\theta'}(K)[2-2g+\mid\!\gB(\theta')\!\mid]
$$
with $\theta=-(d+d_1)\alpha$ and $\theta'=(d-d_1)\alpha$. 
\end{itemize}
\end{Pp}
\begin{Prf} We calculate the direct image with respect to the composition 
$\Bun_{B,\tilde G}\times_{\Bunt_G} \Bunb_{\tilde B}\toup{\tilde\gp} \Bun_{B,\tilde G}\toup{\gq} \Bunt_T$. Write a point of $\Bun_{B,\tilde G}\times_{\Bunt_G} \Bunb_{\tilde B}$ as (\ref{ext_cE^-1_by_cE}) together with a subsheaf $\cE_1\hook{} M$ and lines $\cU,\cU_G$ equipped with $\cU^N\,\iso\, (\cL_T)_{\cE_1}$, $\cU_G^N\,\iso\, (\cL_T)_{\cE}$. \\
1) In this case $\Hom(\cE_1, \cE)=\Hom(\cE_1, \cE^{-1})=0$.\\
2) In this case $\Hom(\cE_1, \cE)=0$, and there remains the integral over the open substack $\tilde Z^{\theta}_{\Bunt^d_T}\subset \Bun_{B,\tilde G}\times_{\Bunt_G} \Bunb_{\tilde B}$ given by the conditions that $\cE_1\to\cE^{-1}$ is injective, $\deg\cE_1=d_1$, $\deg \cE=d$. Here $\theta=-(d+d_1)\alpha$. 

 Let $\Modt^{+,\theta}_{\Bunt^d_T}$ be the stack classifying $\cE\in \Bun_T^d$, $D_1\in X^{\theta}$ and lines $\cU,\cU_G$ equipped with $\cU^N\,\iso\, (\cL_T)_{\cE(D_1)}$, $\cU_G^N\,\iso\, (\cL_T)_{\cE}$. Here $\cE^{-1}$ is the `background' $T$-torsor. Let $\pi_B: \tilde Z^{\theta}_{\Bunt^d_T}\to \Modt^{+,\theta}_{\Bunt^d_T}$ be the natural projection. By Theorem~\ref{Th_3},  $K_{d,d_1}$ vanishes unless $d+d_1\in n\ZZ$, and in the latter case it is as follows. One gets $X^{\gB(\theta)}=X^{(\frac{-d-d_1}{n})}$. The corresponding map $X^{\gB(\theta)}\to X^{\theta}$ sends $D$ to $D_1=nD$. 
Consider the diagram
$$
\Bunt_T\,\getsup{h^{\la}_T}\,  X^{\gB(\theta)}\times_{X^{\theta}}\Modt^{+,\theta}_{\Bunt^d_T}\,\toup{\sigma h^{\ra}_T}\, \Bunt^d_T,
$$
where $h^{\ra}_T$ sends $(D,\cE, \cU,\cU_G)$ to $(\cE^{-1}, \cU_G)$, and $h^{\la}_T$ sends this point to $(\cE^{-1}(-nD), \cU)$. By Theorem~\ref{Th_3}, 
$$
K_{d,d_1}\,\iso\, (\sigma h^{\ra}_T)_!( \IC^{\gB(\theta)}_{\Bun_M, \zeta}\otimes(h^{\la}_T)^*K)[\frac{d+d_1}{n}-\dim\Bun_T]
$$ 
3) In this case $\Hom(\cE_1, \cE^{-1})=0$. Let $\theta=(d-d_1)\alpha$. By Corollary~\ref{Cor_3}, $K_{d, d_1}$ vanishes unless $d-d_1\in n\ZZ$. In the latter case $\mid\!\gB(\theta)\!\mid=\frac{d-d_1}{n}$ and 
$$
K_{d,d_1}\,\iso\, \IH^{\theta}(K)[2-2g+\mid\!\gB(\theta)\!\mid]
$$
4) Stratify $\Bun_{B,\tilde G}\times_{\Bunt_G} \Bunb_{\tilde B}$ by the property that $\cE_1$ factors through $\cE$ or not. Calculate the direct image with respect to this stratifications. 
\end{Prf}

\medskip

 Let $E$ be a $\check{T}^{\sharp}$-local system on $X$. Write $\cK=\cK_E$ for the Hecke eigen-sheaf on $\Bunt_T$ associated to $E$ in (\cite{L}, Proposition~2.2). This is a local system over the components of $\Bunt_T$ corresponding to $\Lambda^{\sharp}$. 
 
\begin{Lm} Let $\nu=-n\alpha$ and $\theta=m\alpha$ with $m\in n\ZZ_+$. One has naturally $\IH^{\theta}(\cK_E)\,\iso\, \cK_E\otimes \RG(X^{(\frac{m}{n})}, (E^{\nu})^{(\frac{m}{n})})[\frac{m}{n}]$.
\end{Lm}

\begin{Rem} By (\cite{L2}, Lemma~15), one has canonically $(AE)_{\Omega^{-n\alpha}}\,\iso\, \det\RG(X, E^{n\alpha})$. Here $\Omega^{-n\alpha}$ is the $T^{\sharp}$-torsor obtained from $\Omega$ via the push-out by $-n\alpha: \Gm\to T^{\sharp}$. Denote also by $^{\sigma}E$ the extension of scalars of $E$ under $w_0:\check{T}^{\sharp}\to \check{T}^{\sharp}$. Now from Proposition~\ref{Pp_action_of_W_on_cK_E} one gets $\sigma\ast\cK_E\,\iso\, \cK_{(^{\sigma}E)}\otimes \det\RG(X, E^{n\alpha})$. Here $\sigma\ast$ stands for the twisted $W$-action (\ref{def_twisted_W_action_on_Bunt_T}). It follows that Proposition~\ref{Pp_CT_of_Eis} is consistent with the functional equation for $\Eis(\cK_E)$ from Conjecture~\ref{Con_functional_equation}.
\end{Rem}

\subsection{Some special sheaves}  
\label{Section_Some special sheaves}
Let $E$ be a $\check{T}^{\sharp}$-local system on $X$. Sometimes we think of it simply as the rank one local system on $X$ corresponding to the character $e\alpha$ of $\check{T}^{\sharp}$. Write $\cK=\cK_E$ for the Hecke eigen-sheaf on $\Bunt_T$ associated to $E$ in (\cite{L},  Proposition~2.2). This is a local system over the components of $\Bunt_T$ corresponding to $\Lambda^{\sharp}$. 

Recall the Shatz stratification of $\Bun_G$. Let $Shatz^0\subset\Bun_G$ be the open substack of semi-stable torsors. So, $M\in Shatz^0$ iff for any rank one subsheaf $L\subset M$ one has $\deg L\le 0$. 
For $d>0$ let $Shatz^d$ denote the stack classifying $\cE\in \Bun_1^d$ and an exact sequence $0\to \cE\to M\to \cE^{-1}\to 0$. The map $Shatz^d\to \Bun_G$ sending this point to $M$ is a locally closed immersion. Besides, $Shatz^d$ for $d\ge 0$ form a stratification of $\Bun_G$. The stack $Shatz^d$ is irreducible of dimension $2g-2-2d$ for $d>0$, and $\dim Shatz^0=\dim\Bun_G=3g-3$. 
 
 Let $\Bun_G^{st}\subset \Bun_G$ denote the open substack of stable sheaves. It is not empty for $g\ge 2$. The image $\Bun^{\le 0}_G$ of $\Bunb^0_B\to\Bun_G$ is the complement of $\Bun_G^{st}$ in $\Bun_G$. For $d>0$ the image of $\bar\gp: \Bunb^d_B\to\Bun_G$ is the closure $\ov{Shatz}^d$ of $Shatz^d$. Let $Shatz^d_{\tilde G}$ (resp., $\ov{Shatz}^d_{\tilde G}$) be obtained from $Shatz^d_G$ (resp., $\ov{Shatz}^d$) by the base change $\Bunt_G\to\Bun_G$. 
 
 For $d>0$ a point of $Shatz^d_{\tilde G}$ is given by (\ref{ext_cE^-1_by_cE}) together with a line $\cU_G$ equipped with 
\begin{equation} 
\label{U_G^N_iso} 
\cU_G^N\,\iso\, (\cL_T)_{\cE}
\end{equation}
For $d>0$ let $\tilde q: Shatz^d_{\tilde G}\to \Bunt_T$ be the map sending the above point to $(\cE, \cU_G)$ equipped with (\ref{U_G^N_iso}). 
\begin{Def} Let $\IC(E, d)$ denote the intermediate extension of $\tilde q^*\cK_E[\dim Shatz^d_{\tilde G}]$ under $Shatz^d_{\tilde G}\hook{} \ov{Shatz}^d_{\tilde G}$. Note that $\IC(E, d)\in \D_{\zeta}(\Bunt_G)$. 
\end{Def}

Recall the stack $\Bunb^d_B$ from Section~\ref{Section_Constant term of Eis}.  For $K\in \D_{\zeta}(\Bunt_T)$ write $\Eis^d(K)$ for the contribution of the component $\Bunb^d_{\tilde B}$ to $\Eis(K)$. Recall that $\Eis^d(\cK)$ vanishes unless $d\in e\ZZ$. 
\begin{Pp}  
\label{Pp_special_sheaves_IC(E,d)}
Let $d>0$ with $d\in e\ZZ$. \\
(1) If $E^{n\alpha}$ is not trivial then $\Eis^d(\cK)\,\iso\, \IC(E, d)$ canonically.\\
(2) If $E^{n\alpha}$ is trivial then 
$$
\Eis^d(\cK)\,\iso\, \oplus_{b\ge 0} \,\IC(E, d+nb)
$$
\end{Pp}
\begin{Prf} The map $\bar\gp: \Bunb^d_B\to\Bun_G$ is an isomorphism over $Shatz^d_G$. It follows that $\IC(E, d)$ appears in $\Eis^d(\cK)$ with multiplicity one. Consider a point $(\cE\hook{} M, \cU_G)\in Shatz^r_{\tilde G}$ for some $r>d$. The fibre of $\bar\gp: \Bunb^d_B\to\Bun_G$ over this point identifies with 
$X^{(r-d)}$. Namely, to $D\in X^{(r-d)}$ there corresponds the subsheaf $\cE(-D)\subset M$. Denote by $S$ the $*$-fibre of $\Eis^d(\cK)$ at this point.
By Corollary~\ref{Cor_3}, $S$ vanishes unless $r-d\in n\ZZ$. If $r-d\in n\ZZ$ then we get an isomorphism
$$
S\,\iso\, \cK_{(\cE, \cU_G)}\otimes \RG(X^{(\frac{r-d}{n})}, (E^{-n\alpha})^{(\frac{r-d}{n})})[-2r+2g-2+\frac{2(r-d)}{n}]
$$
The codimension of $Shatz^r$ in $\ov{Shatz}^d$ is $2(r-d)$. If $E^{n\alpha}$ is not trivial then the $*$-restriction of $\Eis^d(\cK)$ to $Shatz^r$ is placed in perverse degrees $<0$. Part (1) follows. Under the assumption of (2) we see that $\IC(E, r)$ appears in $\Eis^d(\cK)$ with multiplicity one. Our claim follows. 
\end{Prf}

\medskip
\begin{Rem} 
\label{Rem_parities_of_IC_d_Section2.2}
If $n$ is even then by Lemma~\ref{Lm_2-action_for_Bunt_T} one has the following. If $d\in n\ZZ$ (resp., $d\in e\ZZ$ and $d\notin n\ZZ$) then $\IC(E, d)\in \D_{\zeta,+}(\Bunt_G)$ (resp., $\IC(E, d)\in \D_{\zeta,-}(\Bunt_G)$).
\end{Rem}

\subsubsection{Case $g=0$} In this subsection we assume $g=0$. For $d>0$ set for brevity $\IC_d=\IC(\Qlb, d)$. The open substack $Shatz^0\subset \Bun_G$ classifies trivial $G$-torsors.

\begin{Def} The line bundle $\cL_c$ is trivial on $Shatz^0$, its trivialization yields an isomorphism 
$Shatz^0_{\tilde G}\,\iso\, Shatz^0\times B(\mu_N)$. View 
$\IC(Shatz^0)\boxtimes \cL_{\zeta}$ as a perverse sheaf on $Shatz^0_{\tilde G}$ via this isomorphism. Let $\IC_0$ be its intermediate extension to $\Bunt_G$. 
\end{Def}

\begin{Lm} 
\label{Lm_restricting_IC_d_case_d>0}
Assume $d>0$ with $d\in e\ZZ$. Let $r>d$. The $*$-restriction $\IC_d\mid_{Shatz^r_{\tilde G}}$ vanishes unless $r-d\in n\ZZ$. If $r-d\in n\ZZ$ then 
$$
\IC_d\mid_{Shatz^r_{\tilde G}}\,\iso\, \IC_r[\frac{2(r-d)}{n}]\mid_{Shatz^r_{\tilde G}}
$$
\end{Lm}
\begin{Prf}
By Proposition~\ref{Pp_special_sheaves_IC(E,d)}, $\Eis^d(\cK)\,\iso\, \Eis^{d+n}(\cK)\oplus \IC_d$. Restricting this isomorphism to $Shatz^r_{\tilde G}$, one obtains the desired result as in Proposition~\ref{Pp_special_sheaves_IC(E,d)}.
\end{Prf}

\begin{Lm} 
\label{Lm_great_fibres_of_IC_0}
(1) One has $\Eis^0(\cK)\,\iso\, \IC_0[1]\oplus \IC_0[-1]\oplus (\oplus_{b\ge 0} \, \IC_{2n+bn})$. \\
(2) The $*$-restriction $\IC_0\mid_{Shatz^r_{\tilde G}}$ vanishes unless $r\in n\ZZ$. For $r\in n\ZZ$ and $r>0$ one has
$$
\IC_0\mid_{Shatz^r_{\tilde G}}\,\iso\, \IC_r[\frac{2r}{n}-1]\mid_{Shatz^r_{\tilde G}}
$$
\end{Lm}
\begin{Prf}
For $(M, \cU_G)\in Shatz^0_{\tilde G}$ the fibre of $\bar\gp: \Bunb^0_B\to\Bun_G$ over $M$ is isomorphic to $\PP^1$. So, $\Eis^0(\cK)\,\iso\, \IC_0[1]\oplus \IC_0[-1]$ over $Shatz^0_{\tilde G}$. For $d>0$ and a point $(\cE\subset M, \cU_G)\in Shatz^d_{\tilde G}$ the fibre of $\bar\gp: \Bunb^0_G\to\Bun_G$ over $M$ identifies with $X^{(d)}$. To $D\in X^{(d)}$ there corresponds $(\cE(-D)\subset M)\in \Bunb^0_B$. Now arguing as in Proposition~\ref{Pp_special_sheaves_IC(E,d)}, one shows that the $*$-restriction $\Eis^0(\cK)\mid_{Shatz^d_{\tilde G}}$ vanishes unless $d\in n\ZZ$. For $d\in n\ZZ$ we get
\begin{equation}
\label{complex_for_Lm17}
\Eis^0(\cK)\mid_{Shatz^d_{\tilde G}}\,\iso\, \tilde q^*\cK\otimes \RG(X^{(\frac{d}{n})}, \Qlb)[-2d-2+2\frac{d}{n}]
\end{equation}
The codimension  of $Shatz^d$ in $\Bun_G$ is $2d-1$. It follows that $\Eis^0(\cK)\mid_{Shatz^d_{\tilde G}}$ is placed in perverse degrees $\le 0$, and its $0$-th perverse cohomology sheaf is isomorphic to $\IC_d$.

 If $d=n$ then (\ref{complex_for_Lm17}) writes as $?\oplus Z[1]\oplus Z[-1]$, where $Z=\IC_0\mid_{Shatz^n_{\tilde G}}$, and $?$ is self-dual. It follows that $\IC_0\mid_{Shatz^n_{\tilde G}}\,\iso\, \IC_n[1]\mid_{Shatz^n_{\tilde G}}$ and $?=0$. 
 
 For $d=2n$ the complex (\ref{complex_for_Lm17}) writes as $?\oplus Z[1]\oplus Z[-1]$, where $Z=\IC_0\mid_{Shatz^d_{\tilde G}}$, and $?$ is self-dual. It follows that $?=\IC_{2n}$ and $\IC_0\mid_{Shatz^{2n}_{\tilde G}}\,\iso\, \IC_{2n}[3]$. 
 
 For $d=3n$ the complex (\ref{complex_for_Lm17}) writes as 
$$
?\oplus Z[1]\oplus Z[-1]\oplus \IC_{2n}\mid_{Shatz^d_{\tilde G}},
$$
where $?$ is self-dual, and $Z=\IC_0\mid_{Shatz^d_{\tilde G}}$. By Lemma~\ref{Lm_restricting_IC_d_case_d>0}, $\IC_{2n}\mid_{Shatz^d_{\tilde G}}\,\iso\, \IC_{3n}[2]$ over $Shatz^d_{\tilde G}$. It follows that $?=\IC_d$ and $\IC_0\mid_{Shatz^d_{\tilde G}}\,\iso\,\IC_{3n}[5]$. Continuing this, our claim easily follows by induction.
\end{Prf}

\medskip

\begin{Rem} If $n=2$ then $\IC_0=\Aut_g$ and $\IC_1=\Aut_s$ in the notations of \cite{L4} and (\cite{L2}, Appendix~A). These are the direct summands of the theta-sheaf $\Aut$. Our description of the fibres of $\IC_0$ and $\IC_1$ in Lemmas~\ref{Lm_restricting_IC_d_case_d>0} and \ref{Lm_great_fibres_of_IC_0} extends (\cite{L4}, Theorem~1).
\end{Rem}

\medskip

Recall the Hecke functors for $G$ defined in Section~\ref{section_Hecke_functors_for_tilde_G}. For $\nu\in\Lambda^{\sharp,+}$ we set $_x\H^{\nu}_G=(\id\times i_x)^*\H^{\nu}_G[-1]$, where $i_x: \Spec k\to X$ is the point $x$, and $id\times i_x: \Bunt_G\to \Bunt_G\times X$.

\begin{Lm} 
\label{Lm_6.5}
Let $\nu=e\alpha$, $d>0$ with $d\in e\ZZ$. \\
(1) Assume $n$ even. For $d\ge n$ one has $_x\H^{\nu}_G\IC_d\,\iso\, \IC_{d+e}\oplus \IC_{d-e}$. Besides, 
$$
_x\H^{\nu}_G\IC_e\,\iso\, \IC_0[1]\oplus \IC_0[-1]
$$
(2) Assume $n$ odd. For $d\ge 2n$ one has
$_x\H^{\nu}_G \IC_d\,\iso\, \IC_{d+n}\oplus\IC_d\oplus \IC_{d-n}$.
Besides, 
$$
_x\H^{\nu}_G \IC_n\,\iso\, \IC_0[1]\oplus\IC_0[-1]\oplus \IC_{2n}
$$
\end{Lm}
\begin{Prf} 
By Proposition~\ref{Pp_special_sheaves_IC(E,d)} and Lemma~\ref{Lm_great_fibres_of_IC_0}, $\Eis^d(\cK)\,\iso\, \Eis^{d+n}(\cK)\oplus \IC_d$ and $$
\Eis^0(\cK)\,\iso\, \IC_0[1]\oplus \IC_0[-1]\oplus \Eis^{2n}(\cK)
$$

Write $\Bunt_T^d$ for the component of $\Bunt_T$ classifying $(\cE, \cU)$ with $\deg \cE=d$. Applying Theorem~\ref{Th_1} to the complex $\cK\mid_{\Bunt^d_T}$ one gets the following.\\
1) Assume $n$ even. One has
$$
_x\H^{\nu}_G \Eis^d(\cK)\,\iso\, \Eis^{d+e}(\cK)\oplus \Eis^{d-e}(\cK)
$$
and $_x\H^{\nu}_G\Eis^{d+n}(\cK)\,\iso\, \Eis^{d+n+e}(\cK)\oplus \Eis^{d+n-e}(\cK)$. For $d\ge n$ this implies
$$
_x\H^{\nu}_G\IC_d\,\iso\, \IC_{d+e}\oplus \IC_{d-e}
$$
Besides, $_x\H^{\nu}_G\IC_e\,\iso\, \IC_0[1]\oplus \IC_0[-1]$. 

\smallskip\noindent
2) Assume $n$ odd. One has
$$
_x\H^{\nu}_G \Eis^d(\cK)\,\iso\, \Eis^{d+n}(\cK)\oplus \Eis^d(\cK)\oplus \Eis^{d-n}(\cK)
$$
and $_x\H^{\nu}_G \Eis^{d+n}(\cK)\,\iso\, \Eis^{d+2n}(\cK)\oplus \Eis^{d+n}(\cK)\oplus \Eis^d(\cK)$. For $d\ge 2n$ this implies
$$
_x\H^{\nu}_G \IC_d\,\iso\, \IC_{d+n}\oplus\IC_d\oplus \IC_{d-n}
$$
Besides, $_x\H^{\nu}_G \IC_n\,\iso\, \IC_0[1]\oplus\IC_0[-1]\oplus \IC_{2n}$.  
\end{Prf} 

\medskip

\begin{Pp} 
\label{Pp_description_of_cA^nu_for_nu_simplest}
Let $\nu=e\alpha$. \\
(1) If $n$ is even $\cA^{\nu}_{\cE}$ is the extension by zero from $\wt\Gr^{\nu}_G$. \\
(2) If $n$ is odd then for $\mu\in\Lambda^+$, $\mu<\nu$ the $*$-restriction $\cA^{\nu}_{\cE}\mid_{\wt\Gr_G^{\mu}}$ vanishes unless $\mu=0$, and $\cA^{\nu}_{\cE}\mid_{\Gr_G^0}\,\iso\, \Qlb[2]$. 
\end{Pp}
\begin{Prf}
In the proof we use some notations from \cite{FL}. Let $\cO=k[[t]]$. Recall that $\AA^e\,\iso\,\Gr_B^0\cap \ov{\Gr}_G^{\nu}$ via the map sending $f=b_et^{-e}+b_{e-1}t^{1-e}+\ldots+b_1t^{-1}$ with $b_i\in \A^1$ to
$$
\left(
\begin{array}{cc}
1 & f\\
0 & 1
\end{array}
\right)G(\cO)
$$
The open subscheme $\Gr_B^0\cap \Gr_G^{\nu}\subset \Gr_B^0\cap \ov{\Gr}_G^{\nu}$ is given by $b_e\ne 0$. In the notations of (\cite{FL}, Lemma~4.2) one has $F^0_T(\cA^{\nu}_{\cE})=0$ for $n$ even, and $F^0_T(\cA^{\nu}_{\cE})\,\iso\, \Qlb$ for $n$ odd. 

 For $\mu\in\Lambda^+$, $\mu<\nu$ the $*$-restriction $\cA^{\nu}_{\cE}\mid_{\wt\Gr_G^{\mu}}$ vanishes unless $\mu=0$. Indeed, if $\mu\in\Lambda^+$, $\mu<\nu$ and $\mu\in\Lambda^{\sharp}$ then $\mu=0$. One has the diagram 
$$
\wt\Gr_T^0 \;\getsup{\gt_B}\; \wt\Gr_B^0\;\toup{\gs_B} \;\wt\Gr_G
$$ 
In the notations of (\cite{FL}, Lemma~4.2) we see that $a_{B,0}^*\cA^{\nu}_{\cE}$ is not constant (resp., is constant) over $\Gr_B^0\cap \Gr_G^{\nu}$ for $n$ even (resp., for $n$ odd). 

 Calculate $(\gt_B)_!\gs_B^*\cA^{\nu}_{\cE}$ using the stratification of $\Gr^0_B\cap \ov{\Gr}_G^{\nu}$ by the locally closed subschemes $\Gr^0_B\cap \Gr_G^{\mu}$ with $\mu\in\Lambda^+$, $\mu\le\nu$. For $n$ even the contribution of $\wt\Gr_B^0\cap \wt\Gr_G^{\nu}$ to this direct image vanishes, so the contribution of $\wt\Gr^0_B\cap \wt\Gr_G^0$ also vanishes.
 
  For $n$ odd the contribution of $\wt\Gr_B^0\cap \wt\Gr_G^{\nu}$ to this direct image is 
$$
\RG_c(\Gm\times\A^{n-1}, \Qlb[2n])\,\iso\, \Qlb[1]\oplus\Qlb
$$ 
Since $F^0_T(\cA^{\nu}_{\cE})\,\iso\, \Qlb$, and $\cA^{\nu}_{\cE}\mid_{\wt\Gr_G^0}$ is placed in strictly negative degrees, our claim follows from the exact triangle 
$
(\Qlb[1]\oplus \Qlb)\to \Qlb\to \cA^{\nu}_{\cE}\mid_{\Gr_G^0}
$.
\end{Prf}

\medskip

 Write $_x\cH^{\nu}_{\tilde G}$ for the preimage of $_x\cH^{\nu}_G$ in $_x\cH_{\tilde G}$.
 
\begin{Lm} Let $\nu=e\alpha$. For a point $(M, M',\beta, x)\in \cH^{\nu}_G$ one has canonically
$$
\frac{\det\RG(X, M')}{\det\RG(X, M)}\,\iso\, ((M((e-1)x)+M')/M((e-1)x))^{\otimes 2e}\otimes \Omega_x^{e(e-1)}
$$
Here $\dim_k M((e-1)x)+M')/M((e-1)x)=1$. There is an isomorphism
$$
\kappa: {_x\cH^{\nu}_{\tilde G}}\,\iso\, (\Bunt_G\times_{\Bun_G}\, {_x\cH^{\nu}_G})\times B(\mu_N),
$$
where we used $h^{\la}_G$ in the fibred product, and the projection to the first term corresponds to $\tilde h^{\la}_G: {_x\cH^{\nu}_{\tilde G}}\to \Bunt_G$. 
\end{Lm}
\begin{Prf}
The symplectic form on $M$ yields a nondegenerate pairing between the $k$-vector spaces $(M+M')/M$ and $(M+M')/M'$. So, 
$$
\det\RG(X, M')\otimes\det\RG(X, M)^{-1}\,\iso\, \det\RG(X, (M+M')/M)^2
$$ 
The first claim follows now from the natural isomorphism
$$
\det\RG(X, (M+M')/M)\,\iso\, ((M((e-1)x)+M')/M((e-1)x))^{\otimes e}\otimes \Omega_x^{e(e-1)/2}
$$ 
A point of $_x\cH^{\nu}_{\tilde G}$ is given by a collection $(M,M',\beta)\in {_x\cH^{\nu}_G}$ and the lines $\cU,\cU'$ together with the isomorphisms $\cU^N\,\iso\ \cL_M$, $\cU'^N\,\iso\, \cL_{M'}$. Define $\kappa$ as the map sending $(M,M',\beta, \cU, \cU')$ to $(M,M',Ê\beta, \cU, \cU_0)$, where
$$
\cU=\cU'\otimes\cU_0\otimes ((M((e-1)x)+M')/M((e-1)x))^{\otimes \frac{2e}{n}}\otimes \cE_{X,x}^{\frac{2e(e-1)}{n}}
$$ 
For $n$ odd the line $\Omega_x^{\frac{e(e-1)}{n}}\,\iso\, \cE_{X,x}^{\frac{2e(e-1)}{n}}$ does not depend on the choice of $\cE_X$. 
\end{Prf}

\medskip

 For $\cM\in \D_{\zeta}(\Bunt_G)$ the contribution of $_x\cH^{\nu}_{\tilde G}$ to $_x\H^{\nu}_G(\cM)$ now writes as
\begin{equation}
\label{contribution_of_cH^nu_G_to_Hecke_functor}
(\tilde h^{\la}_G)_!((\tilde h^{\ra}_G)^*\cM\otimes \kappa^*\cL_{\zeta})[2e]
\end{equation}
for the diagram $\Bunt_G \getsup{\tilde h^{\la}_G} {_x\cH^{\nu}_{\tilde G}}\toup{\tilde h^{\ra}_G} \Bunt_G$. For $n$ odd the contribution of $_x\cH^0_{\tilde G}$ to $_x\H^{\nu}_G(\cM)$ is $\cM[2]$. 

\begin{Th} 
\label{Th_last}
Let $\nu=e\alpha$.\\
(1) For $n$ even one has $_x\H^{\nu}_G\IC_0\,\iso\, \IC_e[1]\oplus \IC_e[-1]$. \\
(2) For $n$ odd one has
$$
_x\H^{\nu}_G\IC_0\,\iso\, \IC_0[2]\oplus \IC_0\oplus \IC_0[-2]
$$
\end{Th}
\begin{Prf}
Recall that $_x\H^{\nu}_G$ is given by a version of (\ref{formula_for_H^nu_G_Section04}) with $x$ fixed. \\
(1) By Lemma~\ref{Lm_great_about_gradings} and Remark~\ref{Rem_parities_of_IC_d_Section2.2}, $_x\H^{\nu}_G\IC_0\in \D_{\zeta, -}(\Bunt_G)$. This implies that the $*$-restriction $(_x\H^{\nu}_G\IC_0)\mid_{Shatz^i_{\tilde G}}$ vanishes unless $i\in e+n\ZZ$. 

Let $i\ge 0, i\in e+n\ZZ$ and $M\in Shatz^i$. Let $Y$ denote the fibre of $h^{\la}_G: {_x\cH^{\nu}_G}\to\Bun_G$ over $M$. Write $\PP(M_x)$ for the projective space of lines in $M(ex)/M((e-1)x)$. We have a map $\eta: Y\to \PP(M_x)$ sending $M'$ to the line $(M'+M((e-1)x))/M((e-1)x)$. Each fibre of $\eta$ identifies with $\AA^{2e-1}$. Denote by $S$ the $*$-fibre of $_x\H^{\nu}_G\IC_0$ at $M$.

 For $d\ge 0$ with $d\in n\ZZ$ let $Y_d\subset Y$ be the locally closed subscheme given by $M'\in Shatz^d$. Since $i>0$, $M$ has a canonical $B$-structure given by $(\cE\subset M)$ with $\deg \cE=i, \cE\in\Bun_1$. 
 
  If $i=e$ then $Y_0\,\iso\, \AA^n$, and the contribution of this locus to $S$ is $\Qlb[-3-n]$. Besides, $Y_n\,\iso\, \Gr_B^{-e\alpha}\cap \Gr_G^{e\alpha}\,\iso\, \Spec k$. So, the contribution of $Y_n$ to $S$ is $\Qlb[-1-n]$.  We see that the $*$-restriction of $_x\H^{\nu}_G\IC_0$ to $Shatz^e$ identifies with $\IC_e[1]\oplus \IC_e[-1]$.

 Assume $i=e+bn$ with $b>0$, $d\ge 0$, $d\in n\ZZ$. Then $Y_d\subset Y$ is not empty only for $d=nb$ or $d=n(b+1)$. One has $Y_{nb}\,\iso\, \Gr_B^{e\alpha}\cap \Gr_G^{e\alpha}\,\iso\, \AA^n$. The contribution of this locus to $S$ is 
$$
\Qlb[-3-2bn+2b-n]\,\iso\, \IC_e[-1]\mid_{M}
$$ 
Further, $Y_{n(b+1)}\,\iso\, \Gr_B^{-e\alpha}\cap \Gr_G^{e\alpha}\,\iso\, \Spec k$. The contribution of $Y_{n(b+1)}$ to $S$ is  
$$
\Qlb[-1-2bn+2b-n]\,\iso\, \IC_e[1]\mid_M
$$
Part (1) follows.\\
(2) The case $n=1$ is well-known, assume $n>1$. We denote by $(M,M',\beta)$ a point of $\ov{\cH}^{\nu}_G$. Let $i\ge 0, i\in n\ZZ$ and $M\in Shatz^i$. Let $\bar Y$ denote the fibre of $h^{\la}_G: {_x\ov{\cH}^{\nu}_G}\to\Bun_G$ over $M$. Let $Y\subset \bar Y$ be the preimage of $_x\cH^{\nu}_G$ in $\bar Y$. 

 For $d\ge 0$, $d\in n\ZZ$ let $Y_d\subset Y$ (resp., $\bar Y_d\subset \bar Y$) denote the locally closed subscheme given by $M'\in Shatz^d$. Write $S$ for the fibre of $_x\H^{\nu}_G\IC_0$ at $M$. As above, we have a map $\eta: Y\to \PP(M_x)$, each fibre identifies with $\AA^{2n-1}$.
 
 First, assume $i=0$. Then only $\bar Y_n=Y_n$ and $\bar Y_0$ contribute to $S$. There is a section $\PP(M_x)\to Y$ of $\eta$, whose image identifies with $Y_n$. Any local system on $\PP^1$ is constant. So, the contribution of $Y_n$ to $S$ is $\RG(\PP^1, \Qlb[-1])\,\iso\, \Qlb[-1]\oplus\Qlb[-3]$. 
 
 The scheme $\bar Y_0$ can be written as the subscheme of $\Gr_G$ of points of the form $AG(\cO)$, where 
$$
A=\left(
\begin{array}{cc}
1+a_1 & a_2\\
a_3 & 1+a_4
\end{array}
\right) \in G(k[t^{-1}])
$$
with $a_i\in k[t^{-1}]$ of degree $\le n$ in $t^{-1}$. In particular, $G$ acts on $\bar Y_0$. This action commutes with the loop rotations $\Gm\subset \Aut(\cO)$ action on $\bar Y_0$. The scheme $\bar Y_0$ can equally be seen as the scheme classifying matrices $A$ as above with $a_i$ of the form $b_{n,i}t^{-n}+\ldots+b_{1,i}t^{-1}\in k[t^{-1}]$ for all $i$.  In the latter form the action of $G$ is given by the conjugation.

 Recall the formulas (\ref{contribution_of_cH^nu_G_to_Hecke_functor}) and (\ref{formula_for_H^nu_G_Section04}). The group $\Gm\subset \Aut(\cO)$ of loop rotations acts on $\bar Y_0$ and contracts it to the point $M\in \bar Y_0$. The complex $(\tilde h^{\ra}_G)^*\IC_0\otimes \IC^{\nu}$ is monodromic with respect to this action. Let $i_0: \Spec k\to \bar Y_0$ denote the point $M$. By (\cite{ICDC}, Lemma~5.3) we get 
\begin{equation}
\label{integral_over_Y_0_vanishes}
\RG_c(\bar Y_0, \IC^{\nu}\otimes(\tilde h^{\ra}_G)^*\IC_0)\,\iso\, i_0^!(\IC^{\nu}\otimes (\tilde h^{\ra}_G)^*\IC_0)\,\iso\, \Qlb[-5]
\end{equation} 
It follows that over $Shatz^0$ one has
$$
_x\H^{\nu}_G\IC_0\,\iso\, \IC_0[2]\oplus \IC_0\oplus \IC_0[-2]
$$ 

 Let now $i>0$, $i\in n\ZZ$. Then $M$ has a distinguished $B$-structure given by the unique subbundle of degree $i$. Then only $\bar Y_{n+i}$, $\bar Y_i$, $\bar Y_{i-n}$ may contribute to $S$. Note that $\bar Y_{n+i}=Y_{n+i}$ is the point scheme. Its contribution to $S$ is $\Qlb[-1-2i+\frac{2i}{n}]$. One has $Y_{i-n}=\bar Y_{i-n}\,\iso\, \Gr_B^{\nu}\cap \Gr_G^{\nu}\,\iso\, \AA^{2n}$. So, the contribution of $\bar Y_{i-n}$ to $S$ is $\Qlb[-5-2i+\frac{2i}{n}]$. 
 
  Finally, $\bar Y_i\,\iso\, \Gr_B^0\cap \ov{\Gr}_G^{\nu}\,\iso\,\AA^n$. To calculate the contribution of $\bar Y_i$ to $S$, argue as in Proposition~\ref{Pp_description_of_cA^nu_for_nu_simplest}. In the notations of (\cite{FL}, Lemma~4.2), the contribution of $\bar Y_i$ to $S$ identifies with 
$$
F^0_T(\cA^{\nu}_{\cE})[-3-2i+\frac{2i}{n}]\,\iso\, \Qlb[-3-2i+\frac{2i}{n}]\,\iso\, (\IC_0)_M
$$ 
Using Lemma~\ref{Lm_great_fibres_of_IC_0} we see that $S$ is isomorphic to the $*$-restriction of $\IC_0[2]\oplus \IC_0\oplus \IC_0[-2]$ to $M$. We are done.
\end{Prf}

\appendix

\section{Proof of Theorem~\ref{Th_trace_of_Frob_Eis'}}
\label{Section_appendixA}

\subsection{} Assume $k=\Fq$. A version of Corollary 5.1 holds also in this case. It is understood that the Tate twists are recovered in the corresponding formulas. Take $P=B$, so $M=T$. Let $\theta\in\Lambda^{pos}$. In this case the stack $\Mod^{+,\theta}_{\tilde M}$ defined in Section~\ref{Section_Zastava spaces} classifies $\bar D\in X^{\theta}$ and a line $\cU$ together with $\cU^N\,\iso\, \cL_{\cF^0_T(-\bar D)}$. 
 
 The stack $\Modt^{+,\theta}_{\Bunt_T}$ from Section~\ref{Section_521} classifies $(\bar D\in X^{\theta}, \cF'_T\in\Bun_T, \cU,\cU')$, where $\cU,\cU'$ are lines equipped with $\cU^N\,\iso\,\cL_{\cF'_T(-D)}$, $\cU'^N\,\iso\, \cL_{\cF'_T}$. Recall that $J$ is the set of positive roots of $\check{G}_n$ for $\check{B}_n$.

 For $\gB(\theta)=\sum_{\nu\in J} n_{\nu}\nu \in \Lambda^{pos,pos}_{G,B}$ the perverse sheaf $\IC^{\gB(\theta)}_{\Bun_T, \zeta}$ on $X^{\gB(\theta)}\times_{X^{\theta}}\Modt^{+,\theta}_{\Bunt_T}$ is a rank one shifted local system. Consider the diagram 
$$
X^{\gB(\theta)}\times \Bunt_T\getsup{f}
X^{\gB(\theta)}\times_{X^{\theta}}\Modt^{+,\theta}_{\Bunt_T}\toup{h} \Bunt_T
$$ 
where $f$ sends $D\in X^{\gB(\theta)}$ with image $\bar D\in X^{\theta}$, $(\bar D, \cF'_T, \cU,\cU')\in \Modt^{+,\theta}_{\Bunt_T}$ to $(D, \cF'_T, \cU')$, and $h$ sends the above point to $(\cF'_T(-\bar D), \cU)\in\Bunt_T$. 

\smallskip

Recall that $X^{\gB(\theta)}=\mathop{\prod}\limits_{\nu\in J} X^{(n_{\nu})}$ and $\mid\!\gB(\theta)\!\mid$ is the dimension of $X^{\gB(\theta)}$. One has 
$$
h^*\cK_E\otimes \IC^{\gB(\theta)}_{\Bun_T,\zeta}
\,\iso\, f^*((\mathop{\boxtimes}\limits_{\nu\in J} (E^{-\nu})^{(n_{\nu})})\boxtimes \cK_E)\otimes (\Qlb[1](\frac{1}{2}))^{\dim\Bun_T+\mid\gB(\theta)\mid}
$$
So, by Corollary~\ref{Cor_3}, the contribution of the stratum $_{\theta}\Bunt_{\tilde B}$ to $\Funct(\Eis(\cK^{\mu}_E))$ is 
$$
\Funct(\Eis'(\cK_E^{\mu-\theta})) \prod_{\nu\in J}\Tr(\Fr, \RG(X^{(n_{\nu})}, (E^{-\nu})^{(n_{\nu})})\otimes \Qlb(n_{\nu}))
$$
Theorem~\ref{Th_trace_of_Frob_Eis'} is proved.

\medskip\noindent
\select{Acknowledgements.} The author is grateful to V. Lafforgue and M. Finkelberg for many fruitful discussions that helped us to overcome some difficulties on the way. We thank D.~Gaitsgory for answering our questions. It is a pleasure to thank the organizers of the conference `Automorphic forms and harmonic analysis on covering groups' (The American Institute of Mathematics, Palo Alto, 2013), which stimulated our work. The author was supported by the ANR project ANR-13-BS01-0001-01.

\end{document}